\documentclass[12pt]{article}
\usepackage{amsmath,amssymb,amsfonts,amsthm,amscd,graphicx,psfrag,epsfig}
\newtheorem{theorem}{Theorem}[section]

\newtheorem{theorem-definition}[theorem]{Theorem-Definition}
\newtheorem{theorem-construction}[theorem]{Theorem-Construction}
\newtheorem{lemma-definition}[theorem]{Lemma--Definition}
\newtheorem{lemma-construction}[theorem]{Lemma--Construction}
\newtheorem{lemma}[theorem]{Lemma}
\newtheorem{proposition}[theorem]{Proposition}
\newtheorem{corollary}[theorem]{Corollary}
\newtheorem{conjecture}[theorem]{Conjecture}
\newtheorem{definition}[theorem]{Definition}

\newtheorem{problem}{Problem}
\newenvironment{remark}[1][Remark.]{\begin{trivlist}
\item[\hskip \labelsep {\bfseries #1}]}{\end{trivlist}}

\newcommand{\old}[1]{}

\newcommand{\Z}{{\mathbb Z}}
\newcommand{\R}{{\mathbb R}}

\newcommand{\C}{{\mathbb C}}
\newcommand{\T}{{\mathbb T}}
\newcommand{\V}{{\mathbb V}}
\renewcommand{\P}{{\mathbb P}}
\newcommand{\PGL}{\mathrm{PGL}}
\newcommand{\veps}{\varepsilon}

\newcommand{\lms}{\longmapsto}
\newcommand{\lra}{\longrightarrow}
\newcommand{\hra}{\hookrightarrow}

\newcommand{\be}{\begin{equation}}
\newcommand{\ee}{\end{equation}}
\newcommand{\bt}{\begin{theorem}}

\newcommand{\et}{\end{theorem}}
\newcommand{\bd}{\begin{definition}}
\newcommand{\ed}{\end{definition}}
\newcommand{\bp}{\begin{proposition}}
\newcommand{\ep}{\end{proposition}}

\newcommand{\bl}{\begin{lemma}}
\newcommand{\el}{\end{lemma}}
\newcommand{\bc}{\begin{corollary}}
\newcommand{\ec}{\end{corollary}}
\newcommand{\bcon}{\begin{conjecture}}
\newcommand{\econ}{\end{conjecture}}
\newcommand{\la}{\label}
\renewcommand{\L}{{\mathcal L}}

\begin{document}
\date{}
\title{Dimers and cluster integrable systems}  
\author{A.B. Goncharov, R. Kenyon}

\maketitle

\begin{abstract} We show that the dimer model on a bipartite graph $\Gamma$ on a torus
gives rise to a quantum integrable system of special type, which we call 
a {\it cluster integrable system}. The phase space of the classical
system contains, as an open dense subset, the moduli space 
${\L}_\Gamma$ of line bundles with connections on the graph $\Gamma$. 
The sum of Hamiltonians  is  
essentially the partition function of the dimer model. 

We say that two such graphs $\Gamma_1$ and $\Gamma_2$ are {\it equivalent} 
if the Newton polygons of the corresponding partition functions   
coincide up to translation. We define elementary transformations 
of bipartite surface graphs, and show that two equivalent minimal 
bipartite graphs are related by a sequence of elementary transformations. 
For each elementary transformation 
we define a birational Poisson isomorphism ${\L}_{\Gamma_1} \to {\L}_{\Gamma_2}$  
providing an equivalence of the integrable systems.  
We show that it is a cluster Poisson transformation, as defined in \cite{FG1}.

We show that for any convex integral polygon $N$ 
there is a non-empty finite set of minimal graphs $\Gamma$ for which 
$N$ is the Newton polygon of the partition function related to $\Gamma$. 
Gluing the varieties ${\L}_\Gamma$ for graphs 
$\Gamma$ related by elementary transformations 
via the corresponding 
cluster Poisson transformations, we get a Poisson space ${\mathcal X}_N$.  
It  is  a 
natural phase space for the integrable system. The Hamiltonians are 
functions on ${\mathcal X}_N$, parametrized by the 
interior points of the Newton polygon $N$. 
We construct Casimir functions whose level sets are the symplectic leaves of 
${\mathcal X}_N$.

The space ${\mathcal X}_N$
 has a structure of a cluster Poisson variety.  
Therefore the algebra of regular functions on ${\mathcal X}_N$ 
has a non-commutative $q$-deformation to a $\ast$-algebra ${\mathcal O}_q({\mathcal X}_N)$. We show that the Hamiltonians 
give rise to a commuting family of quantum Hamiltonians. Together with the quantum Casimirs, they  
provide a quantum integrable
system. Applying the general quantization scheme 
\cite{FG2}, we get a 
$\ast$-representation of the $\ast$-algebra ${\mathcal O}_q({\mathcal X}_N)$ 
in a Hilbert space. The quantum Hamiltonians act by 
commuting unbounded selfadjoint operators. 

For square grid bipartite graphs 
on a torus we get {\it discrete quantum integrable systems}, where 
the evolution is a cluster automorphism of the $\ast$-algebra 
${\mathcal O}_q({\mathcal X}_N)$ 
 commuting with the quantum Hamiltonians. 
We show that the {\it octahedral recurrence}, closely related to 
{\it Hirota's bilinear difference 
equation}
 \cite{Hir}, appears this way. 

Any graph $G$ on a torus ${\mathbb T}$ gives rise to a bipartite graph 
$\Gamma_G$ on ${\mathbb T}$. 
We show that 
the phase space ${\mathcal X}$ related to  the graph $\Gamma_G$ has 
a  Lagrangian subvariety ${\mathcal R}$, defined in each 
 coordinate system by a
 system of monomial equations. 
 We
identify it with the space parametrizing resistor networks on $G$. 
The pair $({\mathcal X}, {\mathcal R})$ has a large group of cluster automorphisms. 
In particular, for a hexagonal grid graph 
 we get a {discrete quantum integrable system} on ${\mathcal X}$ 
whose restriction 
to ${\mathcal R}$ is essentially given by 
the {\it cube recurrence} 
 \cite{Miwa}, \cite{CS}. 

The set of positive real points ${\mathcal X}_N(\R_{>0})$ of the phase space  
is well defined. It is isomorphic to the 
moduli space of simple Harnack curves with divisors studied in \cite{KO}. The 
Liouville tori of the real
integrable system are given by the product of ovals of the simple Harnack curves. 

In the sequel  \cite{GK} to
this paper we show that the set of complex points ${\mathcal X}_N(\C)$ of the phase space is birationally isomorphic
to a finite cover of
 the Beauville  complex algebraic integrable system related to the toric surface 
assigned to the polygon $N$.

\end{abstract}

\tableofcontents

\section{Introduction}

A \emph{line bundle} $V$ on a graph $\Gamma$ is given by assigning a $1$-dimensional
complex vector space $V_v$ to each vertex of $\Gamma$. 
A \emph{connection} on $V$ is a choice of isomorphism $\phi_{vv'}:V_v\to V_{v'}$
whenever $v,v'$ are adjacent, satisfying $\phi_{v'v}=\phi_{vv'}^{-1}$.

We construct a Poisson structure on the space of line bundles with
connections on a bipartite graph embedded on a surface.
When the graph is embedded on a torus $\T$ we construct
an algebraic integrable system, with commuting Hamiltonians which can be
written as Laurent polynomials in natural coordinates on $\L_\Gamma$, the moduli space of line bundles with connections on $\Gamma$.
The Hamiltonians are sums of dimer covers of $\Gamma$.

A \emph{dimer cover} of a graph $\Gamma$ is a set of edges with the property
that every vertex is the endpoint of a unique edge in the cover.
Probability measures on dimer covers have been the subject of much recent work
in statistical mechanics.
In this paper we study some non-probabilistic applications of the dimer model,
in particular to the construction of cluster integrable systems. 

Below we introduce the key non-technical definitions which we use 
throughout the paper, and discuss the main results. 

\subsection{Dimer models and Poisson geometry.} 
A {\it surface graph $\Gamma$} is a graph embedded
on a compact oriented surface $S$ whose {\it faces}, 
i.e. the connected components of $S-\Gamma$,  are contractible. 
It is homotopy equivalent to an open surface $S_0\subset S$, obtained by putting a puncture in every face. 
A {\it bipartite} graph is a graph with vertices of two types,  
black and white, such that each edge 
has one black and one white vertex. 

\subsubsection{Conjugate surface graph and Poisson structure}\la{sec1.1.2}
Let ${\L}_\Gamma$ be the moduli space of line bundles with 
connections on  a graph $\Gamma$. 
Any oriented loop $L$ on $\Gamma$ gives rise to a function 
$W_L$ on ${\L}_\Gamma$ given
 by the monodromy of a line bundle with connection 
along $L$. One has $W_{-L}=W_{L}^{-1}$, 
where $-L$ is the loop $L$ with the opposite orientation. 
The monodromies around the loops provide an isomorphism
\be \la{3.5.09.1}
{\L}_\Gamma \cong {\rm Hom}(H_1(\Gamma, \Z), \C^*) = H^1(\Gamma, \C^*). 
\ee

For any surface graph $\Gamma$, the moduli space ${\L}_\Gamma$
has the  {\it standard Poisson structure}. Namely, since $\Gamma$ is homotopy equivalent to the surface 
$S_0$, we have   ${\L}_\Gamma = H^1(S_0, \C^*)$. 
This latter space has a Poisson structure provided by 
the intersection pairing  on $H_1(S_0, \Z)$. 

Now let $\Gamma$ be a bipartite surface graph. 
Then we introduce a \emph{new} Poisson structure on ${\L}_\Gamma$
related to the standard one by a ``twist". 

A surface graph 
is the same thing as a {\it ribbon graph}: a ribbon graph is a graph with an additional structure given by, for each vertex, a cyclic order 
of the edges at that vertex. Given a ribbon graph, we can replace 
its edges by ribbons, getting an oriented surface.

A bipartite surface graph $\Gamma$ gives rise to a new ribbon graph $\widehat \Gamma_w = \widehat \Gamma$, 
the {\it conjugate graph}, obtained by reversing the cyclic orders at 
all white vertices\footnote{Reversing the cyclic order at the black vertices we get a ribbon graph $\widehat \Gamma_b$, 
which differs from 
$\widehat \Gamma_w$ by the orientation.}. 
The topological surface with boundary 
corresponding to the ribbon graph $\widehat \Gamma_w$ is called the 
{\it conjugated surface $\widehat S_w$.} 

Evidently a line bundle with connection on the 
graph $\widehat \Gamma$ is  the same thing as a line bundle with connection on the original graph $\Gamma$. 
So  there is a canonical isomorphism 
\be \la{ipssbl} 
{\L}_{\Gamma} = {\L}_{\widehat \Gamma}.
\ee
 The standard Poisson structure on the 
moduli space ${\L}_{\widehat \Gamma}$  
combined with the isomorphism (\ref{ipssbl})  
provides the Poisson structure on ${\L}_{\Gamma}$ we need. 
Precisely, consider  the intersection pairing on $\widehat S_w$:
\be \la{INTP}
\varepsilon_{\widehat S_w}: H_1(\widehat S_w, \Z) \wedge H_1(\widehat S_w, \Z)  \lra \Z.
\ee
Given two loops $L_1, L_2$ on $\Gamma$, 
we define the Poisson bracket $\{W_{L_1}, W_{L_2}\}$ by setting
\be \la{3/1/09.1}
\{W_{L_1}, W_{L_2}\} = \varepsilon_{\widehat S_w}(L_1, L_2)W_{L_1} W_{L_2}.
\ee
and extending it via the Leibniz rule to a 
Poisson bracket on the algebra generated by the functions 
$W_L$, which coincides with the algebra of regular functions on the variety ${\L}_\Gamma$.

\subsubsection{Natural coordinates on ${\L}_{\Gamma}$.}
The orientation of the surface $S$ 
induces orientations of the  faces  of $\Gamma$. 
The group $H_1(\Gamma, \Z)$ is  
generated by the oriented boundaries $\partial(F)$ of the 
faces $F$ of $\Gamma$,  whose orientation is induced by the 
orientation of $F$, and loops generating $H_1(S,\Z)$. 
The only relation is that the sum of boundaries of all faces is zero. 
There is an exact sequence, where 
 the map $\pi$ is induced by the embedding 
$\Gamma \hra S$:  
$$
0\lra{\rm Ker}~\pi \lra H_1(\Gamma;\Z) \stackrel{\pi}{\lra} H_1(S;\Z)\lra 0.
$$
The subgroup ${\rm Ker}~\pi$ is generated by the oriented boundaries $\partial(F)$ of 
faces $F$.

Summarising,  
the monodromies of connections around all but one 
oriented faces $F$ of $\Gamma$, augmented by the monodromies 
around loops generating $H_1(S, \Z)$, provide a coordinate system $\{X_i\}$ 
on the  moduli space ${\L}_\Gamma$. The Poisson structure in these coordinates 
has a standard quadratic form
\be \la{PT1}
\{X_i, X_j\} = \varepsilon_{ij}X_i X_j, \qquad \varepsilon_{ij}\in \Z,
\ee
where $\veps$ is the intersection pairing on $\hat S_w$.

\subsubsection{
Zig-zag loops and the center of the Poisson algebra of functions on  ${\mathcal L}_\Gamma$.} 
A {\it zig-zag path} (see \cite{Kenyon.dirac, Postnikov}) on a ribbon graph graph  
$\Gamma$ is an oriented path on   
$\Gamma$ which turns maximally left at white vertices and 
turns maximally left at black vertices; see Figure \ref{di15}. Zig-zag paths necessarily close up to form zig-zag loops.
There is a bijection 
\be\la{ZZBH}
\{\mbox{zig-zag loops on $\Gamma$}\} \leftrightarrow \{\mbox{boundaries of the holes on
 $\widehat S_w$}\}.
\ee
The orientations of zig-zag loops match  
the orientations of the boundary loops on $\widehat S_w$ induced by 
the orientation of  $\widehat S_w$. 

\begin{figure}[ht]
\centerline{\epsfbox{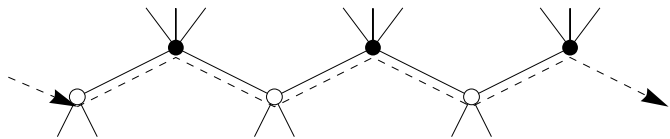}}
\caption{A Zig-zag path on a ribbon graph.\label{di15} }
\end{figure}

 \bl \la{1} Let $Z$ be an (oriented) 
zig-zag loop on a bipartite oriented surface graph $\Gamma$. 
Then as $Z$ runs over zig-zag loops, the functions 
$W_Z$ generate the center of the Poisson algebra ${\mathcal O}({\L}_\Gamma)$ 
of functions on ${\L}_\Gamma$. 
The product of all of them is $1$. This is the only relation between them. 
\el

We call the functions $W_Z$ the {\it Casimirs}. 

\begin{proof} It follows from the  canonical bijection (\ref{ZZBH}). 
Evidently the boundary loops on $\widehat S_w$ 
generate the kernel of the intersection pairing 
on $H_1(\widehat S_w, \Z)$. Clearly the sum 
of them is zero, and this is the only relation between them. 
\end{proof}

There is a similar Poisson moduli space ${\L}_\Gamma$ 
for a bipartite graph $\Gamma$ on an oriented surface $S$ with boundary.

\subsubsection{Gluing the Poisson phase space}
The set of homology classes of dimer coverings of $\Gamma$ (as defined below)
defines a convex
polyhedron $N\subset H_1(S, \R)$. It is the Newton polyhedron of the 
{\it dimer partition function}.
Two bipartite surface graphs  are {\it equivalent} 
if their Newton polyhedra coincide up to translation. In each equivalence class there
are \emph{minimal} bipartite surface graphs, which have
the minimal number of faces. 
We prove that every two minimal bipartite surface graphs are related by 
certain {\it elementary transformations}, consisting of  
``spider''  moves 
and ``shrinking of a $2$-valent vertex'' moves.  
Each elementary transformation $\Gamma_1 \to \Gamma_2$ gives rise 
to a cluster Poisson birational isomorphism 
\be \la{gm}
i_{\Gamma_1, \Gamma_2}: {\mathcal L}_{\Gamma_1} \stackrel{\sim}{\lra} {\mathcal L}_{\Gamma_2} 
\ee
We define a  Poisson variety ${\mathcal X}_N$, the {\it phase space}, 
by gluing the moduli spaces ${\L}_\Gamma$ 
via transformations (\ref{gm}). 
The variety ${\mathcal X}_N$ depends only on $N$.

Since the sets of zig-zag loops on equivalent 
bipartite surface graphs 
are naturally identified, the gluing maps (\ref{gm}) 
preserve the Casimirs. Therefore  
the center of the Poisson algebra 
of regular functions on the phase space ${\mathcal X}_N$ is also 
described by the Casimirs.

\subsection{Cluster integrable systems for dimer models on a torus.} 

\subsubsection{Classical integrable system}
Let us restrict to the case when $\Gamma$ is a bipartite graph on a torus ${\T}$. 
In this case the Newton polyhedron $N$  
is a convex polygon in the plane $H_1({\T}, \R)$ with vertices at the 
lattice $H_1({\T}, \Z)$. 
Zig-zag loops are parametrized 
by the intervals 
on the boundary $\partial N$ between two consecutive 
integral points, called {\it primitive boundary intervals}.

We show that the partition function of the dimer model on $\Gamma$ 
gives rise to a regular function ${\mathcal P}$ on  
${\mathcal X}_N$, the {\it modified partition function}, 
well defined up to a multiplication by a monomial of Casimirs.  
It has a natural decomposition into a sum of components  
$H_a$, called the {\it Hamiltonians}, parametrized by the 
internal integral points of 
the Newton polygon $N$. These Hamiltonians are Laurent polynomials in 
the $X_i$'s. 
The Hamiltonian flows are well defined. 

\bt \la{4.10.11.1} The Hamiltonian flows of the $H_a$   
commute, providing an integrable system on ${\mathcal X}_N$. 
Precisely, we get integrable systems on the generic symplectic leaves of ${\mathcal X}_N$, 
given by the level sets of the Casimirs. 
\et
In particular, 
$$
{\rm dim}{\mathcal X}_N = 2 i(N) + e(N)-1,
$$
where $i(N)$ (respectively $e(N)$) is the number of integral points inside (respectively 
on the boundary) of the polygon $N$.

\subsubsection{Quantum integrable system} \la{QS}
The algebra of regular functions 
on any cluster Poisson variety  ${\mathcal X}$ has a natural $q$-deformation to 
a $\ast$-algebra ${\mathcal O}_q({\mathcal X})$, defined 
in   \cite[Section 3]{FG1}
It can be realized, in many different ways, 
 as a subalgebra in the {\it quantum torus algebra}. The latter is generated by the elements 
$X_i^{\pm 1}$ satisfying the following relations, where $q$ is a formal variable commuting with $X_i$'s:
\be \la{PT2}
q^{-\varepsilon_{ij}}X_i X_j = q^{-\varepsilon_{ji}}X_j  X_i, \qquad \varepsilon_{ij}\in \Z.
\ee
It has an involutive antiautomorphism $\ast$, acting on the generators by 
$$
\ast(X_i) = X_i, \quad \ast(q) = q^{-1}. 
$$

So in our case we get a $\ast$-algebra ${\mathcal O}_q({\mathcal X}_N)$. 
The $\varepsilon_{ij}$ in (\ref{PT2}) is the same as the Poisson tensor in (\ref{PT1}). 
The quantum analog of Theorem \ref{4.10.11.1} is 

\bt
The Hamiltonians $H_a$ give rise to commuting selfadjoint elements ${\mathbb H}_a$ in the
 $\ast$-algebra ${\mathcal O}_q({\mathcal X}_N)$, which we call {\rm quantum Hamiltonians}. The Casimirs 
give rise to the generators of the center of the algebra. 

In any of the coordinate systems provided by bipartite graphs, the quantum Hamiltonians 
are given by Laurent polynomials with positive integral coefficients.
\et

Let us assume now that $|q|=1$. Then, according to \cite{FG2}, for any cluster Poisson variety ${\mathcal X}$, there  is 
a family of Hilbert spaces ${\mathcal H}_{\chi}$, 
parametrized by the characters $\chi$ of the center of the algebra ${\mathcal O}_q({\mathcal X})$, and 
  a $\ast$-representation of the $\ast$-algebra  ${\mathcal O}_q({\mathcal X})$ 
in (a certain Schwartz space ${\mathcal S}_{\chi}$, dense in) ${\mathcal H}_{\chi}$. 

So in our case we get a family of Hilbert spaces ${\mathcal H}_{N, \chi}$, parametrized 
by a collection of real numbers $\chi$, the eigenvalues of the Casimirs. 
The  quantum Hamiltonians 
are realized by commuting unbounded selfadjoint operators in 
${\mathcal H}_{N, \chi}$.

\begin{problem}
Find the generalized
eigenfunctions of the quantum Hamiltonians, and the spectral decomposition. 
\end{problem}

We call the emerging object a {\it cluster integrable system}. 
A key point is that a cluster integrable system admits a quantum version, as described above. 
We believe that many interesting integrable systems are cluster integrable systems.

\subsubsection{Discrete cluster  integrable systems.} \la{5.13.11.1}
In the case of the dimer model on a square grid on a 
torus\footnote{The square grid dimer models on a torus are parametrized by  
sublattices of the isomorphism group $\Z^2$ of the 
standard square grid bipartite graph in $\R^2$.}  
our quantum integrable system has 
a remarkable automorphism $A$, called the \emph{octahedron recurrence} or,
in other language, \emph{Hirota's bilinear difference equation} (HBDE) \cite{Hir}. 
We show that it commutes with the quantum Casimirs and quantum Hamiltonians, 
and is given by an element of the {\it cluster modular group} \cite{FG1} of the 
cluster variety ${\mathcal X}$ -- see Section \ref{HBDEsection}.

There are similar discrete cluster  integrable systems related to 
hexagonal graphs $G$ on a torus, where  the discrete evolution operator  
is described by the {\it cube recurrence}, see Sections 
\ref{resistorautsection}-\ref{cuberecsection}. 
Its phase space is the phase space of the dimer model related to  the 
resistor network model on $G$, see Section \ref{1.3}. 

We stress that discrete cluster integrable systems 
are discrete quantum integrable systems. Indeed, the ring of regular functions on any 
cluster Poisson variety admits a canonical non-commutative $q$-deformation, and 
the discrete evolution operator is given by an element of the cluster modular group, which 
acts by an automorphism of the $q$-deformed space  \cite{FG1}. 
Finally, the classical Hamiltonains, in all examples we know,  are given by 
functions which in any of the coordinate systems are Laurent polynomials with positive integral coefficients. 
Therefore they admit a natural upgrade 
to their non-commutative counterparts.

 \subsection{Cluster nature of a resistor network model on a torus}\la{1.3}

\subsubsection{Dimer model arising from a  graph on a torus} 
A special case of the dimer model is one arising from a planar resistor network $G$ via
the generalized
Temperley's trick \cite{KPW}. Namely, an arbitrary graph $G$ on a torus gives rise to a bipartite 
graph $\Gamma=\Gamma_G$, obtained by taking a barycentric subdivision of the graph 
$G$ and declaring the centers of the edges of $G$ to be the white vertices
and the 
centers of the faces of $G$ together with the 
vertices of $G$ to be the black vertices. 
The graph $\Gamma$ contains the dual graph $G'$ to the 
graph of $G$, see Figure \ref{di36}, 
where $G$ is shown by a solid lines, and $G'$ by dotted lines.

\begin{figure}[htbp]
\epsfxsize120pt
\centerline{\epsfbox{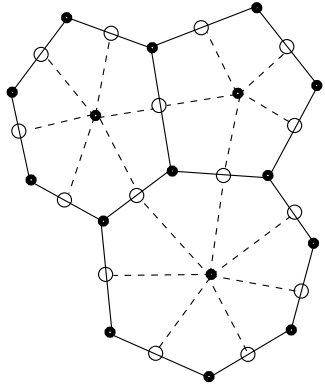}}
\caption{The bipartite graph $\Gamma_G$ (dashed and solid edges)
assigned to a graph $G$ (solid edges).\label{di36}}
\end{figure}

Given a surface graph $G$ embedded on a surface $S$ there is a graph $\widetilde G$ 
embedded on the
universal cover $\widetilde S$, which is the natural lift of $G$. 
A surface graph $G$ embedded on surface $S$ is called \emph{minimal} if
zig-zag paths (defined as unoriented 
paths which turn alternately maximally left and maximally right
at consecutive vertices) on $\widetilde G$ do not have selfcrossings and 
different zig-zag paths do not intersect in more than one edge.  

The Newton polygon $N(G)$ of the bipartite graph $\Gamma_G$ is a centrally-symmetric integral polygon in 
$H_1(\T, \Z)$ centered at the origin. It coincides with the Newton polygon of the spectral curve of the Laplacian 
related to a generic {\it resistor network} on $G$, see Section 
\ref{resistornetsection}.

Any such polygon arises as a Newton polygon $N(G)$ 
for a non-empty finite collection of minimal graphs $G$ on a torus. 
Any two minimal graphs $G_1$ and $G_2$ on a torus with the same Newton polygon are related by 
a sequence of elementary transformations, called Y-$\Delta$ moves, see Figure \ref{di34}. 

A Y-$\Delta$ move $G_1 \to G_2$ induces a transformation $\Gamma_{1} 
\to \Gamma_{2}$ 
of the corresponding bipartite graphs. 
We decompose it into a composition of four spider moves (Lemma \ref{6.1.11.1}), which provides us 
with a cluster birational isomorphism
\be \la{5.19.11.1}
\mu_{Y-\Delta}: ~{\mathcal L}_{\Gamma_{1}} \lra {\mathcal L}_{\Gamma_{2}}. 
\ee
Gluing the tori ${\mathcal L}_{\Gamma}$ assigned to the bipartite graphs $\Gamma$ 
related to the  
equivalent minimal graphs $G$, we get  a cluster Poisson space ${\mathcal X}$. \footnote{
Gluing this space we use only the graphs $\Gamma$ which are equivalent by 
the Y-$\Delta$-moves. We may have other bipartite  graphs
 which are still equivalent as bipartite graphs.}

\subsubsection{Resistor network Lagrangian subvariety}

The moduli space ${\mathcal L}_{\Gamma_G}$ of line bundles with connections on the graph $\Gamma_G$ has 
a subvariety ${\mathcal R}_{G}$ defined by the condition that the monodromies 
over any loop $\alpha$ in $G$ and any loop $\alpha'$ in $G'$ in the same homology class 
on $\T$ are equal. 
These equations are given by equating certain 
monomials in the $X_i$'s to $1$. 

We prove that ${\mathcal R}_{G}$ is a Lagrangian subvariety 
of the Poisson variety ${\mathcal L}_{\Gamma_G}$, i.e. its intersections with the generic 
symplectic leaves 
are Lagrangian subvarieties. We show that the space 
${\mathcal R}_{G}$ parametrizes {resistor networks} on 
$G$, justifying the name {\it resistor network Lagrangian subvariety}.

Equivalently, the conjugated surface 
$\widehat S_{\Gamma_G}$ is canonically realized as 
the boundary of a certain handlebody. 

We show that the Y-$\Delta$-cluster transformation (\ref{5.19.11.1}) provides a birational 
isomorphism of  the 
Lagrangian subvarieties
$$
\mu^R_{Y-\Delta}: ~{\mathcal R}_{G_1}\lra {\mathcal R}_{G_2}.
$$  
Gluing the subvarieties ${\mathcal R}_{G}$ via these birational isomorphisms, we get 
a Lagrangian subspace ${\mathcal R}$ of the corresponding cluster
 Poisson space ${\mathcal X}$. 
It 
gives rise to a uniquely defined up to a constant 
functional on the Schwartz space ${\cal S}_\chi$, see Section \ref{QCMLS}.

%
%
%
\subsection{Toric surfaces, spectral data and the Beauville integrable system}

Let ${\rm T}:= {\rm T}_{\mathbb T}\cong(\C^*)^2$ be the algebraic torus related to the 
torus ${\mathbb T}$.  
Recall that a bipartite graph $\Gamma$ on  ${\mathbb T}$ gives rise to a Newton polygon 
$N$ in $H_1({\mathbb T}, \R)$ with vertices at  $H_1({\mathbb T}, \Z)$. 
Such a  Newton polygon $N$, considered up to translation, provides us 
a toric surface ${\mathcal N}$ together with a line bundle ${\mathcal L}$. 
The algebraic torus ${\rm T}$ acts on the pair $({\mathcal N}, {\mathcal L})$. 
The points of ${\mathcal N}$ where the action is not free form 
the divisor at infinity ${\mathcal N}_\infty$ of ${\mathcal N}$. It is a collection of ${\mathbb P}^1$'s 
which is combinatorially isomorphic to the polygon $N$. 

\subsubsection{The Kasteleyn operator}\la{1.4.1}

A \emph{Kasteleyn line bundle with connection} ${\mathbb K}$ on a bipartite surface graph 
$\Gamma$ 
is a line bundle with connection with monodromy $(-1)^{\ell/2+1}$
around faces having $\ell$ edges, and $\pm1$ around topologically
nontrivial loops. There are $2^{2g}$ non-isomorphic choices of ${\mathbb K}$ 
when $S$ has genus $g$. 
Given a line bundle with connection $V$ on a bipartite surface graph $\Gamma$,
we construct a new one $V\otimes {\mathbb K}$, and set 
\be \la{4.8.11.10}
{\bf V}_{\rm B}:= \bigoplus_{v \in  \{\text{black vertices of $\Gamma$}\}}(V\otimes {\mathbb K})_v, 
\ee
\be \la{4.8.11.11}
{\bf V}_{\rm W} := 
\bigoplus_{v\in  \{\text{white vertices of $\Gamma$}\}}(V\otimes {\mathbb K})_v.
\ee 
A connection on the line bundle $V\otimes {\mathbb K}$ is encoded by 
an operator 
$$
K: {\bf V}_{\rm B}\lra {\bf V}_{\rm W},
$$
the {\it Kasteleyn operator},
whose matrix elements are given by parallel transports 
for the connection along the edges of $\Gamma$. We let 
\be \la{dl}
{\bf L} = {\rm det}{\bf V}^*_{\rm B}\otimes {\rm det}{\bf V}_{\rm W}
\ee
be the determinant line (here $\det V$ denotes the highest 
exterior power of a vector space $V$).
The determinant $\det K$ is an element of the determinant line.

\subsubsection{The spectral data.} \la{1.4.2}
Let 
$$
{\rm T}:= {\rm Hom}(H_1({\mathbb T}, \Z), \C^*) = H_1({\mathbb T}, \C^*) \stackrel{\sim}{=}(\C^*)^2;
$$
The group ${\rm T}$ is identified with the group of complex line bundles with flat 
connections 
on  the torus ${\mathbb T}$. The embedding 
$i$ of the graph $\Gamma$ into the torus ${\mathbb T}$ provides a free action 
of the group ${\rm T}$ on the moduli space ${\mathcal L}_\Gamma$, 
 given by $\ast \lms \ast \otimes i^*L$, where $L$ is a flat line bundle on 
the torus ${\mathbb T}$. This action preserves the Poisson structure on ${\L}_\Gamma$.

The {\it partition function} can be calculated 
as the determinant of the Kasteleyn operator. 
It is a section of the determinant line 
bundle over ${\L}_\Gamma$. Let us restrict it to 
a ${\rm T}$-orbit on ${\L}_\Gamma$. Then it extends to a section of the line bundle 
${\mathcal L}\otimes {\bf L}$ over the toric surface ${\mathcal N}$ which compactifies the orbit. 
 The divisor of zeros of this section is called the {\it spectral curve} $C$. 

The spectral curve $C$ intersects the divisor at infinity ${\mathcal N}_\infty$ at a divisor 
$C_\infty$. A choice of a point  of ${\mathcal N}_0:= {\mathcal N} - {\mathcal N}_\infty$ provides a coordinate on every 
component of the divisor ${\mathcal N}_\infty$. The divisors $C_\infty$ satisfy a single condition: 
\be \la{5.16.1.11}
\mbox{The product of the coordinates of all points  
of $C_\infty$  is equal to $1$.}
\ee
 We call such divisors {\it admissible divisors at infinity}. 
There is a bijection 
\be \la{nunu}
\nu: \{\mbox{zig-zag paths on $\Gamma$}\} \stackrel{\sim}{\lra} 
\mbox{divisor at infinity $C_\infty$}. 
\ee
It is uniquely determined by the condition that the Casimir functions on the phase space given by the monodromies along 
zig-zag paths $\alpha$ are proportional to the Casimir functions given by the coordinates  
of the boundary points $\nu(\alpha)$ on the corresponding components of ${\mathcal N}_\infty$.

Pick a black vertex $b$ of $\Gamma$. A non-zero vector in the fiber $V_b$ 
provides a section $s_b$ of the sheaf ${\rm Coker}K$. The latter sheaf is a line bundle for a smooth $C$.  
The restriction of its divisor of zeros to 
${\mathcal N}_0$ is a degree $g$ effective divisor $S$, 
where $g$ is the genus of $C$. 

The triple $(C, S, \nu)$ is called the {\it spectral data}. 
It was studied in   \cite{KO}.

\subsubsection{A cover of the Beauville integrable system.} \la{1.4.3}
Let ${\mathcal S}$ be the moduli space parametrizing the spectral data. 
Let ${\mathcal B}$ be the moduli space of pairs $(C, \nu)$.
Forgetting $S$, we get a projection
\be \la{bis1}
\pi: {\mathcal S}\lra {\mathcal B}, \qquad  (C, S, \nu) \lms 
(C, \nu).
\ee
We claim that it is an integrable system, which covers  
the Beauville integrable system -- the 
latter does not 
take into account parametrizations (\ref{5.16.1.11}).

Namely, let ${\mathcal S}_{\chi}\subset {\mathcal S}$ be the subspace parametrizing the spectral data $(C, S, \nu)$ 
with a given divisor  $C_\infty $. It carries a rational symplectic  
form $\Omega_{\mathcal S}$ defined as follows.  
The canonical ${\rm T}$-invariant symplectic form $\omega$ on ${\mathcal N}_0$ 
gives rise to the Beauville symplectic form on ${\rm Sym}^g{\mathcal N}_0$:
$$
\Omega_B:= \sum_{i=1}^g s_i^*\omega. 
$$
Here $s_i: {\mathcal N}^g \lra {\mathcal N}$ is the projection 
to the $i$-th factor. There is a projection
$$
p: {\mathcal S}_{\chi} \lra {\rm Sym}^g{\mathcal N}_0, \qquad (C, S, \nu)\lms S.
$$
It is a Galois cover, whose Galois group is a product of symmetric groups  
acting by permuting the points of $C_\infty$ at each boundary component, and hence altering the bijection (\ref{nunu}). 
 Indeed, one shows that, given a generic admissible 
divisor at infinity $C_\infty \subset {\mathcal N}_\infty$ and a 
generic degree $g$ divisor $S\subset {\mathcal N}_0$, there is a unique spectral curve $C$ 
passing through $S$ and intersecting ${\mathcal N}_\infty$ 
at the divisor $C_\infty$. 
So the pull back $\Omega:= p^*\Omega_B$  is a 
rational symplectic form on  
${\mathcal S}_{\chi}$. 
Let ${\mathcal B}_\chi\subset {\mathcal B}$ be the subspace given by the condition that 
$C \cap {\mathcal N}_\infty$ is a given divisor. There is a fibration
\be \la{bis}
\pi_{\chi}: {\mathcal S}_{\chi}\lra  {\mathcal B}_{\chi}.
\ee
Its fibers are Lagrangian. The fiber over a point 
 $(C, \nu)$ is identified with 
an open part of  ${\rm Pic}_g(C)$, 
which is a principal homogeneous space over the Jacobian of $C$. 
So (\ref{bis}) and therefore (\ref{bis1}) are  algebraic integrable systems. 

\subsubsection{Connection with the dimer integrable system.} \la{1.4.4}
A proof of Theorem \ref{10.13.1.10} below will be given in \cite{GK}. 
Combined with the results of this paper 
it allows to quantize the  integrable system (\ref{bis1}). 

\bt \la{10.13.1.10} For any Newton polygon $N$, 
the spectral data provides a birational isomorphism ${\mathbb S}$ over the base 
${\mathcal B}$:
\be \la{SDI}
\begin{array}{ccccc}
{\mathcal X} &&\stackrel{{\mathbb S}}{\lra} && {\mathcal S}\\
&\searrow&&\swarrow &\\
&&{\mathcal B}&&
\end{array}
\ee
It identifies the dimer integrable system with  
the integrable system (\ref{bis1}), matching  
the symplectic leaves of ${\mathcal X}$ and   
the symplectic varieties $({\mathcal S}_{\chi}, \Omega)$. 
\et
The map ${\mathbb S}$ intertwines the action 
of the torus ${\rm T}$ on ${\mathcal X}$ with the action on 
${\mathcal S}$ provided by the action of ${\rm T}$ 
on the surface ${\mathcal N}$.  

In Section 7 we prove a weaker statement: ${\mathbb S}$ is 
a finite cover over the generic part of ${\mathcal S}$. 
It implies the independence of the Hamiltonians. 
%
%
%

\subsection{Analogies between dimers, Teichm\"uller theory and cluster varieties.}

Cluster Poisson varieties provide a framework for study of both (classical and higher) Teichm\"uller theory 
\cite{FG} 
and theory of dimers. Here is a dictionary relating key objects in 
these three theories. 
\begin{center}
\begin{tabular}{ c | c | c }
\bf{Dimer Theory}&{\bf Teichm\"uller Theory}&{\bf Cluster varieties}\\
\hline\hline
Convex integral polygon $N$&Oriented surface $S$ &\\
&with $n>0$ punctures&\\
\hline
Minimal bipartite graphs&ideal triangulations of $S$&seeds\\
on a torus &&\\
\hline
spider moves of graphs&flips of triangulations&seed mutations\\\hline
Face weights&cross-ratio coordinates&Poisson cluster \\
&&coordinates\\\hline
Moduli space of spectral&Moduli space of framed&cluster Poisson variety\\
data on toric surface $N$&$\PGL(2)$ local systems &\\\hline
simple Harnack curve + divisor&complex structure on $S$&\\\hline
Moduli space of &Teichm\"uller space of $S$&positive real points of\\
simple Harnack curves + divisors&&cluster Poisson variety\\\hline
Tropical Harnak curve&Measured lamination&\\
with divisor&&\\\hline
Moduli space of tropical&space of measured&real tropical points of\\
Harnack curve + divisors&laminations on $S$&cluster Poisson variety\\\hline
Dimer integrable system&Integrable system related to&\\
&pants decomposition of $S$&\\\hline
Hamiltonians&Monodromies around loops&\\
&of a pants decomposition&\\\hline
\end{tabular}\end{center}

What distinguishes these two examples --  the dimer theory and the Teichm\"uller theory -- 
from the general theory of cluster Poisson varieties is that 
in each of them the set of real / tropical real points of the relevant cluster variety 
has a meaningful and non-trivial interpretation as the moduli space of some geometric objects.

Here by moduli space of certain objects related to the toric surface ${\mathcal N}$
 we  mean the space 
parametrizing the orbits of the torus ${\rm T}$ acting on the objects. 
For example,  the  moduli space of spectral data means the  space 
${\mathcal S}/{\rm T}$. So combining results of \cite{GK} 
with Theorem \ref{1,1} would imply that the latter 
corresponds to the 
hypersurface ${\mathcal X}_N^0$ in the cluster Poisson variety ${\mathcal X}_N'$.

\paragraph{Discrete cluster integrable systems in Teichm\"uller theory.} 

Consider the moduli space 
${\mathcal X}_{\PGL_2, S}$ of framed $\PGL_2$-local systems on $S$ \cite{FG}. 
If $S$ has no punctures, it is just the moduli space of $\PGL_2$-local systems on $S$, which  
has the classical symplectic structure. If $S$ is a surface with punctures,  the moduli space 
${\mathcal X}_{\PGL_2, S}$ has a cluster Poisson variety structure ({\it loc. cit.}). 
Take a collection of loops $\{\alpha_i\}$ 
providing a pants decomposition of $S$. 
The monodromies of local systems over  the loops $\alpha_i$ 
commute under the Poisson bracket 
and, together with  the Casimirs given by the monodromies around the punctures, provide an 
 integrable system. When $S$ is a surface with punctures it is 
a cluster integrable system. 

Furthermore, the Dehn twists along the loops $\alpha_i$ generate a commutative subgroup 
$\Z^{3g-3}$ in the modular group $\Gamma(S)$ of the surface $S$, which commutes 
with the Hamiltonians and acts by cluster Poisson automorphisms. 
So we get a discrete cluster integrable system. 

\vskip 3mm
We continue to develop these analogies in \cite{GK}. 
Yet parts of this dictionary  are still missing. In particular, it would be very interesting 
to construct a canonical basis -- in the spirit of the Duality Conjectures from 
\cite{FG1} -- in the space of regular 
functions on the cluster variety ${\mathcal X}_N$.

\subsection{Further perspectives.} 
We assigned to a Newton polygon $N$ a cluster integrable system. 
Here is another way to assign to $N$ an integrable system: 

1. Newton polygons $N$, considered up to translation, 
parametrize smooth toric Calabi-Yau threefolds. 
Namely,  the canonical bundle ${K}_{\mathcal N}$  of the toric surface ${\mathcal N}$ assigned to $N$ 
is a toric Calabi-Yau threefold: 
the multiplicative group ${\mathbb G}_m$ acts 
on the fibers, so the torus ${\rm T} \times {\mathbb G}_m$ acts on ${K}_{\mathcal N}$.  
Alternatively,  it can be described as follows. 
Put a Newton polygon $N$ in the $z=1$ plane in the lattice $\Z^3$, 
and make a cone over it. The corresponding toric threefold is a Calabi-Yau. 
Every smooth quasiprojective toric Calabi-Yau threefold is obtained this way. 
See \cite{IU} for the connections between dimer models and toric Calabi-Yau threefolds. 

2. Any Calabi-Yau threefold $Y$ gives rise to an $N=2$ supersymmetric Yang-Mills ($N=2$ SUYM) theory on $\R^4$.  
Namely, consider a IIB type
 string theory in $10$-dimensional space $Y(\C) \times \R^4$ 
and compactify it on 
$Y(\C)$. 

3. An $N=2$ SUYM theory on $\R^4$ gives rise to an algebraic integrable system -- 
see \cite{NS} and references there. 

Combining these three steps we conclude that any Newton polygon $N$ 
gives rise to an algebraic integrable system. 
It is natural to conjecture that it is the dimer cluster 
integrable system assigned to $N$. 
 It would be interesting to prove this, and understand the role 
of the cluster structure, and in particular of quantum integrability.   

\subsection{Historical remarks}

Several of the concepts we use are present with different citations 
in the literature on the dimer model and the literature on 
cluster algebras/cluster varieties. 

Zig-zag paths were defined in the dimer model in \cite{Kenyon.dirac} 
but are implicitly present in work of Baxter
on the Yang-Baxter equation in the $6$-vertex model, see e.g. \cite{Baxter}. 
They also occur in a more general form, for ``bicolored" planar graphs, in \cite{Postnikov}.

Elementary transformations for dimers and in particular the spider move were first defined by Kuperberg
in the 90s
(using the term ``urban renewal" which was coined by Propp), see \cite{Ciucu, KPW}. 
In the cluster algebra literature these are sometimes referred to as Postnikov moves. 
The analogous elementary
transformations for electrical networks (the star-triangle move, also known as Y-Delta move)
appeared in Kennelly \cite{Kennelly}. 

The notion of minimal network appeared for resistor networks in
\cite{CdV} and \cite{CIM}, and for dimers/clusters in \cite{Thurston} and \cite{Postnikov}.
 \vskip 2mm

Cluster algebras were introduced in \cite{FZI}. 
A family of Poisson structures on cluster algebras 
was introduced in \cite{GSV}. 

Cluster Poisson varieties were introduced in \cite{FG1}. 
They are in duality with cluster algebras. 
The field of rational functions on a cluster Poisson variety has
a natural non-commutative $q$-deformation, see Section 3 of  {\it loc. cit.}. 
The cluster modular group, defined there as well, acts by automorphisms of 
the cluster Poisson variety, its $q$-deformation, and the cluster algebra. 
In general it does not, however, preserve the Poisson structures 
introduced in \cite{GSV}. Non-commutative 
 $q$-deformations of the cluster algebra corresponding to 
these Poisson structures were 
defined in \cite{BZ} and called quantum cluster algebras. 

Cluster Poisson varieties are  related to cluster algebras as follows. 
The subvariety of a cluster Poisson variety obtained 
by equating all monomial Casimirs to $1$ is a distinguished symplectic leaf. 
The restrictions of the cluster Poisson coordinates to that leaf can be 
realized as ratios of cluster monomials in the corresponding cluster algebra, considered already in \cite{GSV}. 
The obtained elements generate 
a subfield in the field of fractions of the cluster algebra.  
Its transcendentality degree is smaller than the 
one of the  field of fractions of the cluster algebra 
by the number of  
independent Casimirs, which coincides with the codimension 
of the generic symplectic leaf in the cluster Poisson variety. 
\medskip

In \cite{GSV2} Gekhtman, Shapiro and Vainshtein study a related model of paths in
directed networks on an annulus, originally defined by Postnikov \cite{Postnikov} for planar
networks.  The relation with our model is not completely
straightforward,  but in the cases of interest we believe the models are essentially equivalent.
The results of \cite{GSV2}
have some overlap with ours: in particular they construct the same Poisson structure,
but do not discuss integrability (the annular system is not integrable). 
It would be worthwhile to study more closely the 
connection between these models.

\medskip

The octahedron recurrence, although known as  
Hirota's bilinear difference equation, goes back to Hirota 
\cite{Hir}. Thecube recurrence, 
also known as the discrete BKP equation or 
Hirota-Miwa equation, goes 
back to the work of Miwa \cite{Miwa}, and was rediscovered and studied 
several times afterwords, see \cite{FZL},  \cite{CS}.

\subsection{The structure of the paper.}

In Section \ref{sec3} we study the set  of minimal 
bipartite graphs $\Gamma$ with  a given Newton polygon. 

In Section \ref{Pdefsection} we define a  modified 
partition function, and prove that it is a sum of commuting
Hamiltonians, 
providing a quantum integrable system. 

In Section \ref{spidermovesection} we show that the spider move gives rise to a 
cluster Poisson transformation. We show that these transformations 
preserve the partition functions. 
We also briefly recall, for the convenience of the reader, 
basic definitions of quantum cluster varieties.

In Section \ref{resistornetsection} we study the dimer model of special kind corresponding 
to an arbitrary torus graph $G$, describing  a resistor network on $G$. 

In Section \ref{HBDEsection} 
we construct discrete cluster integrable systems related to the 
dimer models on square grids and dimer models related to hexagonal resistor
networks.

In Section \ref{AGsection} we turn to the algebraic-geometric 
perspective of the story. 
As a biproduct we prove the independence of the Hamiltonians. 


In Section \ref{sec2} we give a local calculation of 
 the Poisson structure on the space ${\L}_\Gamma$, and generalize the story to 
bipartite graphs on open surfaces.

\paragraph{Acknowledgments.} We are very grateful to 
Vladimir Fock, Francois Labourie, Grisha Mikhalkin, Michael Shapiro, Andy Neitzke, Andrei Okounkov, Dylan Thurston, 
Masahito Yamazaki for fruitful discussions. We thank the referees for useful comments.
A.G. was supported by NSF grant DMS-1059129. He is grateful to Aspen Center of Physics and IHES for hospitality; R.K. is supported by the NSF.

\section{Newton polygons and minimal bipartite graphs on surfaces} \la{sec3}

\subsection{Zig-zag paths and minimal bipartite surface graphs.}
Let $E=v_1v_2$ be an oriented edge of an oriented ribbon graph. 
There is a notion of the left and right edges adjacent to $E$: 
the left (respectively the right) edge to $E$ 
 is the previous (respectively the next) for the cyclic order at $v_2$.

Let $\Gamma$ be a bipartite graph on a surface $S$. The \emph{medial graph} 
of $\Gamma$
is the $4$-valent graph with a vertex for each edge of $\Gamma$ and an edge whenever
two edges of $\Gamma$ share a vertex and are consecutive in the cyclic
order around that vertex. A \emph{strand} or \emph{zig-zag strand}
is a path in the medial graph passing consecutively through the edges of a zig-zag path
of $\Gamma$ as in Figure \ref{di25}; equivalently it is a path in the medial graph
which goes ``straight" at each (degree-$4$) vertex.
For every edge $E$ of $\Gamma$ there are exactly two  
strands passing through $E$, see Figure \ref{di25}. 
For every $m$-valent 
vertex $v$ of $\Gamma$ there is a unique  $m$-gon formed by the parts of the strands intersecting the 
consecutive pairs of the edges incident to $v$, see the right picture on 
Figure \ref{di25}. We call it the arc $m$-gon surrounding $v$. 

\begin{figure}[ht]
\centerline{\epsfbox{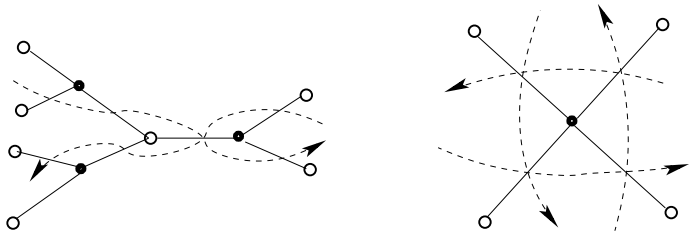}}
\caption{Zig-zag strands and surrounding $4$-gon.\label{di25}}\end{figure} 

\bd
A bipartite surface graph $\Gamma$ is \emph{minimal}  
if, in the universal cover $\tilde \Gamma\subset\tilde S$, strands have no self-intersections, and there are no parallel bigons, 
i.e. pairs of strands  intersecting at two 
points and oriented the same way -- see the left pictures on  Figure \ref{di26}.
\ed
\begin{figure}[ht]
\centerline{\epsfbox{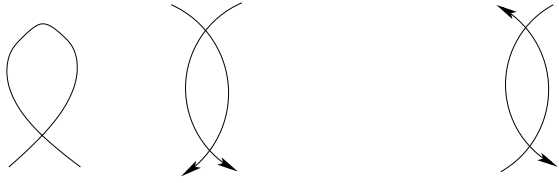}}
\caption{Minimal bipartite surface graphs have no self-intersections and parallel bigons (left); antiparallel bigons (right) 
are allowed. \label{di26}}
\end{figure}

\subsection{From Newton polygons to collections of admissible graphs.} 

Let $N$ be a convex polygon in $\R^2$ with vertices in $\Z^2$, 
considered up to a translation by a vector in $\Z^2$, called below a {\it Newton polygon}. 
It is described by a collection $\{e_i\}$ of 
integral primitive vectors with the zero sum, obtained as follows. 
Take the set of all integral points on the boundary of $N$. 
It is cyclically ordered by the counterclockwise orientation of the boundary.  
The vectors $e_i$ are obtained by connecting each of these points with the next one. 
So the number of integral primitive vectors on the boundary coincides with the number of integral points on the boundary. 

Consider the torus ${\T}   = \R^2/\Z^2$. 
Each vector $e_i$ determines a homology class $[e_i] \in H_1({\T}, \Z)$. 
There is a unique up to translation geodesic representing this class --
 the projection of the vector $e_i$ to the torus. 

Take a family of oriented loops $\{\alpha_i\}$ 
on the torus $\T$ in generic position 
 such that 
the isotopy class of the loop $\alpha_i$ coincides with the isotopy class  of the geodesic 
representing $[e_i]$. 
The union 
$$
G = \cup_i \alpha_i
$$
is an oriented  graph on the torus $\T$. 

\bd
A graph $G$ given by a collection of oriented loops on an oriented surface $S$
 is an \emph{admissible minimal graph}
if the following two conditions hold: 

\begin{enumerate} \la{con1}
\item Going along any of the loops $\alpha_i$, the directions 
of the loops intersecting it alternate. 

\item The total number of intersection points is minimal. 
\end{enumerate}
\ed

The graph $G$ provides a decomposition of the surface ${\rm S}$ into 
a union of polygons $P_j$. The sides of the polygons inherit 
orientations from the loops.   The admissibility condition 1) is equivalent 
to the following: 
 
\begin{enumerate} 
\item[1$'$.] The sides of any of the polygons $P_i$ are either oriented the same way, 
clockwise or counterclockwise, or their directions alternate.  
\end{enumerate}

\subsubsection{From admissible surface graphs to bipartite surface graphs.} \la{sec3.3.1}

The admissibility condition $1')$ 
for a surface graph implies 
that there are three types of domains $P_i$: 
\begin{enumerate}

\item {\it Black domains}: the sides are oriented counterclockwise. 

\item {\it White domains}: the sides are oriented clockwise.

\item {\it Faces}: the directions of the sides alternate. 
\end{enumerate}

Therefore an admissible graph $G$ on $S$ gives rise to a bipartite graph $\Gamma$ on $S$: 
the black (white) vertices of $\Gamma$ are the black (white) domains in ${S}-G$. 
Edges connecting black and white vertices correspond to  the 
vertices of the graph $G$. Indeed, every vertex of $G$ is $4$-valent (since the loops $\alpha_i$ are 
in generic position), and 
locally looks as an intersection of two arrows. So it determines the black and the white domains. 

\subsubsection{Minimal triple point diagrams \cite{Thurston}.} \label{TPDs}

A {\it triple point diagram}\footnote{This
is called in \cite{Thurston} a connected triple point diagram.} is a collection of 
oriented arcs 
in the disc with the ends at the boundary of the disc considered modulo isotopies, such that

\begin{enumerate}
\item Three arcs cross at each intersection point. 
\item  The endpoints of the arcs are distinct points on the boundary of the disc. 
\item The orientation of the arcs induce consistent orientations on the complementary regions. 
\end{enumerate}

Observe that  each complimentary region to a triple point diagram is a disc.

The orientation condition 3) is equivalent to the following:
\begin{enumerate}
\item[3$'$.] Orientations of the arcs alternate ``in'' and ``out'' around each intersection point. 
\end{enumerate}

Let $n$ be the number of arcs. 
Then each of the $2n$ endpoints is oriented ``in'' or ``out'', 
and these orientations alternate as we move around the boundary. 
Furthermore, given a collection of $2n$ alternatively oriented points on the boundary, 
a triple crossing diagram inducing such a collection provides a matching between 
the ``in'' and ``out'' points.

A triple point diagram is {\it minimal} if the number of its intersection points is not bigger 
than for any other diagram 
inducing  the same matching on the boundary. 

\bt[\cite{Thurston}] \la{DT}
1) In a disc with $2n$ endpoints on the 
boundary, all $n!$ matchings of ``in'' endpoints with ``out'' endpoints are achieved by 
minimal triple point diagrams.

2) Any two minimal triple point diagrams with the same matching on the 
endpoints can be related by a sequence of $2 \leftrightarrow 2$ moves, shown on Figure \ref{di18}.
\et

\begin{figure}[ht]
\epsfxsize150pt
\centerline{\epsfbox{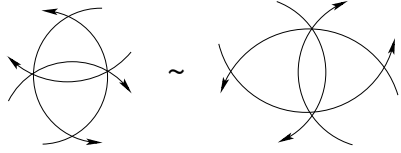}}
\caption{The $2$ - $2$ move.\label{di18}}
\end{figure}

\subsection{From minimal triple point diagrams to minimal bipartite surface graphs.} 
Each crossing point of a triple point diagram has a canonical resolution, 
obtained by moving a little bit one of the strands passing through the point so that 
the emerging  little triangle is oriented by the strands 
the same way as the surface -- 
the counterclockwise resolution shown on Figure \ref{di5b}. 
We apply this procedure to every intersection point of a triple point diagram. 

The complimentary domains for the obtained configuration of arcs are of three possible types: 
{\it black domains}, {\it white domains} and {\it faces}, according to definitions (1)-(3) in section \ref{sec3.3.1}. 
Namely, the black and white domains are the ones consistently 
oriented---counterclockwise and clockwise respectively---while the 
directions of the sides of the remaining domains, 
the faces, alternate.  
So we arrive at  a bipartite graph associated to a minimal triple point diagram, 
see Figure \ref{di5b}. Its black vertices correspond to the triple 
crossing points, and thus  
are $3$-valent. The white vertices can have valencies two, three or higher. 

Let us shrink all edges incident to  
white $2$-valent vertices (and remove the white vertex, gluing its two neighbors into
a single vertex). The bipartite graph obtained can have 
both black and white vertices of valencies four or higher. 

\begin{figure}[ht]
\centerline{\epsfbox{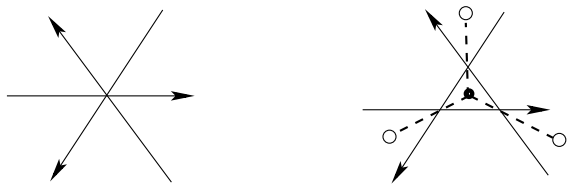}}
\caption{\label{di5b}The bipartite graph obtained by counterclockwise resolution of a triple crossing diagram.}
\end{figure}

Let us call two bipartite graphs 
{\it similar} if the graphs obtained by shrinking their $2$-valent 
white vertices coincide. There is a canonical bijection between 
the faces as well as between the perfect matchings for similar bipartite graphs. 
So similar bipartite graphs have the same partition function, 
discussed in Section \ref{pttnfnsection}. 

\begin{remark}
There is another option: resolving all triple crossings clockwise. 
If a triple crossing diagram 
was associated to a Newton polygon $N$, the counterclockwise resolution is the same as 
the clockwise resolution of the diagram associated to  the polygon 
$N$ rotated by $180$ degrees. Indeed, the rotation 
changes the orientations of all strands of 
the triple crossing diagram. Then resolving clockwise 
is the same as resolving counterclockwise the original triple crossing diagram. 
\end{remark}

\subsection{Boundary intervals of the Newton polygon and zig-zag paths.}

Recall that a zig-zag path has a natural orientation
 such that it turns right at the black vertices. 
Therefore the  homology class of a  zig-zag path is well defined. It 
is a natural invariant of a zig-zag path. 
On a minimal bipartite graph $\Gamma$ zig-zag path has no selfintersections, and so 
it is naturally a simple loop.

\bl \la{4.18..10.5}
Let $N$ be a Newton polygon, and $\Gamma$  a minimal bipartite graph on the torus 
assigned to $N$. The homology class of a zig-zag loop is a primitive vector on a  side 
of $N$, whose orientation is induced by the orientation of $N$. 
The cyclically ordered set of zig-zag paths with a given homology class 
is identified with the cyclically ordered set of 
boundary intervals on the corresponding side of $N$. 
So  the number of zig-zag paths on $\Gamma$ equals to the number of integral boundary points 
on $N$. 
\el

\begin{proof} It is clear from Figure \ref{di16} that  
zig-zag paths on the bipartite graph arising from a triple crossing diagram correspond to 
the strands of the triple crossing diagram. On the other hand 
by construction the latter correspond to the primitive vectors of 
the boundary of the Newton polygon. 
\end{proof}
 
\begin{figure}[ht]
\centerline{\epsfbox{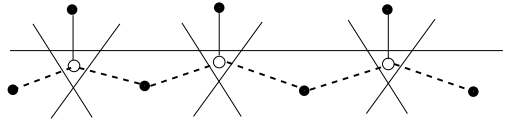}}
\caption{Zig-zag paths correspond to strands of the triple crossing diagram. \label{di16}}
\end{figure}

\subsection{Elementary transformations of bipartite surface graphs.} 

Figure \ref{di23} shows an elementary transformation $\Gamma_1 \to \Gamma_2$
of bipartite graphs, called a {\it spider move}. 
(There are two versions of the move depending on the color of the two 
interior vertices.)
We will see below that for appropriate choice of weights
the spider move does not change the partition function for the dimer model. 
\begin{figure}[htbp]
\centerline{\epsfbox{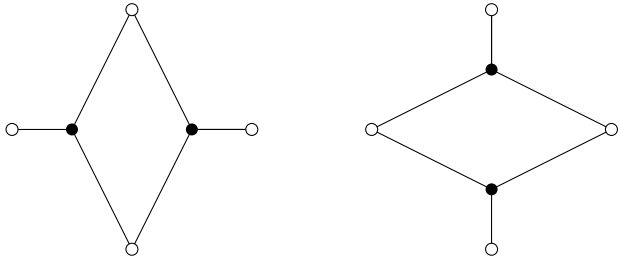}}
\caption{A spider move $\Gamma_1 \to \Gamma_2$ of bipartite graphs.
The white vertices are attached to the rest of the graph; the black vertices 
are only attached to white vertices as shown.\label{di23}}
\end{figure} 

Figure \ref{di23} shows a pair of elementary transformations 
of bipartite graphs: splitting a vertex of valency $k>3$ and 
collapsing a $2$-valent vertex. We will see below that it
 does not change the partition function. 

\subsection{From minimal bipartite surface graphs to minimal triple point diagrams.} 
A minimal bipartite graph can be transformed back into a triple point diagram 
in two steps. 

{\it Step 1}. By a sequence of moves inverse to the shrinking of a $2$-valent white vertex, 
we replace an arbitrary  bipartite graph by a bipartite graph with black vertices of valency $3$, 
see Figure \ref{di24}.

\begin{figure}[ht]
\centerline{\epsfbox{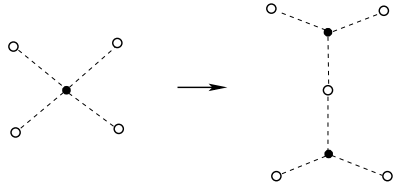}}
\caption{Replacing a $4$-valent black vertex by two $3$-valent black vertices.
A similar vertex splitting splits any $k$-valent black vertex into $k-2$ $3$-valent ones.
\label{di24}}
\end{figure}

{\it Step 2}. 
Take a minimal bipartite graph $\Gamma$ on $S$ with the black vertices of valency $3$. 
Draw  zig-zag strands on $S$ assigned for all edges of $\Gamma$ --  two per edge. 
So each vertex of $\Gamma$ is surrounded by an oriented arc $m$-gon. 
Then there is a unique way to shrink the counterclockwise triangles into points, getting a triple point diagram. 
This procedure is evidently inverse to the one we used to get a bipartite graph from 
a triple point diagram. 

\begin{figure}[ht]
\epsfxsize70pt
\centerline{\epsfbox{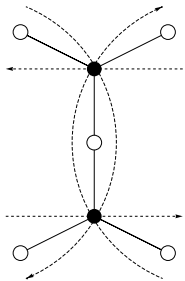}}
\caption{\label{di17}From bipartite graphs (solid) to triple point diagrams (dashed).}
\end{figure} 
 
\subsection{Graphs from polygons} 

Here is the main result of this section.

\bt \la{3.21.10.1}
i) For any Newton polygon $N$ there exists a minimal admissible graph $G$ on a torus
associated with $N$. It produces a minimal bipartite torus graph $\Gamma$
associated with $N$.

ii) Any two minimal bipartite graphs 
on a torus associated with $N$ are 
related by a sequence of spider moves and shrinking / expanding of $2$-valent vertices.
\et

This result is illustrated in an example in Figure \ref{NtoG}.
\begin{figure}[ht]
\centerline{\epsfxsize70pt
\epsfbox{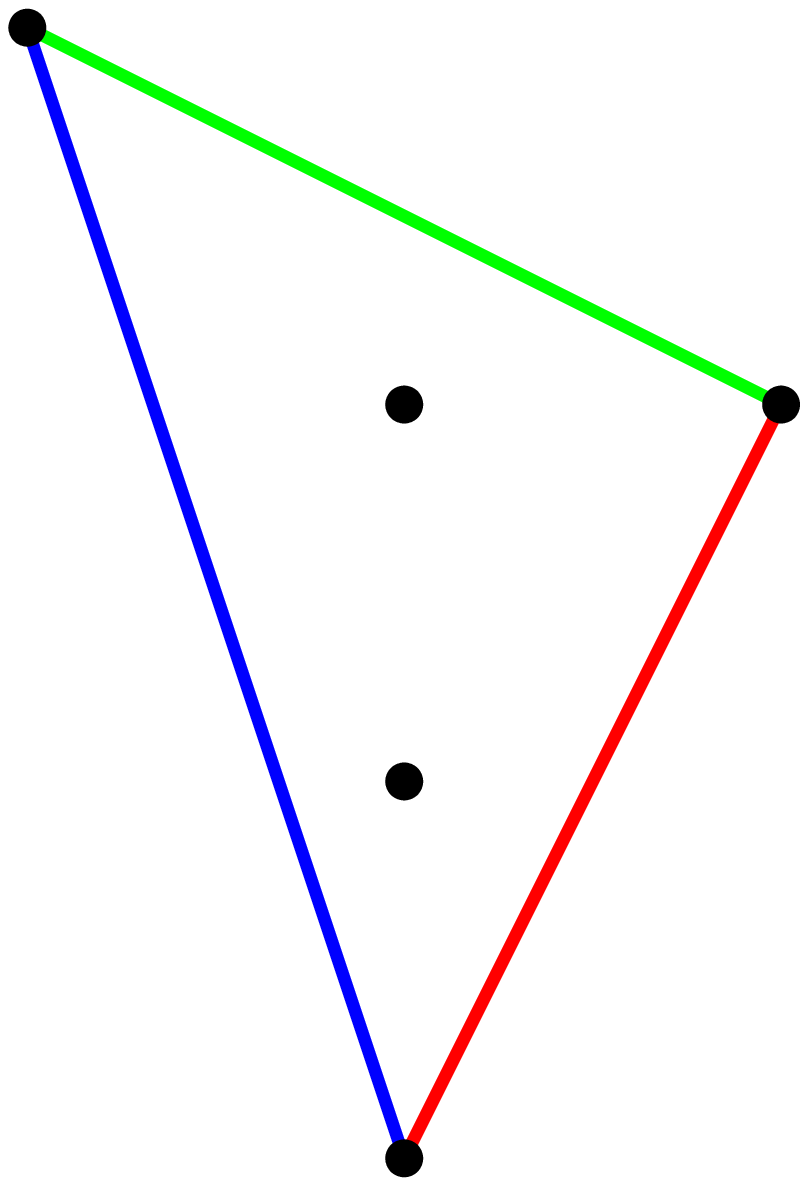}$\rightarrow$\epsfxsize90pt\epsfbox{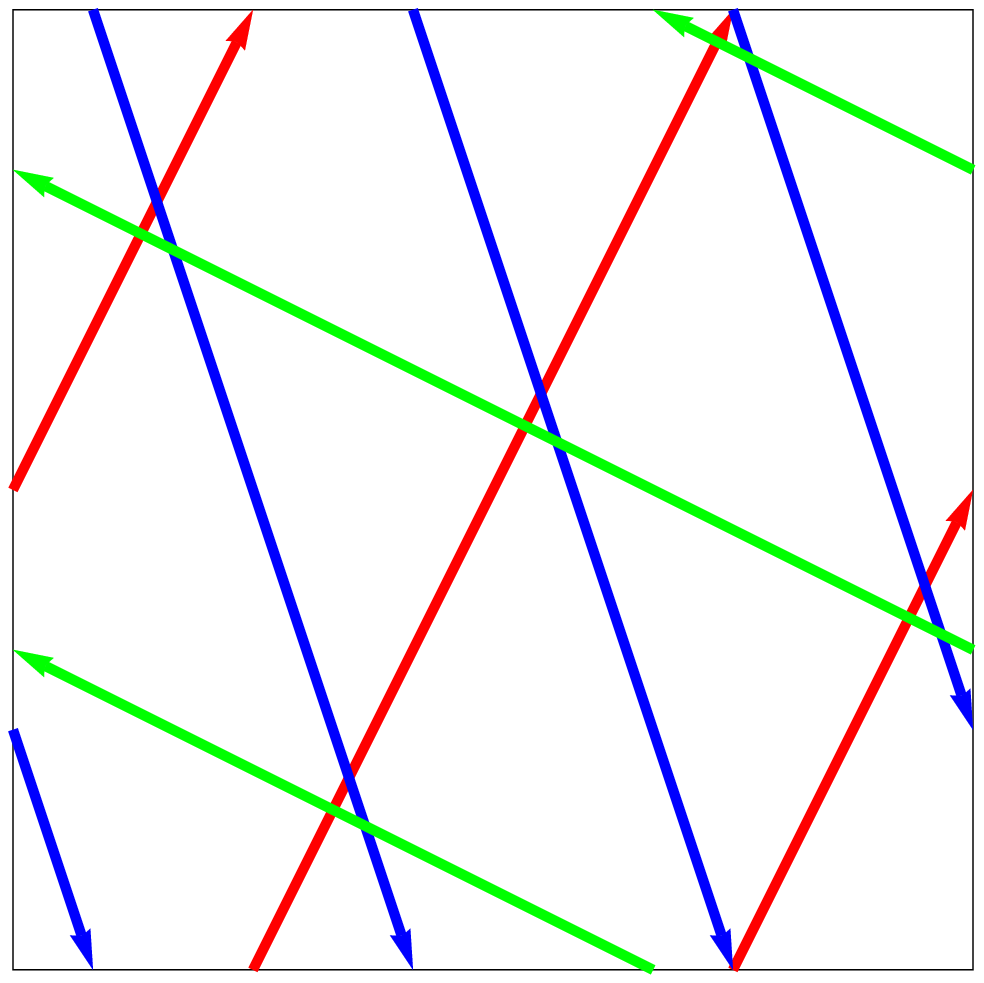}$\rightarrow$\epsfxsize90pt\epsfbox{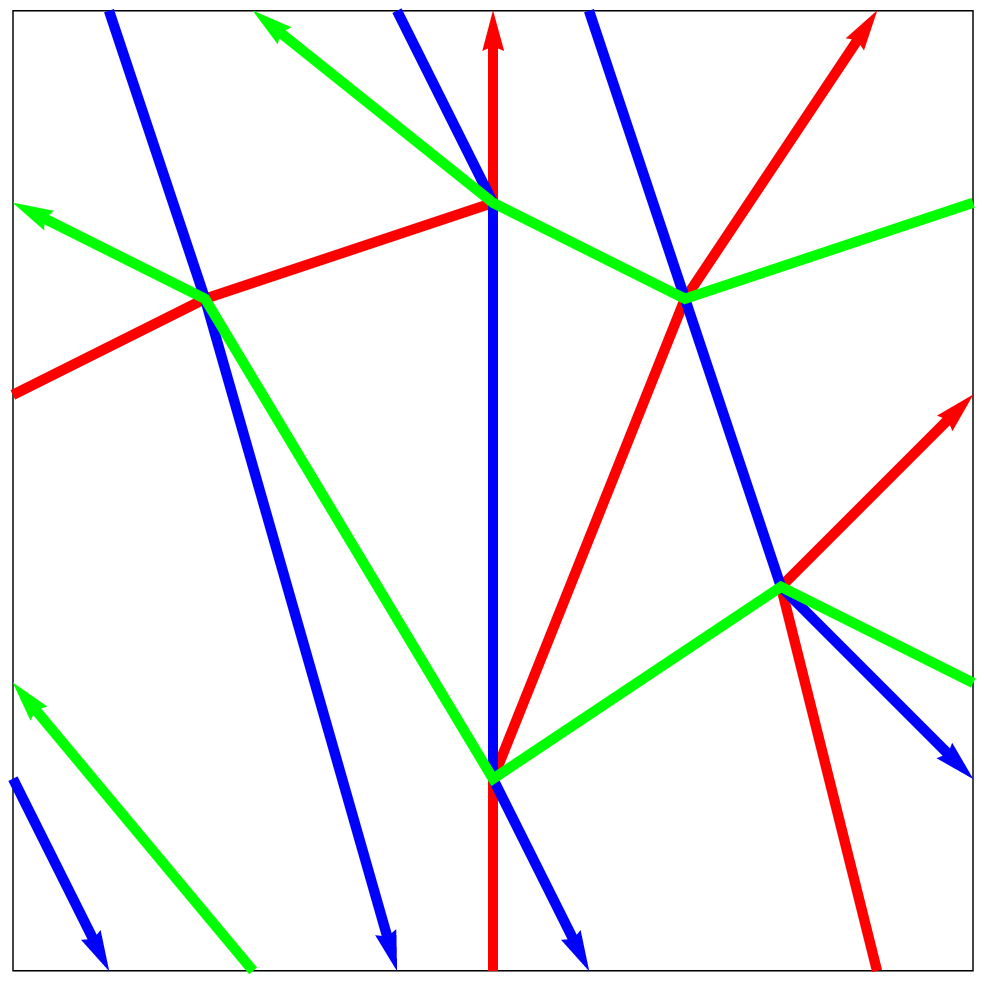}$\rightarrow$\epsfxsize90pt\epsfbox{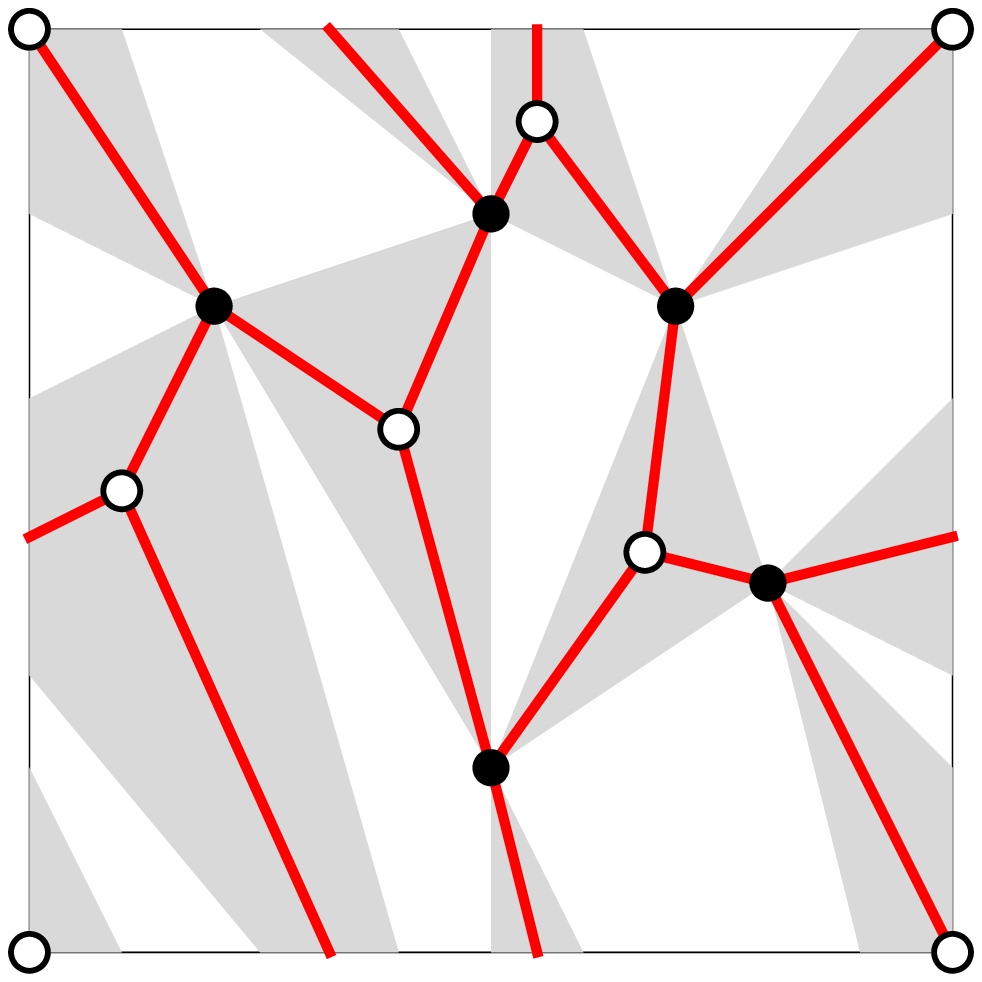}}
\caption{\label{NtoG}Procedure for going from a convex integer polygon to a bipartite
torus graph: 1. A convex integer polygon. 2. The corresponding geodesics
on a torus. 3. Isotoping the geodesics to a triple-crossing diagram having specified
order of intersection with the boundary of the square. 4.
The resulting bipartite graph.}
\end{figure} 

\begin{proof}
Let us glue a torus $\T$ from a rectangle $R$ by identifying opposite edges.
Consider a finite collection of oriented simple loops $\alpha_i$ on $\T$, whose homology classes
correspond to the primitive edge vectors of $N$, read counterclockwise
around $\partial N$. We choose the loops
so that the total number of intersections of any loop with an edge of $R$ is minimal:
this can be accomplished by starting with Euclidean geodesics (but we will isotope
these around preserving the number of intersections with $R$).

Since $\partial N$ is a closed path the total homology class
$\sum_i[\alpha_i]$ represented by these loops is zero. So the total intersection number
of these loops with any edge of the rectangle $R$ is zero.
Therefore we can isotope the loops,
by moving them one through the other, so that the intersection points of the loops with an edge of $R$ alternate
in orientation: outgoing loops interlaced with ingoing loops.
During this isotopy we do not add any more crossings of loops with $\partial R$,
and furthermore we do not exchange two outgoing or two ingoing loops, that is,
the order of outgoing paths and the order of ingoing paths does not change.
We can do this simultaneously for all edges of $R$, so that the crossings
of the loops with $\partial R$ alternate around the entire boundary of $R$.
Note that we have not introduced any new crossings of "outgoing" paths or crossings 
of "ingoing" paths, only crossings of one outgoing and one ingoing path.
Then we can apply the part 1) of Theorem \ref{DT}, and
isotope the paths $\alpha_i$ into
a minimal triple point diagram while fixing the boundary points on $\partial R$.

Note that each path $\alpha_i$ intersects the boundary with a consistent
orientation (consistent with being a straight geodesic). Thus no path $\alpha_i$
can self-intersect: a "monogon" on $\T$ is homotopically trivial and so must either lie entirely in the interior of $R$
(disallowed by minimality of the triple point diagram)
or intersect one of the sides of $R$ with
both orientations (disallowed by construction). Furthermore in $\T^2$ there are no
parallel bigons: such a bigon cannot lie in the interior of $R$ by minimality
of the triple point diagram, and if it crosses the boundary, its vertices correspond to crossings of two outgoing
or two ingoing paths which are not present in the original geodesic path diagram (in other words,
along any isotopy back to a straight geodesic diagram there must occur at some point
an uncrossing of two outgoing paths, a contradiction).

\begin{figure}[ht]
\epsfxsize240pt
\centerline{\epsfbox{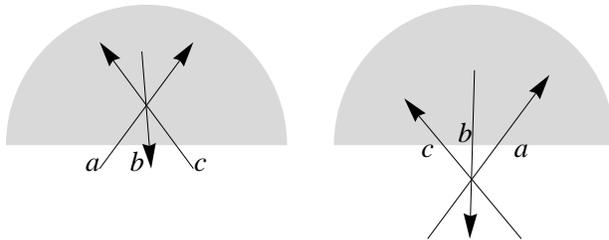}}
\caption{\label{transposition}Permuting the boundary intersections.}
\end{figure}

The components of the intersection of a loop with the interior of $R$ will be called strands.
Thanks to the part 2) of Theorem \ref{DT},
any two minimal triple point diagrams whose strands intersect $\partial R$ in the same order are related by $2 \leftrightarrow 2$ moves. We claim that by isotoping the loops
on the torus $\T$
any two minimal triple point diagrams for the same polygon $N$ are related by
$2 \leftrightarrow 2$ moves.  Consider the set of strands intersecting
an edge $E$  of $R$. Some of these strands may be part of the same loop,
or part of isotopic loops. The relative order of intersection of these subsets of strands
with $\partial R$ is fixed.
However the strands which are parts of different non-isotopic geodesics do not have
a unique relative order along $\partial R$. We need to show that by isotopy and
$2 \leftrightarrow 2$ moves we can permute the order of intersection of these
sets of strands with $E$.

Let $a,c$ be consecutive ``ingoing'' strands on $\partial R$
which by the alternating property necessarily are separated by a single
``outgoing'' strand $b$, so that $a,b,c$ are three
consecutive intersections of strands with $\partial R$. Suppose that
the strands ending at $a,c$ cross inside
the rectangle $R$.
We show that using $2 \leftrightarrow 2$ moves and isotopy
on $\T$ we can rearrange these strands so that they intersect $\partial R$
consecutively in the reverse order $c,b,a$
without changing the order of the other
intersections. In other words, we can arbitrarily permute
the ``ingoing'' strands by doing a sequence of these transpositions $\{a,c\}\to \{c,a\}$
as long as they intersect in $R$,
and similarly for the ``outgoing'' strands.
To accomplish this transposition, \cite{Thurston} shows that 
we can use $2 \leftrightarrow 2$ moves so that the triple point
diagram inside $R$ has the three strands at $a,b,c$ meeting just adjacent to the boundary (that is,
before meeting any other strands). Then one can simply isotope this triple crossing across the boundary
which has the effect of switching the intersections order of the strands to $c,b,a$
as desired. See Figure \ref{transposition}.

Now given a triple point diagram one can construct the minimal bipartite graph 
as above. Observe that there are two different $2 \leftrightarrow 2$ move diagrams: 
the one is obtained from the other by reversing the strand orientations. 
Resolving their triple intersection points in the standard (counterclockwise)  way, we get 
two elementary moves of bipartite graphs: the spider move on Figure \ref{di19} and 
the move on Figure \ref{di20} which amounts to shrinking the edge with a $2$-valent (white) vertex, and expanding it the other way. 
\end{proof}

\begin{figure}[htbp]
\epsfxsize230pt
\centerline{\epsfbox{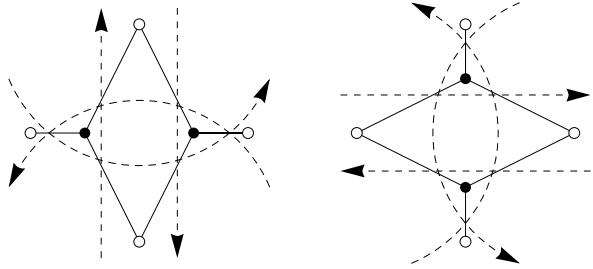}}
\caption{\label{di19}This $2 \leftrightarrow 2$ move amounts to a spider move of bipartite graphs.}
\end{figure}

\begin{figure}[htbp]
\epsfxsize230pt
\centerline{\epsfbox{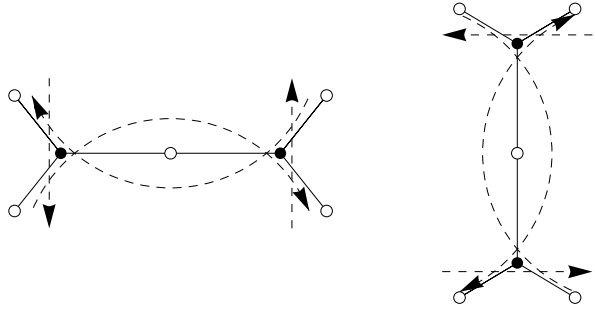}}
\caption{\label{di20}This $2\leftrightarrow 2$ move amounts to collapsing the edges incident to a white $2$-valent vertex and expanding the other way.}
\end{figure}

\section{Integrability of the dimer model} \label{Pdefsection}

\subsection{Edge weights and line bundles with connections on a graph}\la{edgeweights}
A vector bundle with connection on a graph $\Gamma$ is the following 
data: we assign to each vertex $v$ of $\Gamma$  
a vector space $V_v$, and to each oriented edge 
$E$ from a vertex $v_1$ to a vertex $v_2$ an isomorphism
$$
i_{E}: V_{v_1} \to V_{v_2}
$$ 
called parallel transport, 
so that $i_{-E} = i^{-1}_{E}$, where $-E$ is the edge $E$ with 
the opposite orientation.  
The monodromy of the connection along a loop is the composition
of the parallel transports around the loop.  
Two vector bundle with connections $\{V, i_E\},\{V', i'_E\}$ are isomorphic (gauge equivalent) if there are isomorphisms
$\psi_v:V_v\to V'_v$ such that $i'_E = \psi_{v_2}\circ i_E\circ\psi_{v_1}^{-1}$ for all
edges $E=v_1v_2$.

The moduli space  ${\L}_\Gamma$ is the space of 
line bundles with connections  on a graph $\Gamma$ modulo isomorphisms. 
Here is another  description of the moduli space ${\L}_\Gamma$. 
Choose a non-zero vector $e_v$ in each space $V_v$. Then 
a line bundle with connection on $\Gamma$  amounts to a collection of non-zero numbers, the {\it edge weights} 
$a_E$, assigned to the oriented edges $E$ of the graph, 
so that $a_{-E} = a^{-1}_{E}$:
$$
i_{E}e_{v_1} = a_{E}e_{v_2}, \qquad E= v_1 \to v_2.
$$
Another choice of base vectors $e'_v= b_ve_v$ gives rise to 
a different set of edge weights $a'_E$, related to the original one by a gauge 
transformation:  $a'_E = b_va_Eb_{v'}^{-1}$. 
So 
$$
{\L}_\Gamma = \frac{\mbox{edge weights}}{\mbox{gauge 
transformations}}.
$$

Given a trivializations $\{e_v\}$ of a line bundle with connection $\{V, i_E\}$, 
we can encode the connection by the Kasteleyn operator $K$, see Section \ref{1.4.1}, 
which in this case is just a matrix. 
However the determinant ${\rm det}(K)$ is not gauge invariant, and so is not 
a function on the moduli space ${\cal L}_\Gamma$.

\subsection{The partition function and the Hamiltonians}\label{pttnfnsection}
\subsubsection{The weight of a dimer cover}

A {\it dimer cover} of a bipartite graph 
$\Gamma$, also known as a {\it perfect matching},  
is a collection  of edges of $\Gamma$ 
such each vertex is the endpoint of a unique edge in the dimer cover. 

Recall the determinant line ${\bf L}$, see (\ref{dl}). 
Given a connection, a dimer cover $M$ on $\Gamma$ provides 
an ${\bf L}$-valued function $W(M)$  on  
${\L}_\Gamma$, called  the {\it weight of $M$}. For
an edge $E=bw$, let $i_E$ be the parallel transport along $E$ of $L$. Then 
$i_E\in  V_{b}^* \otimes V_{w}$. 
The product of these vectors over the edges of $M$ is the weight:
$$
W(M) \in {\mathcal O}({\L}_\Gamma) \otimes {\bf L}.
$$
Precisely, choose an order of the edges $E_1, ..., E_N$ of the matching. Write 
$i_E = a^*_E \otimes b_E$. Then $W(M) := a^*_{E_1} \wedge ... \wedge a^*_{E_N} \otimes b_{E_1} 
\wedge\dots\wedge b_{E_N}$. It does not change when we alter the order of edges. 
The following result, due to Kasteleyn (in different language),  is well known. 
\bt[\cite{Kast}]\label{detK} We have
\be\label{sf}\det K = \sum_{M}{\rm sgn}(M)W(M)\in {\mathcal O}({\L}_\Gamma) \otimes {\bf L}\ee where the sum is over dimer covers and the
sign ${\rm sgn}(M)$ only depends on the choice of the Kasteleyn line bundle with connection  
(see Section \ref{1.4.1}) 
and the homology class
modulo $2$ of $M$ (defined in Section \ref{sec3.2.2}).
\et

\bd The \emph{partition function} is 
${\mathcal P}_{\Gamma}=\det K$.
\ed 

The partition function depends on our choice of $\kappa$,
which we fixed once and for all.

\subsubsection{The homology class of a dimer cover}\la{sec3.2.2}
A dimer cover $M$ on $\Gamma$ provides a function 
$$
\omega_M: \{\mbox{Edges of $\Gamma$}\} \to \{0,1\}.
$$ 
Its value on an edge is $1$ if and only if the edge belongs to $M$. 
We view it as a $1$-chain 
$$
[M]:= -\sum_E \omega_M(E)[E],
$$
where the sum is over all edges, and $[E]$ is the edge $E$ oriented from black to white. 
The  condition that $M$ is a dimer cover  means that 
$$
d [M] = \sum_{v}{\rm sgn}(v)[v]
$$
where the sum is over all vertices $v$ of $\Gamma$, and ${\rm sgn}(v)
=1$ for the black vertex $v$, and ${\rm sgn}(v) = -1$ for the white.  

If $M_1,M_2$ are dimer covers, $[M_1]-[M_2]$ is a $1$-cycle,
and thus each dimer cover defines an ``affine" homology class. It is more
convenient to define an absolute homology class for each dimer cover 
by fixing a chain
$\Phi$ such that $[M]-\Phi$
is a cycle (for all $M$). We do this in Section \ref{3.2.2}. 

\vskip 3mm
{\bf Remark}. 
A dimer cover $M_2$ trivialises 
the determinant line ${\bf L}$ by providing isomorphisms 
$V_b \to V_w$ for each edge $E=bw$. In Section \ref{3.2.2} 
we consider ``fractional'' dimer covers provided by 
weighted zig-zag paths, and  
use them in a similar way to trivialize ${\bf L}$.

\subsubsection{The $\alpha$-deformed partition function} \la{3.2.2}
For every edge of an oriented ribbon graph there are exactly two 
zig-zag paths containing the edge (one in which the edge is traversed in 
either orientation).
Let $E$ be an edge of $\Gamma$. 
Denote by $z^r$ (respectively $z^l$) 
the zig-zag path containing $E$ in which the white vertex of $E$
precedes the black (respectively follows the black). 
Let us consider a map 
\be \la{alpha1}
\alpha: \{\mbox{zig-zag paths on $\Gamma$}\} \lra \R/\Z.
\ee
It provides a function 
\be \la{alpha2}
\varphi_\alpha = \varphi:  \{\mbox{Edges of $\Gamma$}\} \lra \R, \qquad \varphi(E):= 
\alpha_r-\alpha_l.
\ee
where
$\alpha_r$ (respectively $\alpha_l$) is the evaluation  of the map
$\alpha$
on the zig-zag
path $z^r$ 
(respectively $z^l$) crossing the
edge $E$,  see  Figure \ref{di33}, and $\alpha_r-\alpha_l$ denotes the
length of the counterclockwise arc on $S^1=\R/\Z$ from $\alpha_l$ to $\alpha_r$.

\begin{figure}[ht]
\epsfxsize120pt
\centerline{\epsfbox{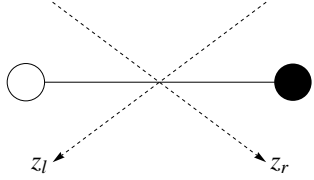}}
\caption{\label{di33}The weight is $\varphi(E) = \alpha_r-\alpha_l$.}
\end{figure} 

The strands, and thus the  oriented zig-zag paths,  are in bijection with the 
set $\{e_1,\dots,e_n\}$ of primitive edges of the boundary $\partial
N$ of $N$, and so inherit a circular order from $\partial N$. 

\bt \la{4.17.10.1}
\label{monotone1}
Let $X$ be the set of circular-order-preserving maps 
$$
\{e_1,\dots,e_n\} \to \R/\Z.
$$ 
For any map $\alpha \in X$, the function
$\varphi_\alpha$ satisfies
\be \la{4.17.10.111}
d \varphi_\alpha = \sum_{v}{\rm sgn}(v)[v].
\ee
\et

The result is more generally true for arbitrary triple crossing diagrams in a disk.

\begin{proof}
We need to show that $d\varphi_\alpha= 1$, respectively $-1$,  at a black, respectively white, vertex.
For this it suffices to show that at each vertex, the set of adjacent strands (the strands
intersecting edges emanating from the vertex) has the same circular order as the 
circular order of the edges containing the vertex. 
This follows from Lemma \ref{monotone} below. \end{proof}

\begin{lemma}\label{monotone}
Given a triple-crossing diagram in a disk, at any triple-crossing
the three intersecting strands occur in the same circular order as their
ingoing endpoints on the boundary of the disk.
\end{lemma}

\begin{proof} The proof is by induction on the number of strands.
Order the strands according to the counterclockwise circular order of their
ingoing endpoints. 
We recall the construction of a triple crossing diagram of \cite{Thurston}:
take a strand with the property
that on one side, no other strand has both its endpoints. One can draw this strand 
following the boundary as illustrated in Figure \ref{monotonefig}, crossing the intermediate
strands in pairs. Each of these crossings is positive. 
The remaining region again has the property that the strands
alternate orientation along the boundary and so 
can be completed to a triple crossing diagram with one fewer strands by induction. 
Moreover the ingoing strands of the new region retain their original circular order.
\end{proof}

\begin{figure}[ht]
\epsfxsize110pt
\centerline{\epsfbox{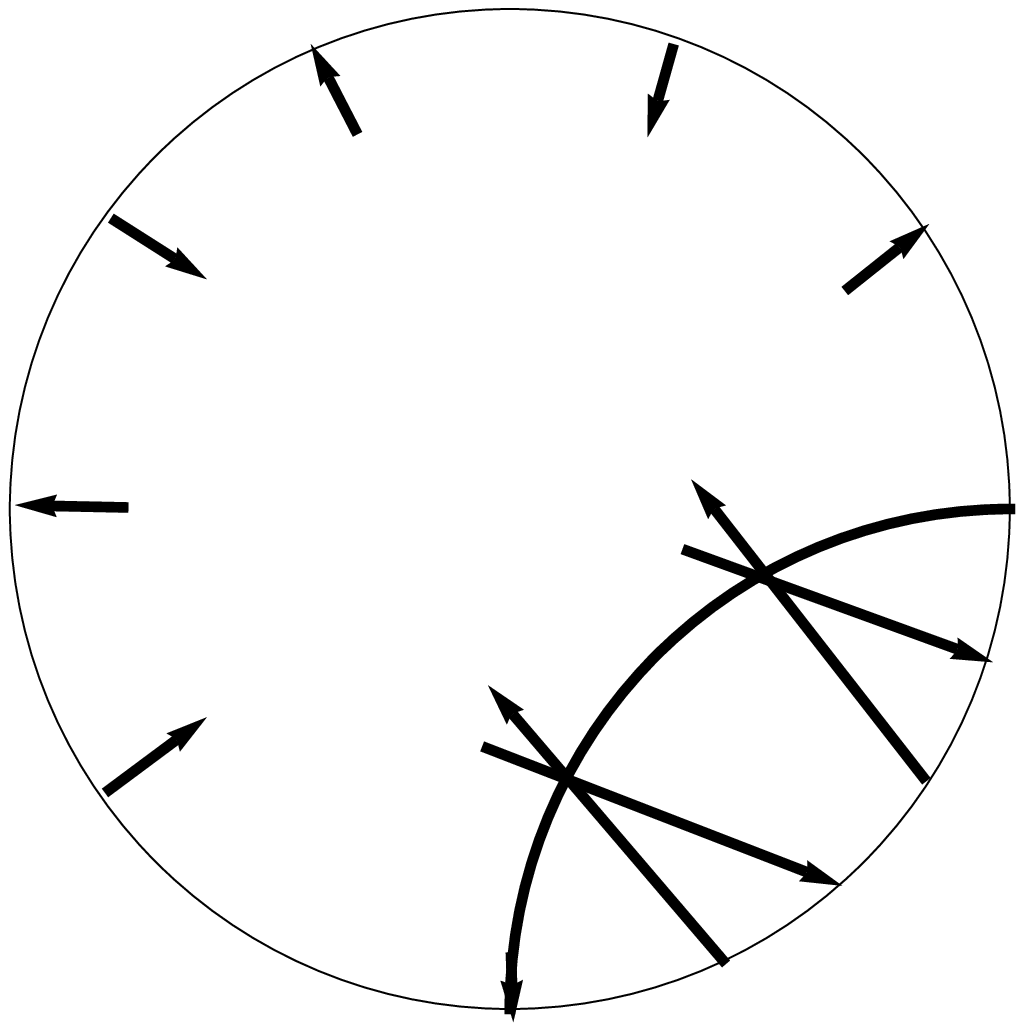}}
\caption{\label{monotonefig}}
\end{figure}

By taking limits in which all points $\alpha(e_i)$ collide (but preserve their circular order)
there are circular-order-preserving maps $\alpha$ for which
$\varphi_\alpha$ takes integer values $\varphi_{\alpha}\in\{0,1\}$.

Given a map $\alpha$ provided by Theorem \ref{4.17.10.1}, there is a $1$-chain 
$$
\Phi_\alpha=\sum_E \varphi_\alpha(E)[E]\quad \mbox{with} \quad 
d\Phi_\alpha=\sum_{v}{\rm sgn}(v)[v].
$$  
It is an integral $1$-chain if $\phi_{\alpha}$ is integer-valued.
Thus
$[M] - \Phi_\alpha$ is an integral cycle. It provides a monomial function on  
${\L}_\Gamma$
\be \la{pafa}
\omega_\alpha(M): {\L}_\Gamma \lra \C^*, 
\ee
given by the monodromy around the cycle $[M] - \Phi_\alpha$. 
It is the complex character of the group ${\L}_\Gamma$ 
corresponding to the homology class $[M] - \Phi_\alpha$ via the isomorphism (\ref{3.5.09.1}). 

\bd Let $\Gamma$ be a bipartite surface graph. Let
 $\alpha$ be a map satisfying (\ref{4.17.10.111}). 
The 
\emph{$\alpha$-deformed partition function ${\mathcal P}_{\alpha}$} 
of the dimer model on $\Gamma$ is a function 
on the space ${\L}_\Gamma$ given as the sum 
of monomials $\omega_\alpha(M)$ over all dimer covers
$M$ of $\Gamma$, with the same signs as in (\ref{sf}): 
\be \la{partsum}
{\mathcal P}_{\alpha}:= \sum_{M}{\rm sgn}(M) \omega_\alpha(M).
\ee
\ed

\subsubsection{The Hamiltonians} \la{Hamcomm}

Let $[M]_\alpha\in H_1(\Gamma, \Z)$ 
be the homology class of the cycle $[M]-\Phi_\alpha$. 
Recall the projection 
$\pi: H_1(\Gamma, \Z) \to H_1(S, \Z)$. 
Given a homology class $a\in H_1(S, \Z)$, 
take the part of the sum (\ref{partsum}) over the matchings $M$ 
such that the projection of  $[M]_\alpha$ to $H_1(S, \Z)$ is the class $a$:
$$
H_{\alpha; a} := \sum_{M: \pi([M]_\alpha) =a}{\rm sgn} (M)\omega_\alpha(M).
$$
By Theorem \ref{detK} the sign ${\rm sgn} (M)$ depends only on $a$, and the partition function  
is a sum
\be \la{decomph}
{\mathcal P}_{\alpha} = \sum_{a\in H_1(S, \Z)} {\rm sgn} (a)H_{\alpha; a}.
\ee

\bd 
The functions $H_{\alpha; a}$ in (\ref{decomph}) are the {\it Hamiltonians of the dimer system}. 
\ed

\begin{remark} A different map $\alpha'$ leads to another collection  
of Hamiltonians $\{H_{\alpha'; a}\}$ which differ from $\{H_{\alpha; a}\}$ 
by a common monomial factor, which lies in the center of the Poisson algebra ${\cal O}({\cal L}_\Gamma)$. 
Indeed, the monodromy along any zig-zag loop is a Casimir element. 
Therefore the  Hamiltonian flows 
do not depend on the choice of $\alpha$,  and on each symplectic leaf any two collections of Hamiltonians 
differ by a common factor. 
\end{remark} 

\subsubsection{A coordinate expression.} 
Let $z_1, ..., z_{2g}$ be loops on $\Gamma$ whose homology classes 
generate $H_1(S, \Z)$. Let $\sum n_i[z_i]$ be the homology class 
of the cycle  $[M] - \Phi_\alpha$. 
Then there is a function $\psi_M: \{\mbox{Faces of $\Gamma$}\} \to \R$, 
which we interpret as a $2$-chain, such that 
$$
d\psi_M = [M] - \Phi_\alpha - \sum n_i[z_i].
$$
So one has 
$$
{\mathcal P}_{\alpha}:= \sum_{M} \pm \prod_{F}W_F^{\psi_M(F)}\prod_iz_i^{n_i}.
$$
Each Hamiltonian is, up to a sign,  
a sum of monomials with coefficients $+1$ in variables
$W_F,z_i$ with a given value of 
$(n_1, ..., n_{2g})$. 
If $S$ is the torus, $P_\alpha$ is a polynomial
in $z_1,z_2$.

\subsection{The classical Hamiltonians commute} \la{chcom}

A minimal bipartite graph on a torus   
provides  an algebraic integrable system:

\bt \la{completeintegrabilitytheorem}
Let $\Gamma$ be a minimal bipartite graph on a torus $\T$. Then 

i) The 
Hamiltonians $H_{\alpha; a}$ commute under the Poisson bracket 
 on 
${\L}_\Gamma$. 

ii) The Hamiltonians are independent and their
 number is the half of the dimension of the generic symplectic leaf. 
\et

The part i) of Theorem \ref{completeintegrabilitytheorem} 
is proved in Section \ref{chcom}. 
The number of Hamiltonians is counted in Section \ref{numHam}, based on the results of Section \ref{ns}. 
Finally, 
the independence of 
Hamiltonians is proved in Section \ref{Concl}.

The quantum version of our integrable 
system is established in Section \ref{QIS}. The key point is that 
the classical Hamiltonians are Laurent polynomials in any coordinate system related to a minimal 
bipartite graph on the torus, 
and thus admit a natural upgrade to elements of the corresponding quantum tori.  
The claim that they commute follows then immediately 
from the combinatorial proof of commutativity of the classical Hamiltonians given below. We also need the independence of 
the classical Hamiltonians, proved in Section \ref{Concl}. 

Section \ref{resistornetsection} is totaly independent: it neither  
relies on nor used in the proof of integrability. Section \ref{HBDEsection} is not 
used in the proof of integrability. It discusses more specific 
discrete cluster integrable systems, assuming Theorem \ref{completeintegrabilitytheorem}.

\vskip 3mm
{\it Proof of part i)}
Take a pair of matchings $(M_1, M_2)$ on $\Gamma$. 
Let us assign to them  another pair of matchings 
$(\widetilde M_1, \widetilde M_2)$ on $\Gamma$. 
Observe that $[M_1]-[M_2]$ is a $1$-cycle. It is a disjoint union of 

\begin{enumerate}

\item homologically trivial loops,

\item  homologically non-trivial loops,  

\item edges shared by both matchings. 
\end{enumerate} 
For every edge $E$ of each homologically trivial loop 
we switch the label of $E$: if $E$ belongs 
to a matching $M_1$ (respectively $M_2$), we declare that it will belong to a matching 
$\widetilde M_2$ (respectively $\widetilde M_1$). 
For all other edges we keep their labels intact. 
Denote by $\mu_i$ and $\widetilde \mu_i$ the homology classes of the cycles
$[M_i] - \Phi_\alpha$ and $[\widetilde M_i] - \Phi_\alpha$. 
Clearly $(\widetilde {\widetilde M_1}, 
\widetilde {\widetilde M_2}) = (M_1, M_2)$. 
We use a shorthand $\varepsilon(\ast, \ast)$ for $\varepsilon(\ast, \ast)_{\hat S_\Gamma}$.

\bl \la{4.16.10.1}
One has 
\be \la{4.16.10.3}
\varepsilon(\mu_1, \mu_2) + \varepsilon(\widetilde \mu_1, \widetilde \mu_2)=0.
\ee
\el

\begin{proof} 
We use a local calculation of the Poisson bracket on ${\mathcal L}_\Gamma$ 
given in Section \ref{sec2}. 
Let us write (\ref{4.16.10.3}) 
as a sum of the local contributions corresponding to the vertices $v$ of $\Gamma$, 
and break it in three pieces as follows: 
$$
\varepsilon(\mu_1, \mu_2) + \varepsilon(\widetilde \mu_1, \widetilde \mu_2) = 
(\sum_{v \in {\mathcal E}_1} + \sum_{v \in {\mathcal E}_2} + \sum_{v \in {\mathcal E}_3})
\Bigl(\delta_v(\mu_1, \mu_2) + \delta_v(\widetilde \mu_1, \widetilde \mu_2)\Bigr).
$$
Here 
${\mathcal E}_1$ (respectively ${\mathcal E}_2$ and ${\mathcal E}_3$) is the set of vertices which belongs to 
the homologically trivial loops (respectively homologically non-trivial loops, double edges).

Then for any $v \in {\mathcal E}_1$ we have 
$\delta_v(\mu_1, \mu_2) + \delta_v(\widetilde \mu_1, \widetilde \mu_2)=0$ since the elements 
$l_v(\mu_1)\in {\mathbb A}_v$ provided by $\mu_i$ by the very construction coincides with the element 
 $l_v(\widetilde \mu_2)$, and similarly $l_v(\mu_2) = l_v(\widetilde \mu_1)$. 
So the first sum is zero since it is a sum of zeros.

The third sum is also a sum of zeros. Indeed, the contribution of the matchings $M_1$ and $M_2$ 
at every vertex $v \in {\mathcal E}_3$ is the same, so the local pairing vanishes by the skew symmetry. 

It remains to prove that 
$$
\sum_{v \in {\mathcal E}_2} 
\Bigl(\delta_v(\mu_1, \mu_2) + \delta_v(\widetilde \mu_1, \widetilde \mu_2)\Bigr)=0.
$$

Let $\gamma$ be an oriented loop on $\Gamma$. We define the {\it bending} 
$b_v(\gamma; \varphi)$ of the function $\varphi$ at a vertex 
$v$ of $\gamma$ as 
$$
b_v(\gamma; \varphi)=\sum_{E \in R_v}\varphi(E) - \sum_{E \in L_v}\varphi(E)\in \R. 
$$
Here $R_v$ (respectively $L_v$) 
is the set of all edges sharing the vertex $v$ which are on the right
(respectively left) of the path $\gamma$, 
following the direction of the path $\gamma$. Lemma \ref{4.16.10.1} now follows
from Lemma \ref{Euclidean} below. 
\end{proof}

Recall the set $X$ of circular-order-preserving maps from the zig-zag paths to $\R/\Z$.
 
\bl\la{Euclidean} There is an  $\alpha\in X$ 
such that the corresponding function $\varphi_\alpha$ 
satisfies, for any simple topologically nontrivial loop $\gamma$,
$$
\sum_{v \in \gamma}b_v(\gamma; \varphi_\alpha)=0. 
$$
\el
 
\begin{proof}
On a Euclidean torus, a simple closed polygonal curve which is not null-homotopic
has the property that its net curvature is zero,
where the net curvature is the total amount of left turning minus right turning.
This fact can be seen by taking an isotopy of the curve to a straight geodesic
through a family of simple polygonal curves; the total curvature is an
integer multiple of
$2\pi$ but changes continuously under the isotopy, and for a geodesic it is zero.

A map $\alpha\in X$ allows us to draw our graph $\Gamma$ on a Euclidean torus
in such a way that $b_v(\gamma,\varphi_\alpha)$ is (with a small proviso) the net curvature
of $\gamma$ at vertex $v$.
This is accomplished as follows. To each edge $E$, recall that
$\alpha_r,\alpha_l$ are the angles of the strands crossing $E$. We define
a $1$-cochain $f$ by $f(E) = e^{i\alpha_l}+e^{i\alpha_r}$.
We have $df=0$, since strands cross two edges of every face,
once with each orientation.
Thus there is a multivalued function $F$ on vertices of $\Gamma$,
which is locally well defined but has additive periods $\omega_1,\omega_2$
for the generators of $H_1(S)$.
Extending $F$ linearly over the edges gives a map from
$\Gamma$ to a Euclidean torus
$\C/\langle\omega_1,\omega_2
\rangle$.
Locally, to each edge we are associating a Euclidean rhombus
with vertices $z,z+e^{i\alpha_l},z+e^{i\alpha_l}+e^{i\alpha_r}, z+e^{i\alpha_r}$,
where $z$ is determined by $F$; the edge $E$ runs from $z$
to $z+e^{i\alpha_l}+e^{i\alpha_r}$, and the vertices
$z+e^{i\alpha_l},z+e^{i\alpha_r}$ are the centers of the faces adjacent to $E$.
 These rhombi are glued edge to edge to form a  Euclidean torus (possibly with conical singularities).

If for each $E$ the angle $\varphi(E)=\alpha_r-\alpha_l$ is less than $\pi$, then the
rhombi are
positively oriented and, by Lemma \ref{monotone1},
the total angle at each vertex is $2\pi$ and our Euclidean surface is smooth, that is, has no singularities.
If a rhombus for edge $E$, connecting vertices $v,v'$,
has angle larger than $\pi$, by Lemma \ref{monotone1} it will be the only rhombus
at $v$ or $v'$ with this property. Hence the singularities
in our Euclidean structure are folds, as shown in Figure \ref{fold}, and
are separated from each other.

It remains to check that if $\gamma$ contains a vertex $v$ which has a singularity
then it is still true that $b_v(\gamma,\varphi)$ measures the net curvature
of $\gamma$ at $v$. If $\gamma$ passes through $v$ but not along edge $E$
this is trivially true. In the remaining case, $\gamma$ contains edge $E=vv'$:
suppose $\gamma$ contains
consecutive vertices $v_1,v,v',v_2$ in order. Then it is straightforward to check that
the sum of the
curvature contributions at $v$ and $v'$ equals the change in Euclidean angle
of the edge $v_1v$ to the edge $v'v_2$.
In all cases therefore $b_v$ is the net curvature of $\gamma$, which is zero.
\end{proof}

\begin{figure}[ht]
\epsfxsize80pt
\centerline{\epsfbox{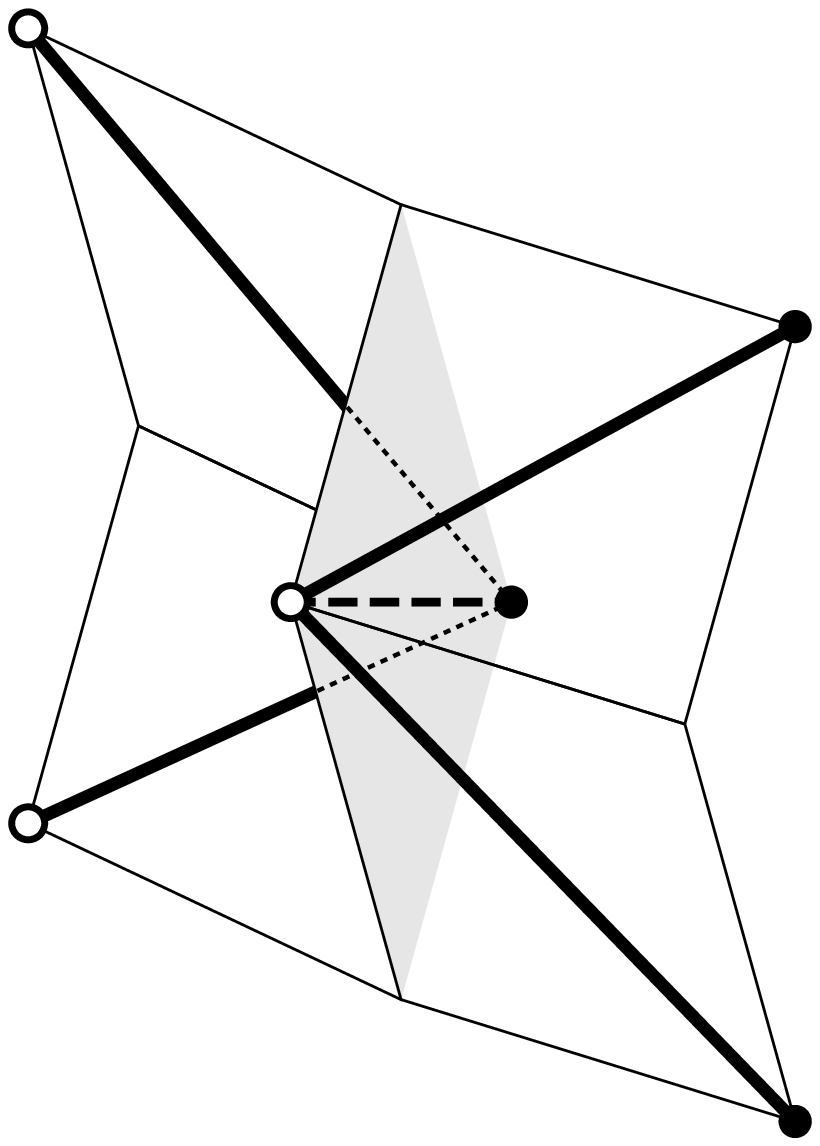}}
\caption{\label{fold}}
\end{figure}

The part i) of Theorem \ref{completeintegrabilitytheorem} follows immediately from Lemma \ref{4.16.10.1}. 
Indeed, the sum calculating the Poisson bracket of two Hamiltonians 
can be presented as a sum of contributions (\ref{4.16.10.3}). 
The part i) of Theorem \ref{completeintegrabilitytheorem} is proved.
\hfill{$\Box$}

\subsection{The matching polytope}

Let $\Omega$ be the {\bf matching polytope} of $\Gamma$. By definition
this is the subset $$\Omega\subset[0,1]^{\{\text{edges of  $\Gamma$}\}}$$
consisting of functions $\omega$ on edges of $\Gamma$ with values in $[0,1]$
such that the sum of the values at every vertex is $1$.  The subset $\Omega$ is a convex
polyhedron.

\begin{lemma}[\cite{LP}] \la{4.18.10.11} The vertices of $\Omega$ are the perfect
matchings of $\Gamma$.
\end{lemma}

\begin{proof} To make the paper self-contained, we give a proof. 
 Let $\omega$ be a vertex of $\Omega$. Let us show that 
for every vertex $v$ of $\Gamma$ there is just one edge $E$ incident to $v$ 
with $\omega(E) \not = 0$. Indeed, otherwise there is a vertex of $v_1$ incident to edges $E_1, E_2$
with the weights in $(0,1)$. Let $v_2$ be another vertex of $E_2$. 
It is incident to another edge $E_3$ with the weight in $(0,1)$ -- otherwise the sum of the weights at 
$v_2$ will be less then $1$. Let $v_3$ be another edge of $E_3$, and so on.
We get  a closed path on $\Gamma$ whose edges have the weights in $(0,1)$. Let us add 
$\varepsilon$ to the weight of each of these edges, where $|\varepsilon|$ is very small, and the signs alternate. The  sign condition
does not lead to a problem since $\Gamma$ is a bipartite graph.
Then we still get an element $\omega_\varepsilon \in
\Omega$. So $\omega$ is inside of a 
segment in $\Omega$.  
Thus it is not a vertex 
of $\Omega$.

It follows that the values of $\omega$ at the edges are either $1$ or $0$. 
The edges where the value is $1$ give a matching. 
\end{proof}

Elements of $\Omega$ are called {\bf fractional matchings}.

\begin{lemma} \la{4.18.10.11001} Let $\Gamma$ be a minimal bipartite graph on a torus.
Then $\Gamma$ has a perfect matching.
\end{lemma}

\begin{proof} It suffices to show that $\Omega$ is nonempty.
Indeed,  by Lemma \ref{4.18.10.11} any vertex of the polyhedron $\Omega$ provides
a perfect matching. Any minimal bipartite graph $\Gamma$ is obtained from
some polygon $N$. The function $\varphi_\alpha$ from Theorem \ref{monotone1}
is then a fractional matching by (\ref{4.17.10.111}). 
\end{proof}

Let $\omega_0\in\Omega$ be a fixed fractional matching.
Then for any $\omega\in\Omega$ the difference
$[\omega]-[\omega_0]$ is a $1$-cycle on $\Gamma$. Taking its homology
class we get a map $\Omega\to H_1(\Gamma,\R)$. Its image is a convex polyhedron
$[\Omega]\subset H_1(\Gamma,\R)$. Its projection under the
canonical map $\pi:H_1(\Gamma,\R)\to H_1(\T,\R)$ is a convex polygon
$$\pi[\Omega]\subset H_1(\T,\R)\cong \R^2,$$
which depends on $\omega_0$ only through a translation.

Recall the space $X$ of order-preserving maps from $\{e_1,\dots,e_n\}$ to
$\R/\Z$.
The image of $X$ under the map $\alpha\to \varphi(\alpha)$ is a subpolyhedron $\Omega_0
\subset\Omega$. It is in fact a simplex: the simplex of order-preserving maps from
the strands to $\R/\Z$, modulo rotation.

Let $\alpha_i\in X$ be the map which sends all $E_i$ to $0$ and
assigns $1$ to the arc between $E_i$ and $E_{i+1}$
and $0$ to all other arcs. The matching $\varphi(\alpha_i)$
maps to a vertex of $\Omega_0,$ and every vertex of $\Omega_0$ is of this type.

\subsection{The Newton polygon}

\begin{theorem} \la{3.21.10.2} Let $\Gamma$ be a minimal bipartite graph with
polygon $N$. Then the Newton polygon of the partition function $P_\alpha(z_1,z_2)$
for $\Gamma$ is $N$ (up to translation).
\end{theorem}

\begin{proof}
Set $\omega_0=\Phi_\alpha$.
The Newton polygon of $P_\alpha(z_1,z_2)$ is a translate of $\pi[\Omega]$.
Indeed, by definition the coefficient in $P$ of $z_1^{i_1}z_2^{i_2}$ is the
sum of weights of dimer covers $M$ of $\Gamma$ for which the homology
class $[M]-\Phi_\alpha$ satisfies
$[M]-\Phi_\alpha=i_1[z_1]+i_2[z_2]$.

We can define the intersection pairing of a chain with a zig-zag path $Z$ as follows. 
If $C=\sum_EC_E[E]$ is a chain, the intersection
pairing of $C$ with $Z$ is $\langle C,[Z]\rangle = \sum_{E\in Z}\pm C_E$
where the sign is alternating, and depends to whether the orientation of the edge of $Z$ 
agrees with the white-to-black orientation or disagrees with it. In other words,
this is the intersection pairing of $C$ with the strand corresponding to $Z$.

We define another polygon $N_\Gamma\subset H_1(\T,\R)$ by a set of inequalities
\begin{equation}\label{inequalities}
N_\Gamma := 
\left\{\gamma\in H_1(\T,\R)~|~
\langle \gamma,[Z]\rangle\le \frac{l(Z)}2-\langle[\omega_0],[Z]\rangle
\mbox{ \rm for every zig-zag path } Z\right\}.
\end{equation}
We claim that $N$, $N_\Gamma$ and $\pi[\Omega]$ are translates of each other.

We first show that $\pi[\Omega]\subset N_\Gamma$. Let $\omega\in\Omega$.
The sum of $\omega(E)$ over the edges $E$ in a zig-zag path $Z$ is at most equal
to $l(Z)/2$, since this is the number of white vertices adjacent to $Z$
(each white vertex adjacent to $Z$ has total flow at most $1$ across $Z$).
Thus
$$
\langle [\omega-\omega_0],[Z]\rangle =\langle [\omega],[Z]\rangle-\langle
[\omega_0],[Z]\rangle
\le \frac{l(Z)}2-\langle[\omega_0],[Z]\rangle
$$
which implies that $\pi[\Omega]\subset N_\Gamma$.

We now show that the edges of $N$ and $N_\Gamma$ coincide.
Take $\gamma_i=[\omega(\alpha_i)-\omega_0].$ This satisfies
the inequalities (\ref{inequalities}) and two of these are
equalities: those for $Z_i$ and $Z_{i+1}$, because in $\alpha_i$ every white vertex
adjacent to $Z_i$ is matched to a black vertex on the other side
of $Z_{i}$, and similarly for $Z_{i+1}$.
Note moreover that  $\gamma_{i+1} = \gamma_i+[Z_i]$
so we see that $[Z_i]=e_i$ is an edge of
$N_\Gamma$. Thus $N_\Gamma$ has sides consisting of the $e_i$ in order,
which proves that $N$ and $N_\Gamma$ are translates.
Note moreover that each corner of $N_\Gamma$ is in the image of $\pi[\Omega]$
(in fact in the image of $\pi[\Omega_0]$). Thus $\pi[\Omega]=N_\Gamma$.
\end{proof}

\subsection{Counting triple crossings and faces }\la{ns}

From Thurston \cite{Thurston},
the number of triple crossings in any triple crossing diagram in a disk
is a function solely of the strand order along the boundary,
and can be obtained as follows. Take a linear order of the strand endpoints
which is consistent with the circular order. Then we say two strands $ac$ and $bd$,
directed from $a$ to $c$ and $b$ to $d$ respectively, have a 
{\it parallel crossing} if their endpoints are in the order $a<b<c<d$ or $d<c<b<a$.
\begin{lemma}[\cite{Thurston}]
In a minimal triple-crossing diagram in a disk, 
the number of triple crossings is the number of pairs of strands with a parallel crossing.
\end{lemma}

Now take a triple crossing diagram on the unit square $[0,1]^2$ obtained
from a convex integral polygon $N$ as in Section \ref{sec3}.
Let $e_1,\dots,e_n\in\Z^2$ be the primitive sides of $N$.

\begin{lemma} \la{ntc}
The number of triple crossings is twice the area of $N$.
\end{lemma}

\begin{proof}
A minimal triple crossing diagram
can be obtained by starting with a configuration of straight geodesics
on the torus $[0,1]^2$, one for each $e_i$, and then isotoping them into position.

We count the number of parallel crossings in a standard configuration
of the geodesics $e_i$ and then notice that this number is invariant when
we isotope the strands to their final position.

Let $\tau\in\R^2$ be a vector with positive coordinates having irrational ratio, and 
let $v_{min},v_{max}$ be respectively the 
unique vertices of
$N$ which minimize, respectively maximize $v\cdot \tau$. 
Let $N_R$ be the set of sides of $N$ between $v_{min}$ and $v_{max}$ in counterclockwise order
and $N_L$ be the set of sides of $N$ between $v_{max}$ and $v_{min}$ in counterclockwise order.
Then after a circular re-indexing, say that 
$e_1,\dots,e_k\in N_R$ and $e_{k+1},\dots,e_n\in N_L$.

Let $p=\varepsilon\tau\in[0,1)^2$ for some sufficiently small $\varepsilon$. Start by drawing all geodesics
$e_i$ in the torus through $p$. No geodesic contains the origin by irrationality of $\tau$; 
order the boundary points on the square counterclockwise starting from the origin. 
Then $e_i$ and $e_j$ have a parallel crossing at $p$ if and only if $e_i$ and $e_j$
are either both in $N_R$ or both in $N_L$. Moreover if $p$ is sufficiently
close to the origin all other crossings of
$e_i$ and $e_j$ will be parallel or non-parallel according to the crossing at $p$. 

Thus the number of parallel crossing strands in total is 
$$\sum_{1\le i<j\le k} |e_i\wedge e_j| + \sum_{k+1\le i<j\le n} |e_i\wedge e_j|.$$
However the area of the part of $N$ below the line through $v_{min}$ and $v_{max}$
is $1/2$ of
$$e_1\wedge e_2+(e_1+e_2)\wedge e_3+\dots+(e_1+\dots+e_{k-1})\wedge e_k=
\sum_{1\le i<j\le k} |e_i\wedge e_j|,$$
as can be seen by triangulating this polygon using extra edges from
$v_{min}$ to each 
other $v$. Similarly for the upper part of $N$.

Now as we isotope the geodesics $e_i$, without letting a strand cross
the origin, it is 
easy to see that the number of parallel
crossings does not change.

\end{proof}

Let us denote by $B$ and $W$ the number of black and white faces 
for a triple point diagram graph $G$ on the torus arising  from a Newton
polygon $N$. Let $V$ be the number of vertices of $G$, that is the
number of the  triple crossings.

\bp \la{IP}
One has 
$$
B = W = V= \mbox{\rm number of the triple crossings} 
= 2(\mbox{\rm area of  $N$}).
$$
\ep

\begin{proof} Euler's formula 
for the triple point diagram $G$ 
gives $V-E+F=0$. Counting flags (vertex, edge) we get $6V=2E$. 
Therefore $F = 2V$.

Lemma \ref{4.18.10.11001} implies that $B=W$. Indeed, the white (respectively black) domains 
of a triple point diagram $G$  
became the white (respectively black) vertices of the bipartite graph $\Gamma$. Thus
$F=B+W$ implies $B=W=V$. Lemma \ref{ntc} implies the last statement. 

\end{proof}

\subsection{The number of Hamiltonians} \la{numHam}
Let $i(N)$ and $e(N)$ be the number of interior and exterior points 
for the Newton polygon $N$. 
Pick's formula plus Proposition \ref{IP} gives
\be \la{4.18..10.3}
2i(N) + e(N) -2 = 2(\mbox{\rm area of  $N$}) = B.
\ee
Observe that $i(N)$ is the number of Hamiltonians,  and $e(N)$ is the
number of Casimirs plus one, 
that is the dimension of center of the Poisson algebra 
${\mathcal O}({\L}_\Gamma)$ minus one. On the other hand, 
$B$ is the number of faces minus one, and ${\rm rk}H_1(S, \Z)=2$. So 
$B+1 = {\rm dim}{\L}_\Gamma$. So (\ref{4.18..10.3}) implies that  
$$
2(\mbox{number of Hamiltonians}) = {\rm dim}{\L}_\Gamma - 
{\rm dim}\Bigl({\rm Center}~{\mathcal O}({\L}_\Gamma)\Bigr).
$$

\section{Dimers, cluster Poisson transformations, and quantum integrability}
\label{spidermovesection}

\subsection{Dimer models on surfaces and cluster Poisson varieties.} 

\subsubsection{Gluing the conjugated surfaces according to a spider move}
Given a bipartite ribbon graph $\Gamma$, let $\widehat S_\Gamma$ be 
the corresponding 
conjugated surface. 
Since the conjugated graph $\widehat \Gamma$ coincides with $\Gamma$ as a graph, 
there are canonical isomorphisms:
\be \la{IL}
H_1(\widehat S_{\Gamma}, \Z) = H_1(\widehat \Gamma, \Z) = H_1(\Gamma, \Z). 
\ee
We denote by $\Lambda_\Gamma$ the lattice (\ref{IL}). It is equipped with 
the skew symmetric integral bilinear form $(\ast, \ast)_{\widehat \Gamma}$ 
given by the intersection pairing on 
$\widehat S_\Gamma$. The  boundaries $\partial F$  of the faces 
$F$ of $\Gamma$, whose orientations 
are induced  by the orientation of $S$,  give rise to a collection of 
cycles $\{\gamma_F\}$ on $\widehat S_\Gamma$. These cycles plus the generators $\alpha_1, ..., \alpha_{2g}$ of $H_1(S, \Z)$ generate $\Lambda_\Gamma$, and satisfy the only relation 
$\sum_F\gamma_F=0$.

A {\it spider move} is a transformation of a bipartite graph shown in dashed lines
on Figure \ref{di30}. We show on the same figure the transformation of the
accompanying 
zig-zag paths (solid), and the unique cyclic order $(1,3,2,4)$ of the zig-zag paths
compatible with the cyclic orders of the triples of the zig-zag paths
at the black vertices. The white vertices give rise to the
opposite cyclic order. These cyclic orders are the same for both graphs.

\begin{figure}[ht]
\centerline{\epsfbox{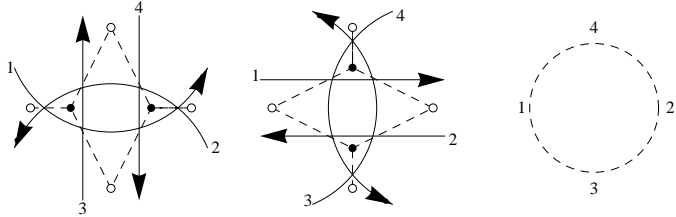}}
\caption{A spider move and zig-zag paths.\label{di30}}
\end{figure} 
Let $F$ be the  rectangular face on  $\Gamma$ which 
determines a spider move $s: \Gamma \lra \Gamma'$. 
Denote by $F_1, F_2, F_3, F_4$ 
the faces of $\Gamma$ sharing edges with $F$. 
We assume for simplicity of the exposition that these are different faces. 
The spider move does not affect 
the graph $\Gamma$ outside of a little domain $U$ containing the face $F$ and the 
two extra edges emanating from the two black vertices of $F$. 
We identify the conjugated surfaces $\widehat S_{\Gamma}$ and 
$\widehat S_{\Gamma'}$ outside of the $U$. 
We want to extend this identification to a homeomorphism of surfaces
\be \la{HOFS}
\widehat s: \widehat S_{\Gamma} \lra \widehat S_{\Gamma'}. 
\ee
\bl \la{UDHS11}
\la{UDHS}
There is a unique up to isotopy homeomorphism (\ref{HOFS}) 
identifying the boundary loops and identical outside of $U$-domains.
Let $\gamma_G':= \widehat s(\gamma_G)$. Then one has 
\be \la{HOFS1}
\gamma_F'=-\gamma_F, \qquad \gamma_{F_1}'=\gamma_{F_1} + \gamma_F, 
\quad \gamma_{F_2}'=\gamma_{F_2}, \quad 
\gamma_{F_3}'=\gamma_{F_3}+\gamma_F, \quad \gamma_{F_4}'=\gamma_{F_4}.
\ee
Otherwise $\gamma'_G=\gamma_G$. 
\el

\begin{figure}[ht]
\epsfxsize340pt
\centerline{\epsfbox{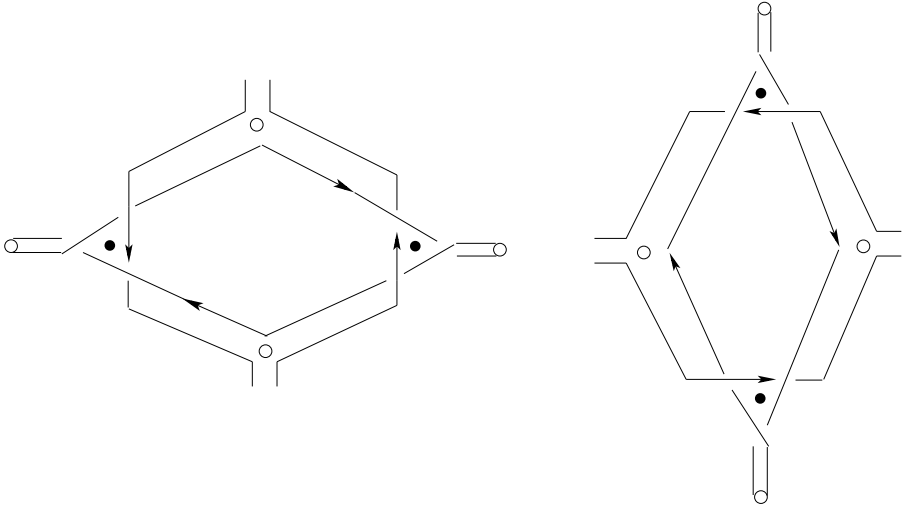}}
\caption {Identifying the conjugated surfaces according to a spider move. 
Recall our convention that a zig-zag path goes 
counterclockwise around black vertices and clockwise around the white ones.
When we turn the surface over, the notion of counterclockwise and clockwise
are interchanged. Since to get $\hat\Gamma$ we turned the white vertices
over, the new orientations is counterclockwise for all vertices. (In the figure, we are seeing
the underside of the surface near the black vertices.)\label{di72}}
\end{figure}
\begin{figure}[ht]
\epsfxsize430pt
\centerline{\epsfbox{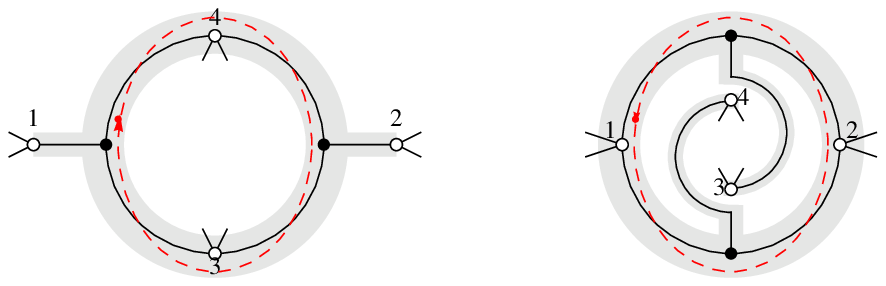}}
\caption{Identifying the conjugated surfaces according to a spider move, 
shown in the untwisted version.\label{untwisted}}
\end{figure}

\begin{proof} Figure \ref{di72} (and Figure \ref{untwisted} in the
untwisted version) shows the ribbon   
graphs $\widehat \Gamma$ and $\widehat \Gamma'$ 
inside of the $U$-domains. Each component of the oriented boundary of a ribbon
 is a part of an oriented zig-zag path on $\Gamma$. 
Any map (\ref{HOFS}) preserving the boundaries 
and identical outside of $U$-domains identifies 
the zig-zag paths on the original graphs.  
This determines the gluing uniquely up to an isotopy. 
So we get the first claim of the Lemma. 

On each surface the four zig-zag paths inside of $U$ bound an annulus. 
A pair of parallel zig-zag paths together with a pair of 
connecting them little segnents form a boundary component of the annulus. 
Our gluing process amounts to gluing these two annuli.

Taking a loop going around the face on the left 
of Figure \ref{di72}, and tracing its image
on the right surface, we get $\gamma_F'=-\gamma_F$. Similarly we get 
the other identities. 
\end{proof}

\subsubsection{Seeds, mutations, conjugated surfaces and spider moves}

\begin{definition} \la{D1}
A {\it seed}  ${\bf s}$ is a datum 
$
\Bigl(\Lambda, \{e_i\}, (\ast, \ast)\Bigr), 
$ where 
\begin{itemize}
\item $\Lambda$ is a lattice, i.e. a free abelian group, $\{e_i\}$ is a 
collection of non-zero vectors in $\Lambda$;
\item $(\ast, \ast)$ a skew-symmetric $\Z$-valued 
bilinear form on $\Lambda$. 
\end{itemize} 
\end{definition}

\begin{definition}\la{D2}
A mutation of a seed ${\bf s}$ 
in the direction of a basis vector $e_k$ is a new seed $\widetilde {\bf s}$.
 The lattice and the form of the seed $\widetilde {\bf s}$ 
are the same as of ${\bf s}$. 
The collection of vectors $\{\widetilde e_i\}$ of $\widetilde {\bf s}$ is defined by 
applying a half reflection at $e_k$ map to the old one:
\begin{equation} \label{12.12.04.2a}
\widetilde e_i := 
\left\{ \begin{array}{lll} e_i + (e_{i}, e_k)_+e_k
& \mbox{ if } &  i\not = k\\
-e_k& \mbox{ if } &  i = k.\end{array}\right. \qquad a_+:={\rm max}(a, 0).
\end{equation}
\end{definition}
Definitions \ref{D1} and \ref{D2} generalize slightly the 
standard definitions of seeds and seed mutations, see 
Sections 1.2.1-1.2.4 of \cite{FG1}. Unlike {\it loc. cit}, we do not demand neither that the vectors 
$e_i$ are linearly independent, nor 
that they span the lattice $\Lambda$.\footnote{We simplify the 
discussion by considering only the
``simply-laced'' case, with the ``multipliers''
$d_i=1$.} 

Here is our main example. 
\bd
The seed ${\bf s}_\Gamma$ 
assigned to a bipartite graph $\Gamma$ on a genus $g$ surface $S$ 
is  
$$
{\bf s}_\Gamma:= (\Lambda_\Gamma, (\ast, \ast)_{\widehat \Gamma}, \{\gamma_F\}).
$$
\ed

\bl The spider move $\Gamma \lra \Gamma'$ centered at a face $F$  
leads to a mutation of seeds $
{\bf s}_\Gamma \lra {\bf s}_{\Gamma'}
$ 
in the direction of the vector $\gamma_F$.
\el

\begin{proof} 
The only non-zero values $(\gamma_F, \gamma_{G})_{\widehat \Gamma}$ are 
\be \la{NZV}
(\gamma_F, \gamma_{F_1})_{\widehat \Gamma} = 1, \quad (\gamma_F, \gamma_{F_2})_{\widehat \Gamma} = -1, \quad (\gamma_F, \gamma_{F_3})_{\widehat \Gamma} = 1, 
\quad (\gamma_F, \gamma_{F_4})_{\widehat \Gamma} = -1
\ee
Thus formalae (\ref{HOFS1}) from Lemma \ref{UDHS11} 
are just equivalent to the mutation formulae 
(\ref{12.12.04.2a}).
\end{proof}

Therefore the geometric objects --  the conjugated surfaces and the 
spider moves -- are encoded, respectively, 
by the corresponding seeds and their mutations. 

\subsubsection{Quantum cluster transformations \cite{FG1}}\la{sec4.1.3}
A lattice $\Lambda$ with  a skew-symmetric $\Z$-valued bilinear form 
$(\ast, \ast)$ gives rise to a quantum torus $\ast$-algebra 
${{\bf T}}_{\Lambda}$. 
The algebra ${{\bf T}}_{\Lambda}$ has a basis $\{X_v\}$ over the ring $\Z[q, q^{-1}]$ parametrized by the 
vectors  $v$ of the lattice $\Lambda$. The multiplication is given by 
$$
q^{-(v_1, v_2)}X_{v_1} X_{v_2} = X_{v_1+v_2}.
$$
There is an involutive antiautomorphism 
$$
\ast: {{\bf T}}_{\Lambda}\lra {{\bf T}}_{\Lambda}, \qquad  
\ast(X_{v}) = X_{v}, ~
\ast(q) = q^{-1}.
$$
Let us choose an order of the basis $\{v_i\}$ of 
$\Lambda$. Then 
\begin{equation} \label{4.28.03.11xewr}
X_{v} = q^{-\sum_{i<j}a_{i}a_{j}(v_i, v_j)}\prod_{i=1}^nX_i^{a_i}, \qquad v = \sum_{i=1}^n a_iv_i, 
\end{equation}

Consider the following 
formal power series, called the 
{\it $q$-exponential}:
 \be \la{psi}
{\bf \Psi}_q(x):= \prod_{a=1}^{\infty}(1+q^{2a-1}x)^{-1} = \frac{1}{(1+qx)(1+q^3x)(1+q^5x)\ldots}.
\ee
It is characterized up to a constant by the 
difference relation
\begin{equation} \label{11.19.06.20}
{\bf \Psi}_q(q^2x) = (1+qx){\bf \Psi}_q(x). 
\end{equation} 

Denote by ${\T}_{\Lambda}$ the non-commutative 
fraction field of the algebra ${\bf T}_{\Lambda}$. 
The {\it quantum mutation map} $\mu_{e_k}^q$ 
in the direction of a vector $e_k$ is an automorphism of the skew field ${\T}_{\Lambda}$ 
given by the conjugation by 
${\bf \Psi}_{q}(X_{e_k})$: 
\be \la{f4}
\mu^q_k:={\rm Ad}_{{\bf \Psi}_{q}(X_{e_k})}: {\T}_{\Lambda} \lra {\T}_{\Lambda}.
\ee
Although  ${\bf \Psi}_{q}(X_{e_k})$ is not a rational function, this 
is a rational map.

\subsubsection{Cluster Poisson varieties \cite{FG1}} 
Cluster Poisson varieties, also known as 
cluster ${\mathcal X}$-varieties \cite{FG1}, 
are  geometric objects which are in duality  
 to  cluster algebras of Fomin-Zelevinsky \cite{FZI}.

We assign to a seed ${\bf s}$ 
 a complex algebraic torus 
${\mathcal X}_{\Lambda}:= {\rm Hom}(\Lambda, \C^*)$, called the {\it seed torus}. 
A vector $v \in \Lambda$ gives rise to a function $X_v$ 
on the torus obtained by evaluation at $v$.  
A mutation in the direction of a vector $e_k$ 
gives rise to a birational isomorphism 
$\mu_{e_k}: {\mathcal X}_{\Lambda} \lra {\mathcal X}_{\Lambda}$, 
obtained by specilalising $q=1$ in (\ref{f4}). It 
acts on the function  
$X_v$ 
by 
\begin{equation} \label{f3a}
\mu_{e_k}^*: X_{v} \lms 
    X_v(1+X_{e_k})^{-(e_i, e_k)}. 
\end{equation} 
The Poisson bracket is  the quasiclassical limit of the commutator:
$$
\{X_{v_1}, X_{v_2}\}:= \lim_{q\to 1}[X_{v_1}, X_{v_2}]/2(q-1) = (v_1,v_2 ) X_{v_1}X_{v_2}.
$$ 
Thus (\ref{f3a}) is a Poisson map. 
An isomorphism of seeds leads to a Poisson
 isomorphism of the seed tori. Compositions of seed mutations and seed isomorphisms are 
 {\it seed cluster transformations}. 
So a seed cluster transformation gives rise to a Poisson birational 
isomorphism of the seed tori. A seed cluster transformation ${\bf s} \to {\bf s}$ 
 acting trivially on the functions $X_v$ 
is  a {\it trivial seed cluster transformation}. 
Consider all seeds 
obtained from a given seed ${\bf s}$ by a sequence of mutations. 
We assign to these seeds  their  seed torus, and glue pairs of tori 
related by mutations into a 
space $\widetilde {\cal X}$ by using 
birational isomorphisms (\ref{f3a}). Trivial seed cluster transformations 
give rise to automorphisms of $\widetilde {\cal X}$. Taking the quotient 
 of $\widetilde {\cal X}$ by these automorphisms, we get a 
cluster Poisson variety ${\cal X}$.\footnote{It 
could be a non-separated space rather then a variety.} 

Non-trivial cluster seed transformations ${\bf s} \to {\bf s}$ 
lead to automorphisms of the ${\cal X}$, which form the 
{\it cluster modular group}. We will see such examples in Section 6.

\vskip 3mm
A cluster Poisson variety is nothing else but a collection of 
seed tori ${\mathcal X}_{\Lambda}$ 
glued via the cluster transformations. 
The algebra of regular functions ${\mathcal O}({\mathcal X}_{\Lambda})$ 
on a cluster torus is an algebra of Laurent polynomials. So 
a cluster Poisson variety is described by a collection of algebras of Laurent polynomials and birational automorphisms (\ref{f3a})  
providing the gluing transformations. 
The algebra ${\mathcal O}({\mathcal X})$ of regular functions on 
a cluster Poisson variety ${\mathcal X}$ can be identified with the 
subalgebra of the algebra  ${\mathcal O}({\mathcal X}_{\Lambda})$ consisting of the 
elements which remains Laurent polynomials after any cluster transformation.

This suggests that one should think about the quantum torus algebra ${\bf T}_{\Lambda}$  
as of the algebra of regular functions 
on a non-commutative quantum torus, and that these algebras together with the mutation maps 
(\ref{f4}) should be thought of as a quantum cluster variety. 
We define the algebra ${\mathcal O}_q({\mathcal X})$ of regular functions 
on a quantum cluster variety as the subalgebra of the quantum torus algebra ${\bf T}_{\Lambda}$ 
consisting of all elements which remain Laurent polynomials after any quantum cluster transformation.

\subsubsection{Gluing the moduli spaces ${\cal L}_\Gamma$ according to a spider move} Instead of using the 
set of all seeds related to an initial seed ${\bf s}$ 
by seed cluster transformations, we can 
pick any subset of this set, and define a Poisson space 
by gluing the corresponding cluster tori. 
The cluster Poisson spaces which we need in our paper 
are defined this way: we use only finitely many seeds corresponding to minimal bipartite graphs 
associated with a given Newton polygon.

\subsubsection{Traditional cluster Poisson variety related to the dimer model}
\la{clusterpoissonsection}
Traditional cluster Poisson varieties are the ones 
defined by using seeds where the vectors $\{e_i\}$ form a basis of the lattice. Let 
$$
{\rm T}_S:= H_1(S, \C^*)\stackrel{\sim}{=}(\C^*)^{2g}. 
$$
It is  the group of line bundles with flat 
connections 
on  $S$. The embedding $\Gamma \hookrightarrow S$ provides a free  Poisson action 
of the group ${\rm T}_S$ on the moduli space ${\mathcal L}_\Gamma$. 

The monodromy of a line bundle on $\Gamma$ around a face $F$  is a function 
$W_F$ on  ${\L}_\Gamma$. 
The product of all $W_F$'s  is  $1$. Define
$$
{\mathcal X}'_\Gamma = (\C^*)^{\{\mbox{faces of $\Gamma$}\}}. 
$$
Let ${\mathbb W}: {\mathcal X}'_\Gamma \to \C^*$ be the product of all coordinates. 
Assigning to a  line bundle on $\Gamma$ its monodromies around the faces   
we get a canonical isomorphism  
$$
({\mathcal L}_\Gamma/{\rm T}_S) \stackrel{\sim}{\lra}
   {\rm Ker}\Bigl({\mathcal X}'_\Gamma\stackrel{{\mathbb W}}{\lra} \C^*\Bigr). 
$$
The birational Poisson transformations (\ref{gm}) 
are 
${\rm T}_S$-equivariant. So to describe them on the quotient is 
just the same as to describe their action on the face weights.

\bl \la{1,1} The algebraic tori $\{{\mathcal X}'_\Gamma\}$ 
parametrized by   
equivalent minimal bipartite graphs $\Gamma$  
on  ${S}$ are glued  by cluster Poisson transformations   (\ref{gm})
 into a cluster Poisson variety  
${\mathcal X}'_N$. 
The maps (\ref{gm}) preserve the function ${\mathbb W}$. 
So there is a hypersurface ${\mathcal X}^0_N \subset {\mathcal X}'_N$ 
given  by  the equation ${\mathbb W} = 1$. 
The group ${\rm T}_S$ acts freely on   ${\mathcal X}_N$. There is a Poisson isomorphism
\be \la{MTTHH}
{\mathcal X}_N/{\rm T}_S \stackrel{\sim}{\lra}
   {\mathcal X}^0_N. 
\ee
\el

\begin{proof} We just need to show that 
the product of the face weights is preserved by the spider move 
cluster transformations. This 
 is clear from formulas (\ref{mll}) below.
\end{proof}

\subsection{Spider moves and partition functions}\la{sec4.2q}
We have seen that elementary transformations $\Gamma \to \Gamma'$ 
give rise to rational maps between the moduli spaces  
${\cal L}_\Gamma \to {\cal L}_{\Gamma'}$. In this Section we show that they
 do not change the partition functions for 
the dimer models 
related to the graphs $\Gamma$ and $\Gamma'$.  
Moreover, we show that the rational maps are characterised 
by the condition that they preserve the partition functions.

\bt \la{mutations}
Given a spider move, there is a unique rational transformation of the face
weights 
preserving the modified 
partition function ${\mathcal P}_\alpha$. This transformation is a cluster Poisson transformation, see Figure \ref{di32}:
\begin{eqnarray}\nonumber\la{TF} 
W'&=& W^{-1},\\ 
W_i' &=& W_i(1+W), ~~~ i=1,3,\la{mll}\\\nonumber
W_j' &=& W_j(1+W^{-1})^{-1},~~~
j=2,4. 
\end{eqnarray}
It preserves the product of the face weights. 
\et
\begin{figure}[ht]
\centerline{\epsfbox{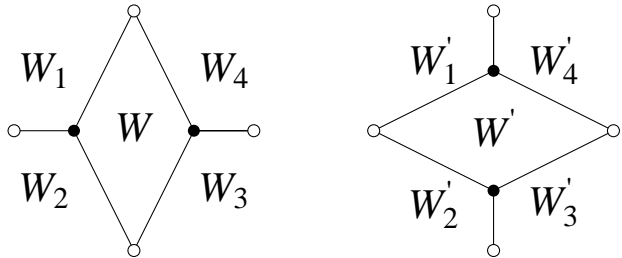}}
\caption{\label{di32}A spider move and a mutation of the face weights.}
\end{figure} 

\begin{proof} Let us show first that formula 
(\ref{TF}) is a cluster Poisson transformation. Precisely: 

\bl
Formula 
(\ref{TF}) describes  
the cluster Poisson mutation centered at $\gamma_F$, where 
$F$ is the center of the spider move. 
\el

\begin{proof} 
Let us compare formula (\ref{TF}) with the formula for the cluster transformation at $\gamma_F$. For this purpose, 
let us write the action 
of the mutation at the vector $e_k$ 
 on the functions $X_{e_i'}$ in terms of the original functions 
$X_{e_i}$:
\begin{equation} \label{f3}
\mu_{e_k}^*: X_{e_i'} \lms \left\{\begin{array}{lll} X_k^{-1}& \mbox{ if } & i=k, \\
    X_i(1+X_k^{-{\rm sgn}(e_i, e_k)})^{-(e_i, e_k)} & \mbox{ if } &  i\neq k. \\
\end{array} \right.
\end{equation} 
Evidently, thanks to (\ref{NZV}), formula (\ref{TF}) 
is a specilaization of the general formula (\ref{f3}). 
\end{proof}

{\it Proof of Theorem \ref{mutations}.}
We  use the edge weights,
where the notation are as on Figure \ref{di31}. The equality of the
modified partition functions reads as follows:
\be \la{PPS1}
\frac{P_{\Gamma}}{M_{\alpha}} = \frac{P_{\Gamma'}}{M_{\alpha}'} 
\ee
where $P_{\Gamma}$ is the partition function of the dimer
model on the bipartite graph $\Gamma$ defined using the edge weights, and 
$$
M_{\alpha}:= \prod_{\rm edges ~E}{\rm weight}(E)^{\varphi(E)}.
$$
Here $\varphi(E) = \alpha_r-\alpha_l$ is determined by the function $\alpha$, see
(\ref{alpha1}), (\ref{alpha2}), and Figure \ref{di33}. 
Observe that the maps $\alpha$ for the graphs $\Gamma$ and $\Gamma'$
are the same maps.

The face weights are calculated via the edge weights as follows, where
$W_i^*$ are factors common for both graphs $\Gamma$ and $\Gamma'$, and
the edge weights are as on Figure \ref{di31}:
$$
W=\frac{dc}{be}, ~~W_1 = \frac{b}{a}W_1^*, ~~ W_2 =
\frac{a}{c}W_2^*~~, 
W_3 = \frac{e}{f}W_3^*, 
~~ W_4 = \frac{f}{d}W_4^*.
$$
$$
W'=\frac{AC}{BD}, ~~W'_1 = \frac{F}{A}W_1^*, ~~ W'_2 =
\frac{B}{E}W_2^*~~, 
W'_3 = \frac{E}{C}W_3^*, 
~~ W'_4 = \frac{D}{F}W_4^*.
$$
\begin{figure}[ht]
\centerline{\epsfbox{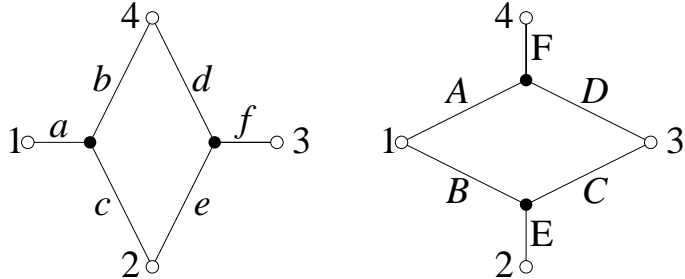}}
\caption{\label{di31}A spider move and a mutation of the edge weights.}
\end{figure} 

Let us number the external vertices on the graphs shown on Figure \ref{di31}
by $1,2,3,4$. We present the graphs as  unions of the {\it inside}
and {\it outside graphs}, where the inside graphs are the ones shown on  Fig
\ref{di31}. 
Given a pair of vertices $i,j\in\{1,2,3,4\}$, consider all perfect matchings 
where the vertices $i$ and $j$ are matched to the
outside vertices. Their contribution
to the partition function is the product of the total weight of
the perfect matchings for the outside graph  by the weight $p_{ij}$ of the perfect matchings
 for the  inside graph.  Similarly
$p'_{ij}$ for $\Gamma'$. One should have
\be \la{PPS}
p_{ij} = Cp_{ij}', 
\ee
i.e. the points in $\C\P^5$  with the
coordinates $(p_{12}:p_{13}: ... : p_{34})$ and $(p'_{12}:p'_{13}:
... : p'_{34})$ coincide. 
One has 
\be \la{MUTTFF1}
p_{12}= bf, ~~ p_{13}= be+dc, ~~ p_{14}= cf, ~~
p_{23}= ad, ~~p_{24}= af, ~~
 p_{34}= ae.
\ee
\be \la{MUTTFF2}
p'_{12}= CF, ~~ p'_{13}= EF, ~~ p'_{14}= DE, ~~
p'_{23}= BF, ~~p'_{24}= AC+BD, ~~
 p'_{34}= AE.
\ee
For example, there are exactly two matchings of
the inside graph which do not use the vertices $1,3$, and their total
weight is $p_{13}= be+dc$. 

Then for every pair $(ij)$ and $(kl)$ there is a condition
$$
\frac{p_{ij}}{p_{kl}} = \frac{p'_{ij}}{p'_{kl}}.
$$
Let us investigate its consequences. 
We deduce that $W'=W^{-1}$ by looking at
$$
\frac{p_{13}p_{24}}{p_{12}p_{34}} =
\frac{(be+dc)af}{(bf)(ae)} = 1 +\frac{dc}{be} = 1+W; 
$$
$$
\frac{p'_{13}p'_{24}}{p'_{12}p'_{34}} = \frac{(EF)(AC+BD)}{(CF)(AE)}
= 1+(W')^{-1}.
$$
The condition $\frac{p_{12}}{p_{24}} =
\frac{p'_{12}}{p'_{24}}$ leads to the following:
$$
\frac{p_{12}}{p_{24}} =\frac{b}{a}, ~~
\frac{p'_{12}}{p'_{24}} = 
\frac{F}{A} \frac{AC}{AC+BD} ~ ~ => ~~ W_1' = W_1(1+W).
$$
The rest of the relations (\ref{mll}), see Figure \ref{di32} are obtained by looking at
$\frac{p_{23}}{p_{24}}$, $\frac{p_{34}}{p_{24}}$,
and $\frac{p_{14}}{p_{24}}$. Conversely, transformation formulas
(\ref{mll}) imply that the vectors $(p_{ij})$ and $(p'_{ij})$ in $\C\P^5$
coincide.

Furthermore, we claim that  (\ref{PPS}) implies (\ref{PPS1}). The subtlety 
is that the modified partition function involves the function
$\alpha$. Let us set, see the right picture on Figure \ref{di30}, 
$$
q_1=p_3-p_1,
q_2=p_2-p_3, q_3=p_4-p_2, q_4=p_1-p_4.
$$ 
Then the denominators in (\ref{PPS1}) for the inner graphs are 
$$
M^{\rm in}_\alpha = a^{p_1-p_2}b^{p_2-p_3}c^{p_3-p_1}d^{p_4-p_2}
e^{p_1-p_4}f^{p_2-p_1} = (cf)^{q_1}(bf)^{q_2}(ad)^{q_3}(ae)^{q_4},
$$
$$
{M'}^{\rm in}_\alpha = A^{p_1-p_4}B^{p_4-p_2}C^{p_2-p_3}D^{p_3-p_1}
E^{p_3-p_4}F^{p_4-p_3} = (DE)^{q_1}(CF)^{q_2}(BF)^{q_3}(AE)^{q_4}.
$$
Therefore, using (\ref{PPS}), $M^{\rm in}_\alpha = C^{q_1+q_2+q_3+q_4}{M'}^{\rm
    in}_\alpha = C{M'}^{\rm in}_\alpha $ since $q_1+q_2+q_3+q_4=1$. 
So, using (\ref{PPS}) again, we get (\ref{PPS1}). 
\end{proof}

\begin{figure}[ht]
\centerline{\epsfbox{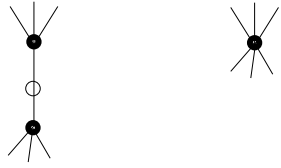}}
\caption{\label{di4}Shrinking a $2$-valent vertex.}
\end{figure}

There is another transformation, see Figure \ref{di4}, which obviously
does not change the face weights and the Poisson structure. 
Its inverse involves
splitting of a high valency vertices by adding valency two vertices.

\subsection{Quantum  integrable system} \la{QIS}

Recall the quantum torus algebra ${\bf T}_\Lambda$ assigned to a 
lattice $\Lambda$ with a skew-symmetric integral 
form $(\ast, \ast)$, see Section \ref{sec4.1.3}. 
Each vector $v\in \Lambda$  gives rise to an element 
$X_v\in {\bf T}_\Lambda$. 
 
Each Hamiltonian $H_{\alpha, a}$ is a sum of different monomials with coefficients $+1$. 
So it can be upgraded uniquely to an element of the   
quantum torus algebra ${\bf T}_\Lambda$ assigned to a minimal bipartite 
surface graph $\Gamma$:
\be \la{QQHHAAMM}
{\mathbb H}_{\alpha; a} := \sum_vX_v \in {\bf T}_\Lambda
\ee
where the sum is over the vectors $v$ parametrising the monomials in 
the classical Hamiltonian $H_{\alpha, a}$. The element (\ref{QQHHAAMM}) 
is the  quantum Hamiltonian. Similarly one defines 
quantum Casimirs corresponding to the classical Casimir monomials. 

\bp\la{LaurentP}
Let $\Gamma\to \Gamma'$ be a spider move. 
Then the corresponding quantum cluster transformation preserves 
the quantum Hamiltonians and Casimirs.
\ep

\begin{proof} The Casimirs lie in the center and thus 
commute with $\Psi_q(X_v)$. 

For the quantum Hamiltonians we use a version of the strategy employed
at Section \ref{sec4.2q}. First, let us do a calculation at $q=1$. 
Let us denote by $q_{ij}p_{ij}$ (respectively $q_{ij}'p_{ij}'$)
the part of the modified partition function 
involving the dimer covers which connect the vertices $i$ and $j$ to the outside 
vertices. Here $p_{ij}$ (respectively $p_{ij}'$) is the factor calculated 
by using the matchings involving only the six edges $a,b,c,d,e,f$ 
(respectively, $A,B,C,D,E,F$) on Figure \ref{di31}.

We write $A\sim B$ when $A/B$ is independent of 
\be \la{THEW}
W, W_1, W_2, W_3, W_4, ~~ \mbox{and ~~$W', W'_1, W'_2, W'_3, W'_4$}.
\ee
Since the product of all face weights is equal to one, we have 
$$
WW_1W_2W_3W_4\sim 1, \quad W'W'_1W'_2W'_3W'_4\sim 1.
$$
Below we work modulo these relations. 
The $1$-cycle assigned to any non trivial Laurent polynomial in $W, W_1, W_2, W_3, W_4$ 
has a non-trivial contribution 
of the  edges $a,b,c,d,e,f$. 

The $q_{ij}$  are
Laurent polynomials in the face weights for $\Gamma$,
 well defined  modulo a common Casimir  
used to correct  the partition function. 
 Therefore they do not 
depend on the left face weights in (\ref{THEW}). 
So we can 
calculate the contribution of  the $W$'s from  (\ref{THEW}) to the 
 ratios of $p$'s by looking 
at (\ref{MUTTFF1}). 
For example,  we get 
$$
\frac{p_{13}}{p_{12}} \sim \frac{be + dc}{bf} = \frac{e}{f} + \frac{dc}{bf} \sim 
W_3+ WW_3.  
$$ 
Indeed, the contribution of $a,b,c,d,e,f$ on $\Gamma$ to the $1$-cycle 
assigned to $W_3$ is  $[e] - [f]$, and the contribution to $WW_3$ is  $[d]+[c]-[b]-[f]$. 
These Laurent polynomials in $W$ are uniquely defined 
modulo $WW_1W_2W_3W_4\sim 1$. 

The same argumentation in general delivers the following result:
$$
\frac{p_{13}}{p_{12}}\sim W_3 + WW_3, \quad \frac{p_{14}}{p_{12}}\sim WW_3W_4, \quad  
\frac{p_{23}}{p_{12}}\sim WW_2W_3,
$$
$$
\frac{p_{24}}{p_{12}}\sim W_1^{-1}, \quad  
\frac{p_{34}}{p_{12}}\sim W_1^{-1}W_3.
$$
Similarly, we get 
$$
\frac{p'_{13}}{p'_{12}}\sim W_3', \quad \frac{p_{14}}{p_{12}}\sim W_3'W_4', \quad  
\frac{p'_{23}}{p'_{12}}\sim W_2'W_3', 
$$
$$
 \frac{p'_{24}}{p'_{12}}\sim 
(W_1')^{-1} + W(W_1')^{-1}, \quad  
\frac{p'_{34}}{p'_{12}}\sim (W_1')^{-1}W'_3.
$$
 Now it is straitforward to chack that the mutation at $W$ 
formula amounts to mapping the five Laurent polynomilas in $W'$'s 
to the similar Lauren polynomials in $W$'s. 
Moreover, the same is true for the quantum mutation maps, assuming that 
the Laurent polynomials above are written with the correcting 
$q$-factors provided by the  formula (\ref{4.28.03.11xewr}): 
$$
W_3' \lms W_3(1+qW).\quad q^{-1}W_3'W_4'\lms W_3WW_4, ~~\ldots
$$It remains to notice that $WW_1W_2W_3W_4 \lms W'W_1'W_2'W_3'W_4'$.
\end{proof}

\bt \la{completeintegrabilitytheorem2}
Let $\Gamma$ be a minimal bipartite graph on a torus $\T$. Then 
the quantum Hamiltonians (\ref{QQHHAAMM}) commute. 

The quantum Casimirs generate the center. 
The quantum Hamiltonians ${\mathbb H}_{\alpha; a}$ together with the quantum Casimirs 
 provide a quantum cluster integrable system in ${\cal O}_q({\cal X}_N)$. 
\et

\begin{proof}
One has 
$$
[X_{v}, X_w] = (q^{(v, w)} - q^{(w, v)})X_{v+w}.
$$
Using the notation 
of Lemma  \ref{4.16.10.1}, we have
$$
[X_{\mu_1}, X_{\mu_2}] + [X_{\widetilde \mu_1}, X_{\widetilde \mu_2}] = 
(q^{(\mu_1,\mu_2 )} - q^{( \mu_2, \mu_1)})X_{\mu_1+\mu_2} +
(q^{(\widetilde \mu_1, \widetilde \mu_2 )} - 
q^{(\widetilde \mu_2, \widetilde \mu_1)})X_{\widetilde \mu_1+ \widetilde \mu_2}.
$$
By construction we have 
$
\mu_1+\mu_2 = \widetilde \mu_1+ \widetilde \mu_2. 
$ 
So  the commutator is zero by Lemma \ref{4.16.10.1}. 
Therefore the quantum Hamiltonians commute. The claim about 
Casimirs is obvious since they are monomials. 
The last claim follows Proposition \ref{LaurentP} 
and  Theorem \ref{completeintegrabilitytheorem}ii). 

\end{proof}

\section{Resistor network model and its cluster nature}\label{resistornetsection}

\subsection{Resistor network model}
Let $G$ be a (not necessarily bipartite) surface graph on  a torus $\T
$.
Let $c$ be a function on the edges of $G$ with values in $\C^*$, called 
the {\it conductance function}, considered up to a multiplication by a non-zero scalar.\footnote{Usually in applications the conductance function is a positive real-valued function.} A graph with a conductance function on its edges is called a \emph{resistor network}.

\subsubsection{Laplacians}

A line bundle with connection $V$ on $G$ gives rise to a vector space 
$$
\V= \oplus_{v}V_v,
$$ given by the sum of the fibers $V_v$ of the line bundle  over the vertices $v$ of $G$, and a Laplace operator $\Delta:\V \to \V$ defined by
$$
\Delta(f)(v) := \sum_{v'\sim v}c(v,v')(f(v)-i_{v'v}f(v'))
$$
where the sum is over neighbors $v'$ of $v$, and 
$i_{v'v}f(v')$ is the parallel transport of the vector $f(v')$ to the fiber over $v$. 

Suppose now that the bundle with connection $V$ on $G$ is flat, 
that is, the monodromies around the faces of $G$ 
are trivial. The connection still has non-trivial monodromies $z_1,z_2$ 
around generators of $H_1(\T, \Z)$. 
The determinant of $\Delta$ is a Laurent polynomial $\det\Delta=P(z_1,z_2)$. It is symmetric, since $\Delta$ is 
Hermitian when $|z_1|=|z_2|=1$ and $c$ is a positive real function: 
\begin{equation}\la{zs}
P(z_1,z_2)= P(z_1^{-1},z_2^{-1}).
\end{equation}
 
The Laplacian determinant is a weighted sum of combinatorial
objects called \emph{cycle-rooted spanning forests} (CRSFs), see \cite{K.laplacian}:
these are maximal subsets of edges of $G$
in which each connected component contains a unique
cycle, and this cycle has nontrivial homology on the torus. The weight of a
CRSF is the product over its edges of the conductances, times the
product over its cycles of $2-w-1/w$ where $w$ is the monodromy of the flat
connection on that cycle. 

\begin{theorem}[\cite{K.laplacian}]\label{CRSFs=detLap} We have
$$\det\Delta = \sum_{\text{CRSFs }C}\,\prod_{e\in C}c(e)\prod_{\text{cycles }\gamma}2-w(\gamma)-1/w(\gamma).$$
\end{theorem}

Two torus graphs $G_1,G_2$ with conductances $c_1,c_2$ are said
to be \emph{electrically equivalent} if their associated Laplacian
determinants $P_1(z_1,z_2)$ and $P_2(z_1,z_2)$ are proportional
(that is, their ratio is a monomial in $z_1,z_2$).
A weaker notion is topological equivalence:
torus graphs $G_1,G_2$ are \emph{topologically equivalent} if
for (any choice of) positive conductances, 
the Newton polygons of $P_1(z_1,z_2)$
and $P_2(z_1,z_2)$ are the same.

\subsubsection{Dual graph} 

Given a resistor network on a surface, the dual graph $G'$ is also a resistor network,
where we assign conductance $1/c$ to an edge dual to an edge of conductance $c$.
The Laplacian determinants on $G$ and $G'$ are related: 

\begin{lemma} For a graph $G$ with dual $G'$ having the same flat line bundle, we have
\be\label{Delta=Delta}
\left(\prod_{\text{edges}}c(e)\right)\det\Delta_{G'}=\det\Delta_G.
\ee 
\end{lemma}

In particular a resistor network on a torus is electrically equivalent to its dual network.

\begin{proof}
Use Theorem \ref{CRSFs=detLap} and
note that, given a CRSF $C$ on $G$, there is a 
dual CRSF $C'$ on $G'$
consisting of the set of edges of
$G'$ which are not crossed by an edge of $C$.  It has the same homology class
as $C$, and thus the same weight $\prod_{\text{cycles}}2-w-1/w$, and the contribution from the edge weights is $\prod_{e\in E}c(e)^{-1}$ times the product of the edge
weights of $C$.
\end{proof}

\subsubsection{Zig-zag paths for a graph $G$ on a torus.} 
A {\it zig-zag path} on a surface graph $G$ is a path which turns maximally
left or right at each vertex, and 
the left/right turns alternate. A zig-zag path on $G$ is unoriented. 
The collection $\{\alpha\}$ of zig-zag paths assigned to a graph $G$ on a torus 
has the following property: 

\begin{lemma} \la{evencrossings}
Each zig-zag path intersects an even number of zig-zag paths (counted with multiplicity). 
\end{lemma}

\begin{proof}
Each vertex of a zig-zag path is either on the left or right of the corresponding
path in the medial graph, so a zig-zag path has an even number of steps. At each
step it crosses a single other zig-zag path.
\end{proof}

Conversely, any collection of smooth unoriented 
loops $\{\alpha\}$ on a torus, in general position, 
satisfying Lemma \ref{evencrossings} and considered up to isotopy,  
arises as the set of medial strands of a graph on the torus. Namely, one can
color the complementary domains black and white 
so that no domains of the same color share an edge. 
We declare the black domains to be the vertices of $G$, and connect two 
vertices of $G$ by an edge if the corresponding domains are separated by an 
intersection point of two zig-zag paths. 
The white domains are the faces of $G$.

\subsubsection{Elementary transformations.} 

Elementary transformations are local rearrangements of a resistor 
network which preserve
electrical equivalence. There are four types: $Y-\Delta$ transformations,
parallel and series reductions, and removal of dead branches. The $Y-\Delta$ transformations 
appeared first in  1899 in  the work of Kennelly \cite{Kennelly}

Recall that a Y-$\Delta$-{\it transformation} of a resistor network is given by replacing 
a triangle $\Delta$ with conductances $\alpha, \beta, \gamma$ by the graph $Y$ 
with conductances $a,b,c$, see Figure \ref{di34}, 
\begin{figure}[ht]
\epsfxsize180pt
\centerline{\epsfbox{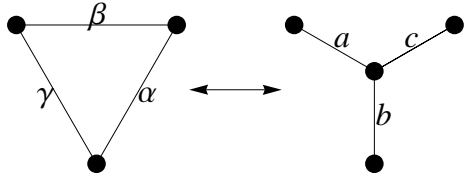}}
\caption{An example of a  Y-$\Delta$-transformation.\label{di34}}
\end{figure}
related as follows: 
\be \la{leiden1}
\alpha = \frac{bc}{a+b+c}, \quad \beta = \frac{ac}{a+b+c}, \quad \gamma = \frac{ab}{a+b+c}.
\ee
\be \la{leiden2}
a= \frac{\alpha\beta + \beta\gamma+\gamma\alpha}{\alpha}, \quad 
b= \frac{\alpha\beta + \beta\gamma+\gamma\alpha}{\beta}, \quad 
c= \frac{\alpha\beta + \beta\gamma+\gamma\alpha}{\gamma}, \quad 
\ee
The graph transformation induced by 
moving a zig-zag strand across a 
crossing of two other zig-zag strands is 
a Y-$\Delta$ move. 

Two more elementary transformations of graphs with conductances, called 
{\it parallel and series reductions}, are shown in 
Figure \ref{di35}.
\begin{figure}[ht]
\epsfxsize320pt
\centerline{\epsfbox{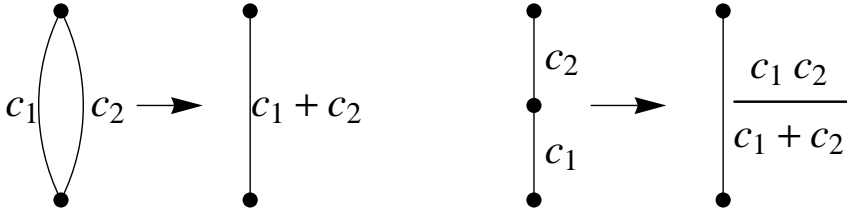}}
\caption{Parallel and series reductions.\label{di35}}
\end{figure}

A last elementary transformation is \emph{removing a dead branch},
which consist of contracting to a vertex any subgraph which is attached
to the rest of the graph by a single vertex. 

A torus graph is said to be \emph{minimal} if the lifts of any zig-zag path to the plane (the 
universal cover of the torus) does not intersect itself, and the lift
of two zig-zag paths to the plane intersect
at most once. 

\begin{theorem}\label{KWproof} 
Every torus graph $G$ with conductances $c$ can be reduced to a minimal
torus graph by a sequence of elementary transformations. 

Two minimal torus graphs $G_1,G_2$ are 
topologically equivalent if and only if $G_2$ is related
to $G_1$ or to its dual graph $G_1'$ by Y-$\Delta$-moves. 
\end{theorem}

We prove this at the end of Section \ref{zigzagNewton}.
The corresponding statement for electrical equivalence is false:
there are minimal graphs which are electrically equivalent but not related
by $Y-\Delta$ moves. In fact the electrical equivalence class
is a subvariety of the Hamiltonian dynamical system;
see below.

\subsection{The associated dimer model on a torus}

\subsubsection{The bipartite graph assigned to a graph on a torus.} 
Drawing the graph  $G$ and its dual graph $G'$ simultaneously one 
can make a bipartite graph $\Gamma_G$ on the torus whose white vertices
are the edges of $G$, black vertices are the vertices and the faces of $G$,
with an edge for each adjacency, see Figure \ref{di36}. 
The graph $\Gamma_G$ has the same number of white and black vertices 
since Euler's formula on a torus reads $E=V+F$.  
One can show \cite{KPW} that the graph $\Gamma_G$ admits dimer covers:
in fact these are in bijection with CRSFs in which both the cycles and dual cycles
(the dual of a CRSF is also a CRSF) 
have a chosen orientation.

\subsubsection{The resistor network subvariety.} 
Let $c$ be a conductance function on the edges of $G$, 
considered up to a non-zero scalar. 
It gives rise to a line bundle with connection $V(c)$ on $\Gamma_G$, 
defined as follows. 
The fibers of $V(c)$ at the vertices of $\Gamma_G$ are identified with $\C$. 
So the connection is defined via the edge weights assigned to the oriented edges 
of $\Gamma_G$. The weight of an edge incident to a vertex corresponding to a face of $G$ is $1$. 
The weight of an edge going from a vertex of $G$ to a midpoint of an edge $E$ 
of $G$ is the conductance $c_E$ assigned to that edge. 

The 
monodromies around a pair of  
generators of $H_1(\T,\Z)$ provide an isomorphism of the 
moduli space of line bundles with flat connections on $\Gamma_G$ with 
$(\C^*)^2$.  
Taking the tensor product of $V(c)$ with a line bundle with a flat connection on 
$\Gamma_G$ we 
get a line bundle with connection on $\Gamma_G$. 
We denote by ${\mathcal R}_{G}$ the moduli space of line bundles with connections on the graph $\Gamma_G$ 
constructed in this way, and call it {\it the resistor network subvariety}. 
It is a closed subvariety
$$
{\mathcal R}_{G}\subset {\L}_{\Gamma_G}, \qquad {\rm dim}{\mathcal R}_{G} = E_G+1
$$
where $E_G$ is the number of edges of $G$ 
(note that multiplying all conductances on $G$
by a constant gives a gauge equivalent connection on $\Gamma_G$).

\subsubsection{Determinant of the Laplacian  and the dimer model partition function} \la{6.4}

Using the conductances on $G$ to define the connection on $\Gamma_G$,
we have the following relation between the dimer partition function and
$\det\Delta$. 

\begin{lemma} \la{KL}
The partition function for the dimer model on 
$\Gamma_G$ is equal to ${\rm det}\Delta_G$ (up to a monomial
factor in $z_1,z_2$). 
\end{lemma}

\begin{proof}  Let $G'$ be the dual graph to $G$ with conductance $1/c$ on an edge
dual to an edge of conductance $c$. Recall Lemma \ref{Delta=Delta} relating the
Laplacian determinants on $G$ and $G'$.

These two operators $\Delta_G,\Delta_{G'}$ acting on the black vertices of 
$\Gamma_G$ are related to the Kasteleyn matrix $K$ as follows:
\be\label{Delta=KK}
\left(\begin{matrix}\Delta_G&0\\0&\Delta_{G'}\end{matrix}\right)=-K_2^*K_1,\ee
where $K_1,K_2$ are two copies of the black-to-white 
Kasteleyn matrix with appropriate gauge, defined as follows.

Embed $G$ in a Euclidean torus with straight edges; this allows us to associate
to each directed edge $b_1b_2$
a unit-modulus complex number $\zeta(b_1,b_2)$ representing its direction. Letting
$w$ be the vertex of $\Gamma_G$ in the center of this edge we then
define $K_1(b_1,w)=c(b_1,b_2)\zeta(b_1,b_2)$ where $c$ is the conductance of edge $b_1b_2$.
If $b_3$ is a black vertex of $G'$ in a face of $G$ adjacent to edge 
$b_1b_2$ then $K_1(b_3,w)$ is
$1$ times the unit modulus complex number in the direction from $b_3$ to $w$
and perpendicular to $\zeta(b_1,b_2)$. One can now check that the monodromy
around a face of $\Gamma_G$ is negative real (for positive conductances)
and thus $K_1$ is a Kasteleyn matrix for $\Gamma_G$. One defines $K_2$
similarly but for the dual graph $G'$. 
Now (\ref{Delta=KK}) follows by inspection of the various cases: 
$-K_2^*K_1(b,b)$ is the sum of the
conductances of edges emanating from $b$; 
$-K_2^*K_1(b,b')=0$ if $b\in G$ and $b'\in G'$
or vice versa; $-K_2^*K_1(b_1,b_2)=-c$ if $b_1,b_2$ are adjacent vertices of $G$ 
and the edge connecting them has conductance $c$; similarly for vertices adjacent
in $G'$.

Finally, $\det K_2,\det K_1$ are both the dimer partition function 
(up to constants which depend on our choice of gauge). Thus from 
(\ref{Delta=Delta}) and (\ref{Delta=KK}) we have
$\det\Delta=\det K_1$ up to multiplicative factor which is a monomial in 
$z_1,z_2$.
\end{proof} 

It follows from (\ref{zs}) that the Newton polygon of 
the dimer model on $\Gamma_G$ is centrally symmetric with respect to the origin.

\subsubsection{Zig-zag paths and the Newton polygon}\label{zigzagNewton}

Zig-zag paths are defined for both a graph $G$ on a torus 
and the associated bipartite graph $\Gamma_G$; zig-zag paths 
on $G$ are unoriented. Zig-zag paths for $\Gamma_G$ are oriented. 
A zig-zag path $\alpha$ on $G$ corresponds to a pair  
of zig-zag paths on $\Gamma_G$, isotopic to $\alpha$ and equipped with the opposite orientations. 
In this way we get all zig-zag paths on $\Gamma_G$. 

The {\it geometric Newton polygon} $N(G)$ is described as follows. Each zig-zag path 
$\alpha$ on $G$ gives rise to a pair of homology classes  $\pm [\alpha] \in H_1(\T, \Z)$, where   
 $[\alpha]$ is the homology class of the path $\alpha$ equipped with an orientation. 
There is a unique, centrally symmetric with respect to the origin, convex 
integral polygon 
$$
N(G) \in H_1(\T, \Z)
$$
with the sides given by the vectors $\pm [\alpha]$ for all zig-zag paths $\alpha$ on $G$. 
Indeed, there is evidently a unique up to a translation
 convex integral polygon in $H_1(\T, \Z)$ with these sides, which is 
centrally symmetric. Thanks to Lemma \ref{evencrossings}, 
it can be centered at the origin.

Clearly the polygon $N(G)$ coincides with the geometric Newton polygon $N(\Gamma_G)$ 
assigned in a similar way, using the homology classes of zig-zag paths, to the bipartite graph $\Gamma_G$.   
Thus Lemma \ref{KL} implies
\bl\la{VIL}
Let $G$ be a minimal graph on a torus. Then the geometric Newton polygon $N(G)$ coincides with the Newton 
polygon of  ${\rm det}\Delta_G$.  
\el

\noindent{\bf Proof of Theorem \ref{KWproof}.} 
By Lemma \ref{VIL}, two equivalent minimal graphs $G$ and $G'$ have the 
same geometric Newton polygons, 
whose boundary sides determine the homology classes of the zig-zag strands. 
Thus the collections of their zig-zag strands are isotopic. 
The isotopy is realized by moving zig-zag strands through zig-zag 
strand crossings, and thus 
the corresponding graphs are related by a sequence of Y-$\Delta$ moves. 
 
To reduce a graph $G$ to a minimal graph, consider the zig-zag strands for $G$,
and isotope each to a straight geodesic through a generic isotopy. Each
type of singularity which resolves corresponds to an elementary transformation:
if it is a parallel, series, or dead branch move, it may change the topology
of the set of strands, but also strictly reduces the number of crossings of strands. 
Eventually one arrives at a graph with a minimal number of strand crossings and
the remaining moves are $Y-\Delta$ moves. 
\hfill$\square$

\subsubsection{Equations defining the resistor network subvariety.} 
Denote by $V_G$, $E_G$, $F_G$ the number of vertices, 
edges and faces of the graph $G$, and similarly $V_\Gamma$, $E_\Gamma$, $F_\Gamma$ those for $\Gamma_G$.
Then $F_\Gamma = 2E_G$. Indeed, every edge $E$ of $G$ is incident to four faces of $\Gamma_G$, and 
each face of  $\Gamma_G$ is counted twice. So one has 
$$
{\rm dim}{\L}_{\Gamma} = F_\Gamma+1 = 2E_G+1.
$$
So the codimension of the resistor network 
subvariety ${\mathcal R}_{G}$ is $E_G$. 
On the other hand,  a line bundle with connection 
on $\Gamma_G$ which belongs to the subvariety $\mathcal R_G$ satisfies the following conditions: 

\begin{enumerate}\label{vertextrivial}

\item The monodromies  around the faces of the graph $G$ and the graph $G'$ are equal to $1$. 
\item Given a loop $\alpha$ on the graph $G$ and a loop $\alpha'$ on the graph $G'$ 
whose homology classes in $H_1(\T, \Z)$ coincide, the monodromies around 
these loops also coincide.  
\end{enumerate}

\noindent{\bf Remarks.} The first  condition tells that for each black vertex $b$ on $\Gamma_G$, 
the product of the face weights for the faces adjacent to $b$ is equal to $1$. 

Secondly, the second condition implies the first: to see this, alter a 
loop $\alpha$ by adding to it a  loop going around 
a face of the graph $G$, and similarly for $G'$. 

Finally,
given the first condition, 
the second condition imposes only two more equations on the
connection: one for each generator of $H_1(\T,\Z)$, since the first condition already
implies that the monodromies around homologous paths, both in $G$ or both in $G'$, are equal. 
\medskip

Since the product of the monodromies over all faces of $G$, as well as over all
 faces of $G'$, is $1$,
we get $E_G$ equations for the subvariety ${\mathcal R}_{G}\subset {\L}_{\Gamma_G}$: 
$$
(V_G+F_G -2) +2 = E_G.
$$
 These conditions characterize the subvariety 
${\mathcal R}_G$ in $\L_{\Gamma}$: 

\begin{lemma} A line bundle with connection on $\Gamma_G$ arises from a resistor network on $G$ if and 
only if it satisfies the conditions 1 and 2 above.
\end{lemma}

\begin{proof}  We have seen already that the line bundle with connection on $\Gamma_G$ 
arising from a resistor network satisfies the conditions 1 and 2. 

Let us prove the converse. 
Notice that the condition 1 says that we have a flat line bundle on the torus $\T$ assigned to the graph $G$, and 
another flat line bundle on the torus assigned to the graph $G'$. 
The condition 2 then says that these two flat line 
bundles are isomorphic. So, untwisting by a flat line bundle on the torus, we can assume that 
our line bundle has trivial monodromy around any loop on $G$ and around any loop on $G'$.

Let us trivialize the line bundle on $\Gamma_G$ on  
the dual graph $G'$. Take an edge 
$E_0$ of
$G$ and assign it conductance $1$. 
Recall that the \emph{medial graph} of $G$ is the graph 
with a vertex for every edge of $G$
and an edge whenever two edges of $G$ share a vertex of $G$. The parallel transports
on $\Gamma_G$ determine a curl-free flow on the medial graph of $G$;
the value of the conductance of an edge $E$ is then the integral of this flow along
a path in the medial graph from $E_0$ to $E$. 
So we have reconstructed conductances up to a non-zero common factor. 
\end{proof}

\subsubsection{The resistor network subvariety is Lagrangian.} \la{6.4a}
An irreducible subvariety of a Poisson variety is {\it Lagrangian}  
if the generic part of its intersection 
with the generic symplectic leaves are Lagrangian subvarieties of the symplectic leaves. 
  
\bt \la{LAGRANGIAN}
${\mathcal R}_{G}$ is a Lagrangian  subvariety of the 
Poisson variety  ${\L}_{\Gamma_G}$. 
\et

\begin{proof} The following crucial lemma implies that the intersection 
of ${\mathcal R}_{G}$ with a generic symplectic leaf 
is coisotropic (that is, the symplectic form is zero on this intersection). 

\bl The Poisson brackets between the defining equations 1-2 of the resistor network subvariety ${\mathcal R}_G$ 
are equal to zero. 
\el

\begin{proof}  Let us show first that the monodromies around the black vertices of the graph $\Gamma_G$ 
Poisson commute. Let $b_1, b_2$ be two black vertices of $\Gamma_G$ corresponding
 to the faces $F_1, F_2$ of $G$. Let $W_{b_1}, W_{b_2}$ be the corresponding monodromies. 
If the faces $F_1, F_2$ are not adjacent, the Poisson bracket is obviously zero. 
If they are adjacent, they share an edge $E$, which has, by construction, 
two black vertices at the ends and one white vertex at the middle, see Figure \ref{di36}. Thus 
$\{W_{b_1}, W_{b_2}\}=0$. The argument in the dual case when 
$b_1, b_2$ are two black vertices of $\Gamma_G$ corresponding
 to the vertices of $G$ is the same. Finally, let $b$ be a black vertex assigned to a face $F$ of $G$, 
and $b'$ a black vertex assigned to a vertex $v$ of $G$. Then the face $F$ of $G$ and the face 
$F'$ of the dual graph $G'$ corresponding to $v$ either do not intersect, or intersect at two points. 
See Figure \ref{di36} where the boundary of $F$ is a black loop and the one for $F'$ is a dashed  loop. 
In the former case the Poisson bracket is evidently zero. In the latter case, 
the contributions of the intersection points come with the opposite signs, and thus cancel. 
\end{proof}

\noindent\emph{Continuation of proof of Theorem \ref{LAGRANGIAN}.}
Let $2m$ be the number of Casimirs for the Poisson variety ${\L}_{\Gamma_G}$. 
It equals to the number of primitive boundary intervals 
of the Newton polygon $N(\Gamma_G)$ and thus is an even number, since the polygon is centrally symmetric. 
Since the product of all Casimirs is $1$, 
the dimension of the generic symplectic leaf on ${\L}_{\Gamma_G}$ is 
$$
2E_G+1 - (2m-1) = 2(E_G -m+1). 
$$
The Casimirs are monodromies along the zig-zag paths, 
and the latter on the graph $\Gamma_G$ come in pairs: $\{z_1, z_1'\}$, ..., $\{z_{m}, z_{m}'\}$, so that 
$\{z_k, z_k'\}$ have the opposite homology classes, and thus 
the corresponding Casimirs $C_k, C_k'$ are inverse to each other, providing a relation 
$C_kC_k'=1$. 
Since the product of all Casimirs is $1$, the last relation $C_mC_m'=1$ follows from the 
others, while the first $m-1$ are independent. 

Since  
${\rm dim}{\mathcal R}$ equals to $m-1$ 
plus half of the dimension of the generic symplectic leaf, 
it is sufficient to show that the generic symplectic leaves intersected by 
 ${\mathcal R}$ are parametrized 
by (no less than) $m-1$ parameters. Let us assume that there is a relation 
$C_1^{n_1} \ldots C_m^{n_m}=1$. Let us show that $n_1=\ldots =n_m$. 
Take a pair of Casimirs $C_a$ and $C_b$ 
corresponding to non-parallel sides of the 
Newton polygon. Then the corresponding zig-zag strands intersect. 
So there is an edge $E$ shared by them. Altering the conductance 
at the edge $E$ we 
alter just the Casimirs $C_a$ and $C_b$. This implies that $n_a=n_b$. 
Therefore $n_1=\ldots =n_m$. 
\end{proof}

\subsubsection{Quantized cluster monomial Lagrangian subvarieties}\la{QCMLS}
The subvariety ${\mathcal R}$ is a {\it cluster monomial Lagrangian subvariety} --
it is a Lagrangian
subvariety of a cluster Poisson variety ${\mathcal X}$, which is given in
any  cluster coordinate
systems by equating certain monomials to $1$.

According to the correspondence principle of quantum mechanics,
a symplectic manifold $X$ should give rise to a Hilbert space $H_X$ --
the quantization of $X$. Furthermore,
a Lagrangian subvariety $L \subset X$ should provide a line in $H_X$.
Here is how the story goes in the cluster situation.

The quantization of a cluster Poisson variety ${\mathcal X}$
\cite{FG2} provides
a collection of Gelfand triples, parametrized by the eigenvalues
$\chi$ of the Casimirs:
$$
{\mathcal S}_{\chi} \subset {\mathcal H}_{\chi}\subset {\mathcal S}^*_{\chi}.
$$
Here ${\mathcal H}_{\chi}$ is the  Hilbert space discussed in Section \ref{QS}.
The  {\it modular double algebra}, given by the 
tensor product of quantum algebras at $\hbar$ and $1/\hbar$: 
\be \la{7.7.11.1}
{\mathcal O}_q({\mathcal X}) \otimes {\mathcal O}_{q^{\vee}}({\mathcal
X}), \quad
q = e^{i\pi \hbar}, q^{\vee} = e^{i\pi /\hbar},
\ee
is realized by unbounded selfadjoint operators in the Hilbert space
${\mathcal H}_{\chi}$.
The subspace ${\mathcal S}_{\chi}$
 is the maximal subspace of ${\mathcal H}_{\chi}$
where the algebra (\ref{7.7.11.1})  acts,
equipped with a natural topology. The space
${\mathcal S}_{\chi}^*$ is its topological dual. 
One can show using the
methods of {\it loc. cit.} that
any cluster monomial Lagrangian subvariety ${L}\subset {\mathcal X}$
gives rise to a functional
$$
\psi_{L} \in {\mathcal S}^*_\chi,
$$
defined uniquely up to a scalar  by the condition
that $\psi_{L}$ is annihilated by the quantized monomial equations
in
the algebra ${\mathcal O}_q({\mathcal X})$, and  their
counterparts in
the algebra ${\mathcal O}_{q^{\vee}}({\mathcal X})$.

Here is the basic example.
 Let ${\mathcal H} = L^2(\R)$ be the Hilbert space of functions $f(x)$.
Consider unbounded operators
$$
X f(x):= f(x+i\pi \hbar ),
\quad  Y f(x):= e^{ x}f(x),
$$
$$
X^{\vee} f(x):= f(x+i\pi ),
\quad Y^{\vee} f(x):= e^{ x/\hbar }f(x).
$$
Then
$$
X  Y = q^2Y X, \qquad X^{\vee}  Y^{\vee} = {q^\vee}^2Y^{\vee} X^{\vee}.
$$
So we get a representation of
the modular double of the quantum torus algebra with two generators.
The subspace ${\mathcal S} \subset {\mathcal H}$ consists of functions
$f(x)$ with exponential decay at $x \to \pm \infty$,
whose Fourier transform has the same property.
By the Payley-Wiener theorem $f(x)$ has an analytic continuation to
$\C$.

The equation $X=1$ defines a Lagrangian subvariety $L$ in the torus
$(\C^*)^2$ with coordinates $(X, Y)$.
The quantized equations are $(X -1)f(x+iy)=0$,  $(X^\vee -1)f(x+iy)=0$.
For irrational $\hbar$, they
imply that $f(x+iy) = f(x)$. 
This implies that the
functional
${\psi}$ is a constant.

\subsection{Cluster nature of the resistor network space} 

Recall that all  our spaces are equipped 
with a free action of the two-dimensional torus 
${\rm T}$, and relating them to  cluster 
varieties we have to consider their quotients by this action, or, equivalently, 
consider the action of all transformations on the face weights only. 

In Section \ref{6.3.1} we use $'$ in the notation ${\mathcal X}'_G$ and ${\mathcal R}'_G$ to indicate that, 
just like in Section \ref{clusterpoissonsection}, we work with the quotients  ${\mathcal X}_G/{\rm T}$ and 
${\mathcal R}_G/{\rm T}$. 
Moreover we do not define  the 
spaces ${\mathcal A}_G$ and ${\mathcal B}_G$, only the ``quotients'' 
${\mathcal A}'_G$ and ${\mathcal B}'_G$.

\subsubsection{The basic commutative diagram for a graph $G$} \la{6.3.1}
Given a graph $G$ on a torus, let us define four split algebraic tori (that is, varieties 
$(\C^*)^k$ for some $k$)
$
{\mathcal A}'_G, ~{\mathcal B}'_G, ~{\mathcal R}'_G, ~{\mathcal X}'_G,  
$
fitting into a commutative diagram
\be \la{leiden}
\begin{array}{ccc}
{\mathcal A}'_G & \stackrel{j}{\hookleftarrow}& {\mathcal B}'_G\\
&&\\
p\downarrow &\swarrow s&\downarrow q \\
&&\\
{\mathcal X}'_G &\stackrel{i}{\hookleftarrow}&{\mathcal R}'_G 
\end{array}
\ee

The tori ${\mathcal A}'_G, ~{\mathcal B}'_G, ~{\mathcal R}'_G, ~{\mathcal X}'_G$ 
are defined as the tori with the following canonical coordinates:

\begin{itemize}

\item ${\mathcal A}'_G$:  coordinates $\{A_f\}$ at the faces $f$ of the graph $\Gamma_G$.

\item ${\mathcal B}'_G$:  coordinates $\{B_b\}$ at the black vertices $b$ 
 of the graph $\Gamma_G$.

\item ${\mathcal R}'_G$:  coordinates $\{C_w\}$ at the white vertices $w$ of 
$\Gamma_G$, defined up to a scalar.

\item ${\mathcal X}'_G$:  coordinates $\{X_f\}$ at the faces $f$ of $\Gamma_G$, whose product is $1$.

\end{itemize}

The $C_w$-coordinates are nothing else but the conductances assigned to the edges of the graph $G$, 
which are identified with the white vertices of  $\Gamma_G$.

We define the maps between these tori by their action on the coordinates. Below $\sim$ means an incidence relation. So $b\sim f$ means that a black vertex $b$ 
is incident to a face $f$, etc.

First, the map $p$ is the standard map \cite{FG1}
\be 
p^*X_f:=\prod_{g}A_g^{\varepsilon_{fg}},
\ee
where the product is over faces $g$ of $\Gamma_G$, and $\varepsilon_{fg}$ is the bilinear 
form on the faces defined in Section 2.

Next, 
\be 
j^*{A}_f:= \prod_{b \sim f} B_b = B_{b}B_{b'}.
\ee
Here $b, b'$ are the two black vertices of the face $f$ of the graph $\Gamma_G$. 
\begin{figure}[ht]
\epsfxsize130pt
\centerline{\epsfbox{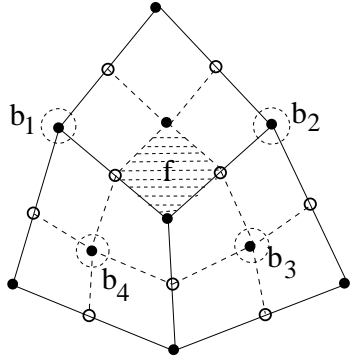}}
\caption{\label{di27}The map $s^*X_f$.}
\end{figure}

Then, we set
 \be \la{5.17.11.2b}
s^*X_f := \frac{{B}_{b_1}{B}_{b_3}}{{B}_{b_2}{B}_{b_4}}.
\ee
 Here $(b_1, b_2, b_3, b_4)$ are the four black vertices which lie outside of the face and 
 incident to the 
two white vertices of the face, as shown on Figure \ref{di27}. 
\begin{figure}[ht]
\epsfxsize80pt
\centerline{\epsfbox{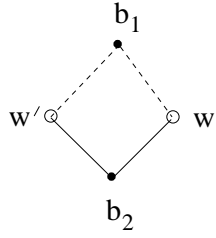}}
\caption{\label{di60}The map $i^*X_f:= {C_{w}}/{C_{w'}}$.}
\end{figure}

Next, 
\be \la{5.17.11.2c}
i^*X_f := \frac{C_{w}}{C_{w'}}.
\ee
Here $w$, $w'$ are the two white vertices of the face $f$, and $w, b, w'$ go counterclockwise around the face, and the black vertex $b$  is a vertex of $f$, Figure \ref{di60}.

Finally, 
\be \la{5.17.11.2d}
q^*C_w := \frac{{B}_{c_1}{B}_{c_3}}{{B}_{c_2}{B}_{c_4}}.
\ee
where $c_1, c_2, c_3, c_4$ are the four black vertices incident to 
the white vertex $w$ as shown on the Figure \ref{di61}. 

\begin{figure}[ht]
\epsfxsize90pt
\centerline{\epsfbox{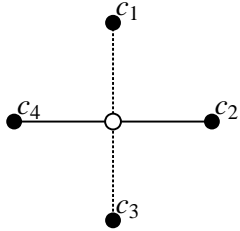}}
\caption{\label{di61}The map $q^*C_w$.}
\end{figure}

Define \cite{FG} a $2$-form on the space ${\mathcal A}'_G$ by 
$$
\Omega_{\mathcal A}:= \sum_{f,g}\varepsilon_{fg}d\log A_f \wedge d\log A_g.
$$

\bp
i) The  diagram (\ref{leiden}) is commutative. 

ii) The map $j$ is an embedding. 
Its image is isotropic for the 2-form $\Omega_{\mathcal A}$.

iii) The 
subvariety $i({\mathcal R}'_G)$ is defined by monomial equations. It is Lagrangian. 

iv) The  subvariety $s({\mathcal B}'_G)$ is Lagrangian in 
the symplectic variety 
$$
{\mathcal U}'_G:= p({\mathcal A}'_G) \subset {\mathcal X}'_G.
$$\ep

\begin{proof} i) One checks this by looking at the  Figure \ref{di27}, \ref{di60}, \ref{di61}. 

ii) Follows easily from the definitions.

iii) The subvariety $i({\mathcal R}'_G)\subset {\mathcal X}'_G$  is given by monomial 
equations
\be\la{leiden4}
\prod_{f:~~ \sum \partial f = \alpha - \beta}X_f =1.
\ee
Here equations are assigned to pairs of loops $\alpha, \beta$ on $G$ which are homologous 
on the torus. The product  is over a set of faces $f$ of $\Gamma_G$ whose total boundary is 
$\alpha - \beta$.
 In particular, taking the empty loop $\beta$, we get equations
 \be \la{leiden5}
\prod_{f\sim b}X_f=1~~~~\mbox{for each black vertex $b$ of the graph $\Gamma_G$.} 
\ee
Indeed, it is easy to see that $i^*(\ref{leiden4})=1$. 
Recall that $E_G$ is the number of the edges of $G$. It is also the number of 
the black as well as the white vertices of $\Gamma_G$. One has 
$$
{\rm dim}~i({\mathcal R}'_G)=E_G-1.
$$Indeed, 
there are $E_G$ independent equations: $E_G-2$ are given by (\ref{leiden5}), and two more 
by  $\beta$'s generating $H_1(\T, \Z)$. 
Since ${\rm dim}~{\mathcal R}'_G=E_G-1$, we proved the first claim. 
The second is similar to the proof of Theorem \ref{LAGRANGIAN}. 

iv) It is similar to the proof of Theorem \ref{LAGRANGIAN}.

\end{proof}

\begin{figure}[ht]
\epsfxsize280pt
\centerline{\epsfbox{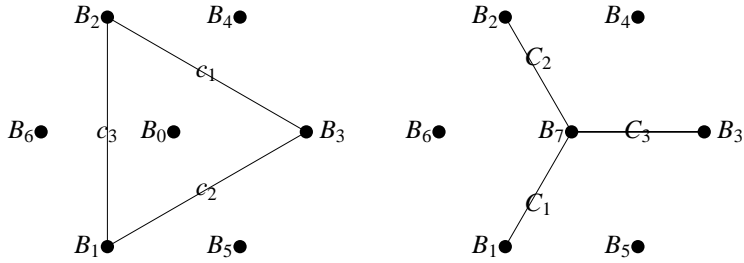}}
\caption{\label{YDeltaBs}Transformation law for $B$'s under a Y-$\Delta$-move.}
\end{figure}

\subsubsection{ Action of the Y-$\Delta$-move on coordinates.} 

A Y-$\Delta$-move acts on the conductances coordinates $C_i$ as in  
(\ref{leiden1})-(\ref{leiden2}). 
Referring to Figure \ref{YDeltaBs},
we have
$$\frac{B_1B_7}{B_5B_6}=C_1=\frac{c_1c_2+c_2c_3+c_3c_1}{c_1}=\frac{\frac{B_2B_3}{B_0B_4}\frac{B_3B_1}{B_0B_5}+\frac{B_3B_1}{B_0B_5}\frac{B_1B_2}{B_0B_6}+\frac{B_1B_2}{B_0B_6}\frac{B_2B_3}{B_0B_4}}{\frac{B_2B_3}{B_0B_4}}$$
from which one obtains
\be \la{6.20.11.1}
B_7=\frac{B_1B_4+B_2B_5+B_3B_6}{B_0}.
\ee
The same equation is obtained starting from the other two conductances $C_2, C_3$.
This equation is the so-called cube recurrence relation, see Section \ref{cuberecsection}
below.

\bl \la{6.1.11.1}
The $Y-\Delta$ move $G \to G'$ is a composition of four cluster mutations. 
Precisely, there are commutative diagrams
$$
\begin{array}{cccccccccc}
{\mathcal A}'_G&\stackrel{j}{\hookleftarrow} &{\mathcal B}'_{G}  
&&&&&{\mathcal X}'_G&\stackrel{i}{\hookleftarrow} &{\mathcal R}'_{G}\\
&&&&&&&&&\\
 \downarrow \mu_{\mathcal A}&&\downarrow &&&&&\downarrow\mu_{\mathcal X} &&\downarrow \\
&&&&&&&&&\\
{\mathcal A}'_{G'}&\stackrel{j}{\hookleftarrow} &{\mathcal B}'_{G'}&&&&& 
{\mathcal X}'_{G'}&\stackrel{i}{\hookleftarrow} &{\mathcal R}'_{G'}
\end{array}
$$
where the down arrows are  
birational maps induced by the  $Y-\Delta$ move action  on the $B$- and $C$-variables, given by  (\ref{6.20.11.1}) and (\ref{leiden1})-(\ref{leiden2}), 
and the right ones 
are compositions of four cluster ${\mathcal A}$- and ${\mathcal X}$-mutations.  
\el

\begin{proof}  i) See Figure \ref{YDfromcluster}.
\begin{figure}[htbp]
\epsfxsize6in
\centerline{\epsfbox{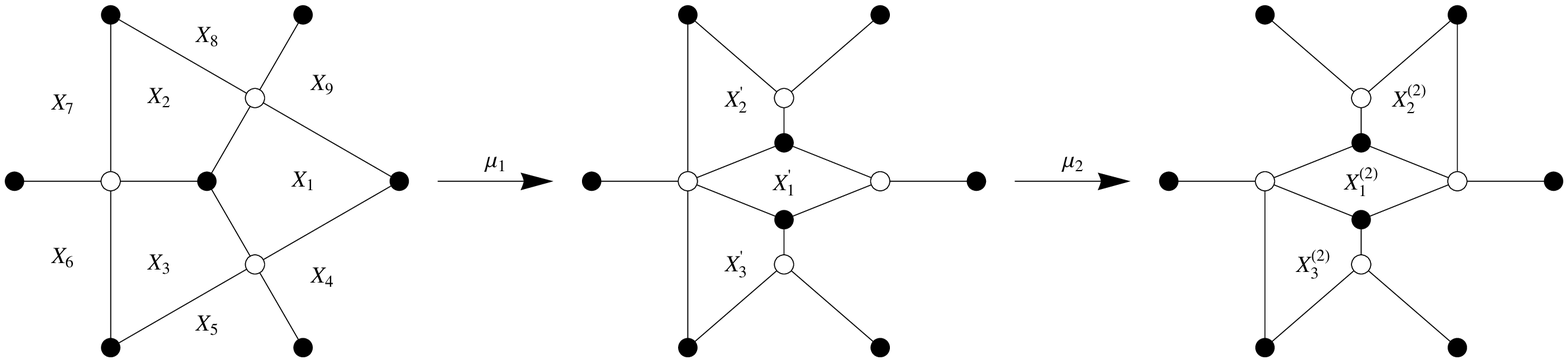}}
\epsfxsize4.8in
\centerline{\epsfbox{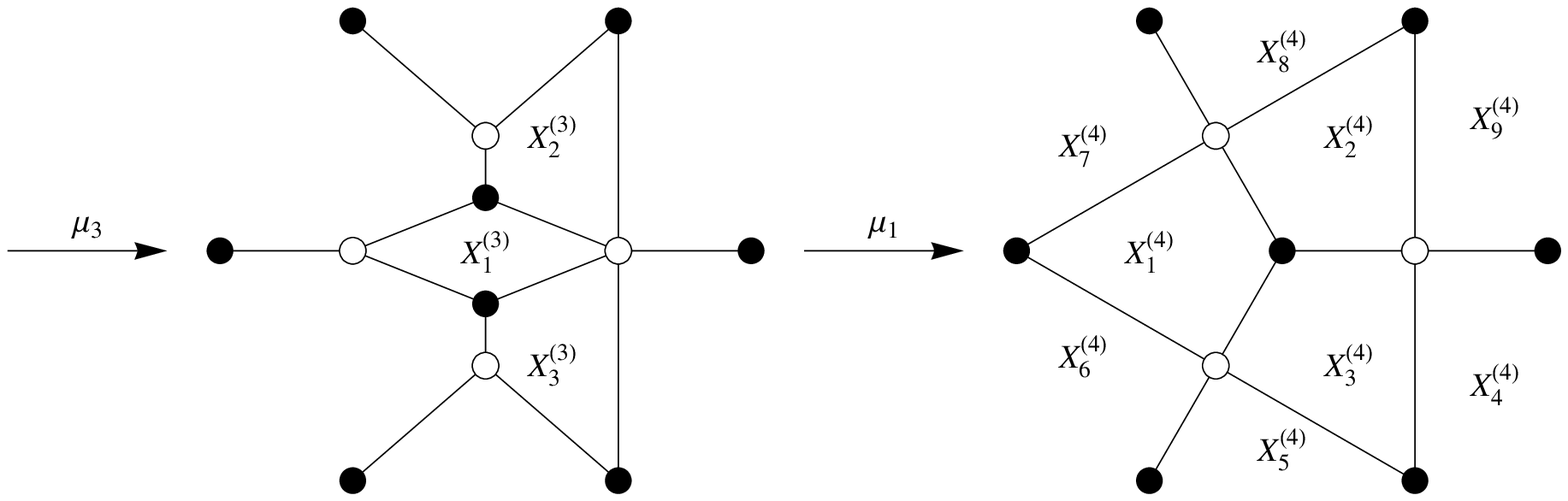}}
\caption{\label{YDfromcluster}The $Y-\Delta$ move from spider moves.}
\end{figure}
After this sequence of $4$ cluster transformations, a little algebra using
Theorem \ref{mutations} shows
that the new $X$ variables
are:
\begin{eqnarray*}X_1^{(4)}&=&\frac{1+X_1+X_1X_3}{X_3(1+X_2+X_1X_2)}\\
X_2^{(4)}&=&\frac{1+X_2+X_1X_2}{X_1(1+X_3+X_2X_3)}\\
X_3^{(4)}&=&\frac{1+X_3+X_2X_3}{X_2(1+X_1+X_1X_3)}\\
X_4^{(4)}&=&X_4(1+X_1+X_1X_3)\\
X_5^{(4)}&=&X_5\frac{X_1X_3}{1+X_1+X_1X_3}\\
X_6^{(4)}&=&X_6(1+X_3+X_2X_3)\\
X_7^{(4)}&=&X_7\frac{X_2X_3}{1+X_3+X_2X_3}\\
X_8^{(4)}&=&X_8(1+X_2+X_1X_2)\\
X_9^{(4)}&=&X_9\frac{X_1X_2}{1+X_2+X_1X_2}.
\end{eqnarray*}

If $X_1X_2X_3=1$, we will also have $X^{(4)}_1X^{(4)}_2X^{(4)}_3=1$.
Moreover if the conductances on the edges of the triangle opposite faces $1,2,3$
are $C_1,C_2,C_3$ respectively then $X_1=C_2/C_3$, $X_2=C_3/C_1,X_3=C_1/C_2$.
One checks that the new conductances are those satisfying (\ref{leiden2}).

ii)  It follows by comparing formulas (\ref{leiden1})-(\ref{leiden2}).
\end{proof}

It follows that the diagram (\ref{leiden}) gives rise to a similar diagram of spaces,  
obtained by gluing the tori assigned to minimal equivalent 
graphs $G$ via the corresponding birational maps 
given either by the Y-$\Delta$-move or the corresponding sequence of four cluster mutations.

\section{Discrete cluster integrable systems}\label{HBDEsection}

\subsection{Discrete cluster integrable systems on a square grid on a torus}

Let us consider the dimer model on a square grid on a 
torus. Such examples are constructed as follows. 
Take the standard square grid bipartite graph in $\Z^2$, 
and map it onto torus by taking the quotient 
by a rank two subgroup of the (bipartite-coloring preserving) translation group 
(that is, $(1,1)\Z+(1,-1)\Z$), see Figure \ref{di0}.

\begin{figure}[ht]
\centerline{\epsfbox{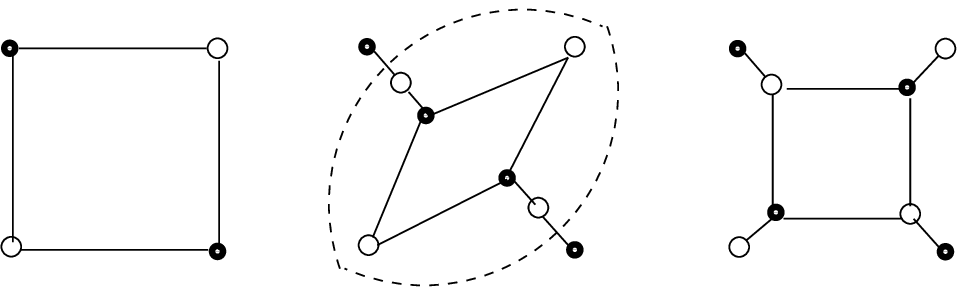}}
\caption{Inserting 
two white $2$-valent vertices, and performing the spider 
move in the dotted region we get a version of the spider 
move on Figure \ref{di2}.}
\label{di70}
\end{figure}

Let us construct an element $A$ of the cluster modular group acting by a non-trivial automorphism 
of the corresponding cluster Poisson variety ${\mathcal X}_N$. 

There is a version of the spider move, see Figure \ref{di2}, given by a composition of two elementary moves on Figure \ref{di70}.

\begin{figure}[htbp]
\centerline{\epsfbox{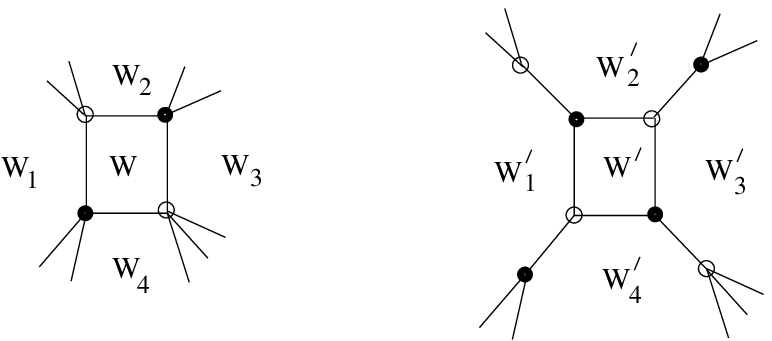}}
\caption{A version of the spider move.\label{di2}}
\end{figure}

The squares of the grid are of two types: one, which we refer to as the 
$\ast$-faces, have the black bottom left vertex. Let us do the Figure \ref{di2} 
version of the 
spider move 
at all $\ast$-faces. The result is illustrated at the right of Figure \ref{di00}. 
Then we shrink the $2$-valent vertices, as shown 
on Figure \ref{di4}. We get again a square grid, 
but now the colors of the vertices flipped, see the bottom of Figure \ref{di00}.  So performing the same sequence 
of moves---the Figure \ref{di2} 
version of the spider moves at the $\ast$-faces and shrinking the $2$-valent 
vertices---we get the original grid. 
The resulting cluster transformation $A$ is non-trivial. 
Thanks to Theorem \ref{mutations} it commutes with the quantum Hamiltonians.

\begin{figure}[htbp]
\centerline{\epsfbox{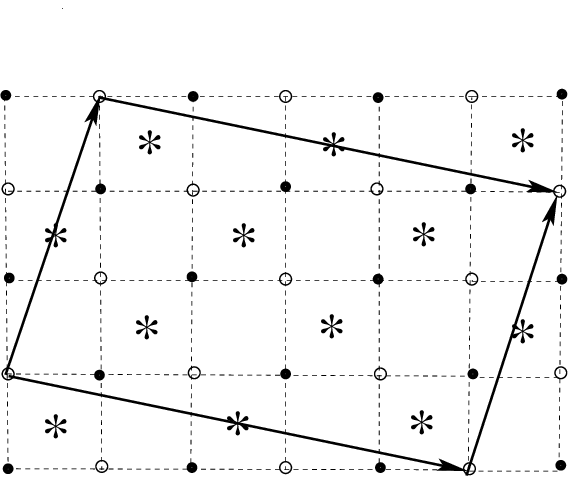}}
\caption{Discrete cluster integrable system on a square grid on a torus.\label{di0}}
\end{figure}

The transformation $A$ has appeared in various guises in the literature.
A version of the transformation $A$ first appeared in \cite{Kor}. 
It was called the \emph{octahedron recurrence} in \cite{Speyer}; 
the relation of the octaherdon reccurence 
with the Hirota bilinear difference equation \cite{Hir} is discussed there. 
A solution of the octaherdon reccurence on a simplex was found in \cite{FG}. 
See also \cite{DFK}, \cite{Naka} for other examples which are special cases.

The cluster transformation $A$ together with the quantum Hamiltonians form a 
discrete quantum integrable system---we think of $A$ as of 
an evolution with a discrete time. 
In the quantized picture \cite{FG2} 
the operator $A$ gives rise to a unitary operator in the Hilbert space of the system 
commuting with the quantum Hamiltonians. It is a composition of 
the quantum dilogarithm intertwiners parametrized by the faces.

\begin{figure}[htbp]
\centerline{\epsfbox{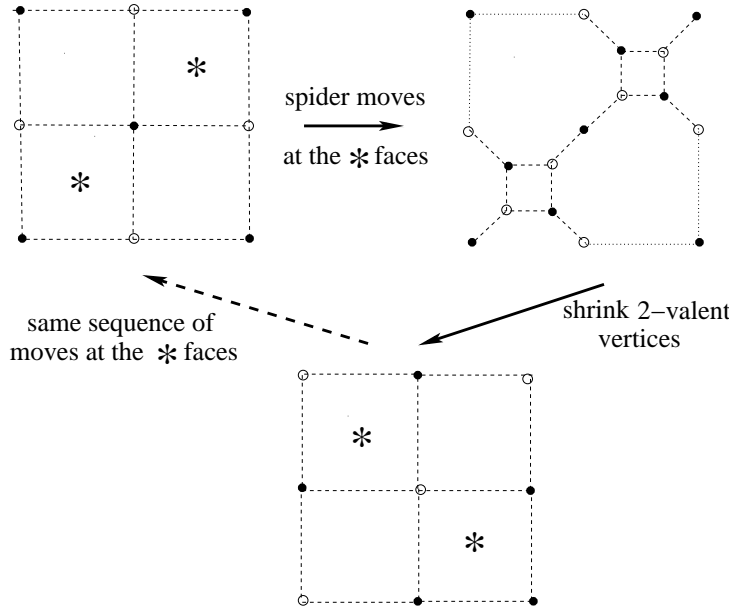}}
\caption{Constructing an element $A$ of the cluster modular group.\label{di00}}
\end{figure}

\subsection{Cluster automorphisms for resistor networks on a torus.} 
\la{resistorautsection}

One also finds cluster automorphisms for cluster varieties associated to
resistor networks on a torus.
The pair $({\cal X}, {\mathcal R})$ assigned to an equivalence class of minimal 
graphs $G$ on a torus 
has a large group ${\mathbb A}$ of cluster automorphisms, 
isomorphic to (a quotient of) $\Z^{s-2}$,  
where $2s$ is the number of sides of the Newton polygon $N(G)$. It is  
constructed as follows. (This material is also found in \cite{HS}.)

For each pair $E, E'$ of centrally symmetric sides  of the Newton polygon $N(G)$ of a minimal graph $G$ on a torus 
there is a collection $\alpha_1, ..., \alpha_k$ of isotopic zig-zag strands 
on $G$ which do not intersect due to minimality of $G$:
their (well defined up to sign) 
homology class in $H_1({\rm T}, \Z)$ is
a primitive vector on $E$.
We assume that they are written in their natural cyclic order provided by the embedding to the torus. 
Isotope them cyclically $(\alpha_1, ..., \alpha_k) \to (\alpha_2, ..., \alpha_k, \alpha_1)$, 
leaving  the other zig-zag paths intact. 
Present this as a sequence of Y-$\Delta$-moves corresponding to the  zig-zag 
strands moving through simple crossings of other zig-zag strands, 
which transforms the graph $G$ to itself.
Take the product of the cluster transformations assigned by Lemma \ref{6.1.11.1} to the Y-$\Delta$-moves.  
We arrive at a cluster transformation, realized by a birational isomorphism 
of the pair $({\cal X}, {\mathcal R})$. It is 
an element of the cluster modular group, as defined in Section 2.4 of \cite{FG1}. The automorphisms assigned to 
different sides commute, and their product is  $1$.

\subsection{The cube recurrence discrete cluster integrable system} \label{cuberecsection}

Let us apply the construction of Section \ref{resistorautsection} to the honeycomb graph on a torus. 
In this case there are three families of parallel zig-zag loops. Moving cyclically 
one of the families we get a cluster automorphism. 
We show that its action on the $B$-variables is given by the 
so-called {\bf cube recurrence} \cite{CS}, 
also known as the {\bf discrete BKP equation} or 
{\bf Hirota-Miwa equation}, which goes 
back to the work of Miwa \cite{Miwa}, see also \cite{BS}.
\medskip

A function $a:\Z^3\to\C$ satisfies \emph{the cube
recurrence} if 
\be\label{cuberecdef}
a_{x+1,y+1,z+1}a_{x,y,z}=a_{x+1,y,z}a_{x,y+1,z+1}+a_{x,y+1,z}a_{x+1,y,z+1}+a_{x,y,z+1}a_{x+1,y+1,z}.
\ee
Note that if $a$ is defined and nonzero for $0\le x+y+z\le 2$ then 
equation (\ref{cuberecdef})
defines it for $x+y+z=3,$ also for $x+y+z=-1$, 
and by iteration for all $(x,y,z)\in\Z^3.$

\begin{figure}[htbp]
\epsfxsize4in
\centerline{\epsfbox{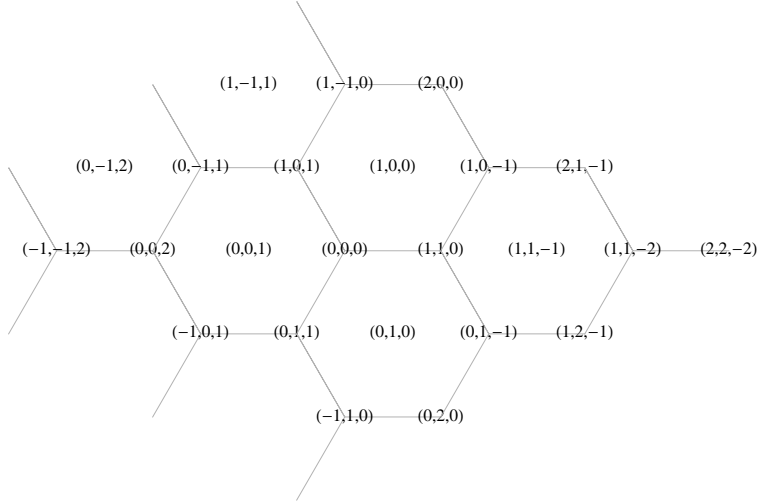}}%
\caption{Vertex  and face labeling for the cube recurrence.\label{cuberecgraph}}
\end{figure}

Let $G$ be the honeycomb graph, thought of as a resistor network.
Vertices of $G$ are the orthogonal projection to the plane $\{x+y+z=0\}$ of the set
$\{(x,y,z)\in\Z^3~|~x+y+z=0\text{ or }2\}$; edges connect vertices whose coordinate sum is $0$ to those nearest neighbors whose coordinate sum is $2$.
We can likewise index faces of $G$ by $(x,y,z)$
satisfying $x+y+z=1$.  See Figure \ref{cuberecgraph}.
Let $V_{k}$ be the set of vertices or faces with $x+y+z=k$.

Assign variables $B_{x,y,z}$ to the vertices and faces of $G$. 

Perform a $Y-\Delta$ move at each vertex $v_{x,y,z}$ where $x+y+z=0$. 
The graph becomes the regular triangulation with vertices at $V_2$,
faces at $V_1$ and new faces at $V_3$. 
Using (\ref{6.20.11.1}), the $B$-value at a new face is obtained from (\ref{cuberecdef}). See Figure \ref{resistoraut}.

Next, perform $\Delta-Y$ moves on all faces in $V_1$. 
The resulting graph is again a honeycomb, with vertices $V_2,V_4$ and faces at
$V_3$. Iterating this procedure we produce variables $B_{x,y,z}$ for
all $(x,y,z)\in\Z^3$ and these variables satisfy the cube recurrence.

Equivalently, we can look at zig-zag strands and move cyclically one of the 
families of parallel zig-zag strands, say the vertical ones on Figure \ref{cuberecgraph}. 
Then each strand will cross two different sets of crossings of other 
two familes, which can be thought of as the vertices of the grid, 
resulting into $Y-\Delta$ and $\Delta-Y$ moves described above. 

By Lemma \ref{6.1.11.1},
 the composition ${\rm B}$ of these $Y-\Delta$ moves and $\Delta-Y$ moves 
is a cluster automorphism, i.e. an element of the cluster modular group, 
 of the dimer model phase space assigned to the graph $G$. 
Therefore it gives rise to a discrete cluster integrable system, with the discrete 
time evolution given by ${\rm B}$.  
Thanks to (\ref{6.20.11.1}), the cube recurrence is obtained by restricting the cluster automorphism ${\rm B}$ 
to the isotropic  subvariety ${\cal B}$.

\begin{figure}[htbp]
\epsfxsize2in
\centerline{\epsfxsize2in\epsfbox{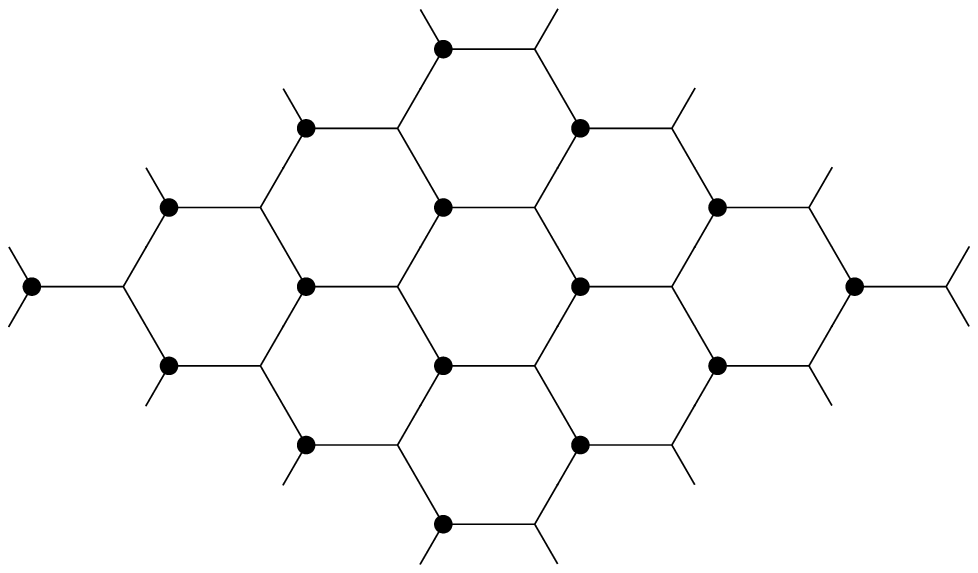}\epsfxsize2in\epsfbox{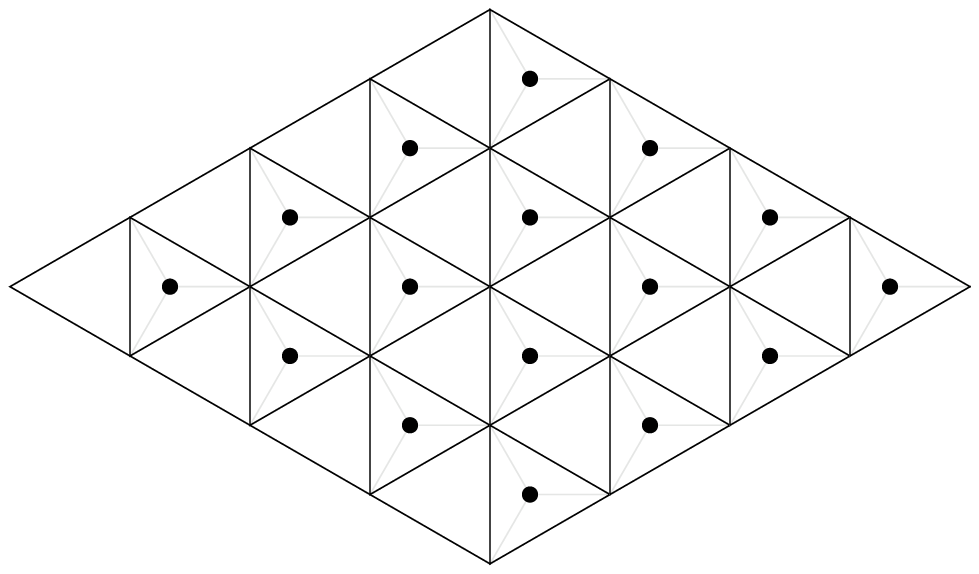}\epsfxsize2in\epsfbox{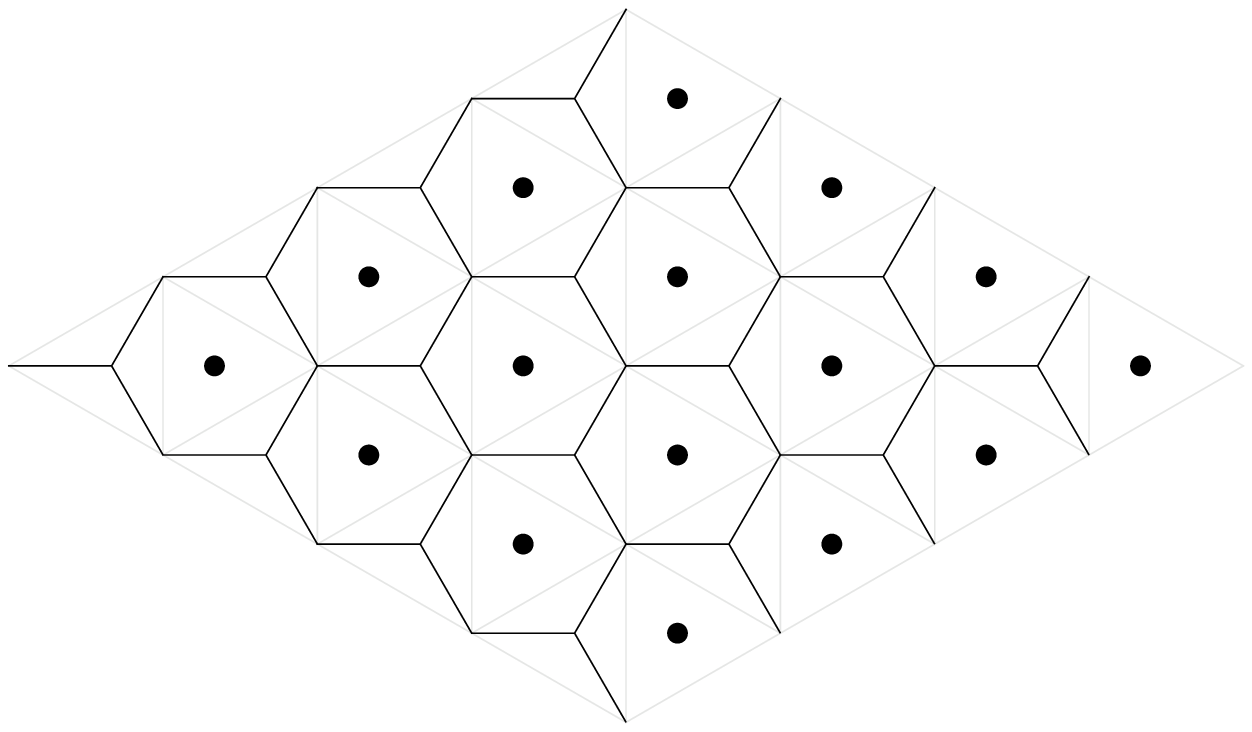}}%
\caption{Steps of the cube recurrence.\label{resistoraut} The black dots
are the vertices $v_{x,y,z}$ with $x+y+z=0$.}
\end{figure}

\section{Newton polygons, toric surfaces and integrable systems}
\label{AGsection}

\subsection{Newton polygons and toric surfaces}
Let us briefly recall this well known connection. See \cite{F} for more details. 
Let ${\rm T}=(\C^*)^2$. 
Denote by  $X^*({\rm T})$ the group of its characters, isomorphic to $\Z^2$. 

A Newton polygon $N$ is a convex polygon in the real plane 
$X_*({\rm T})_\R:=X^*({\rm T})\otimes\R=\R^2$ 
with vertices at the lattice 
$X^*({\rm T})$. It gives rise to a 
polarized projective toric surface ${\mathcal N}$, possibly singular. 
The polarization is given by a ${\rm T}$-equivariant 
line bundle ${\mathcal L}$ on ${\mathcal N}$. Shifting 
the Newton polygon by a lattice vector $l \in X^*({\rm T})$ 
amounts to changing the equivariant structure of the line bundle ${\mathcal L}$ 
by multiplying it by the trivial line bundle with the torus ${\rm T}$ action given  by the 
character $\chi_l$ assigned to $l$. 

Each edge $E$ of $N$ gives rise to a ray $\rho_E$ in the dual space $X_*({\rm T})_\R$, 
consisting of the normals to the edge $E$ looking inside of $N$. 
These rays form a fan $\Delta_N$. It subdivides the plane $X_*({\rm T})_\R$ 
into a disjoint union of cones. Each cone  $\sigma$ provides the dual cone $\sigma^{\vee}$, and hence 
a semigroup $R_\sigma:= \sigma^{\vee} \cap X^*({\rm T})$ and an affine surface 
$U_\sigma:= {\rm Spec}~\Z[R_\sigma]$. Gluing the surfaces  $U_\sigma$ so that $U_{\sigma_1\cap \sigma_2} = 
U_{\sigma_1}\cap U_{\sigma_2}$ we get the surface ${\mathcal N}$.
 
The torus ${\rm T}$ acts on ${\mathcal N}$. There is a unique $2$-dimensional orbit, $U_0$,  
 where $0$ is the origin, denoted below by ${\cal N}_0$.  Its complement 
is the {\it divisor at infinity} ${\mathcal N}_\infty$, whose shape is described by the 
polygon $N$: the irreducible components $D_E$ are projective lines parametrized by the edges $E$ of $N$, and 
their intersections match the vertices of $N$. 
The kernel of the action of the torus ${\rm T}$ on the component $D_E$ is the cocharacter 
$\chi_E^{\vee}: \C^* \to {\rm T}(\C)$ assigned to the conormal line to the edge $E$.

There is a canonical isomorphism, serving as a definition of the line bundle ${\mathcal L}$: 
$$
i_N: {\mathcal L} \stackrel{\sim}{\lra} {\mathcal O}(D), \qquad D = \sum_E a_ED_E.
$$ 
where the 
coefficient $a_E\in \Z$ is the distance from the origin to the edge $E$.\footnote{The latter is the 
value at the origin of the unique  surjective affine linear function $\alpha_E: X_*({\rm T}) \to \Z$ 
which vanishes 
on $E$ and is non-negative on $N$.} The torus ${\rm T}$ action on ${\mathcal L}$ comes from
 the natural 
${\rm T}$-action on ${\mathcal O}(D)$, provided by the action on ${\mathcal N}$ preserving the divisor $D$. 
Shifting the polygon $N$ by a vector $l\in X^*({\rm T})$ amounts 
to multiplication of the isomorphism $i_N$ by the character $\chi_l$, i.e. $i_{N+l}=i_N\chi_l$.  
Using the isomorphism $i_N$, we have 
\be \la{9.20.10.1}
H^0({\mathcal N}, {\mathcal L}) \stackrel{\sim}{\lra} H^0({\mathcal N}, {\mathcal O}(D)) \subset \C({\rm T}) 
\stackrel{\sim}{=} \C(z_1, z_2),
\ee
where $(z_1, z_2)$ is a basis in the group of characters of ${\rm T}$. It is 
identified with the space of all Laurent polynomials with 
the Newton polygon $N$. 

So given a pair $({\mathcal N}, {\mathcal L})$, the polygon $N$ is recovered as 
the  convex hull of the set of characters of the action of the torus ${\rm T}$ on   
$H^0({\mathcal N}, {\mathcal L})$. 

\vskip 2mm
Denote by $\omega$ the canonical up to a sign invariant $2$-form on ${\mathcal N} - {\mathcal N}_\infty$. 
Given a basis  $(z_1, z_2)$ of characters of the torus ${\rm T}$, it 
can be written as  
$$
\omega  = d\log z_1 \wedge d\log z_2.
$$
So it 
is a  
 $2$-form on ${\mathcal N}$ with logarithmic singularities at ${\mathcal N}_\infty$.

\subsection{The spectral data \cite{KO}} 

A {\it parametrization} of a degree $d$ divisor ${C}= \sum^k_{i=1}n_ic_i$ by a set $S$ is a map 
$$
\nu: S \lra \{c_1, ..., c_k\}~~\mbox{such that $|\nu^{-1}(c_i)|=n_i$}.
$$
Let $|{\mathcal L}|$ be a linear system of curves  on 
${\mathcal N}$ given by the zeros of sections of 
${\mathcal L}$.
\bd
A spectral data
is a triple $(C, S, \nu)$ where $C$ is a genus $g$ curve on 
the toric surface ${\mathcal N}$ from 
the linear system $|{\mathcal L}|$, $S$ is a 
degree $g$ effective divisor on $C_0=C\cap {\mathcal N}_0$, and $\nu = \{\nu_E\}$ are 
parametrizations of the divisors ${D}_E\cap C$.

The moduli space ${\mathcal S}_N = {\mathcal S}$ parametrizes the spectral data 
related to $N$. 
\ed

Let $C$ be a curve on ${\mathcal N}$ from the linear system 
 $|{\mathcal L}|$. Its intersection with 
${\mathcal N}_0$ is defined by a Laurent polynomial $P(z_1, z_2)$ 
with the Newton polygon $N$. The intersection of $C\cap D_E$ is determined by the sum of monomials 
of $P$ supported on the edge $E$. 
The ratio of monomials assigned to two integral points on the sides of $N$ are Casimirs. 
So the Casimirs are naturally assigned to the primitive boundary edges of $N$,
and their product is equal to $1$.

\bp
Given a black vertex $v$ of a minimal bipartite 
graph $\Gamma$ assigned to $N$, there is a natural rational map 
$$
\kappa_{\Gamma, v}: {\mathcal X} \lra {\mathcal S}.
$$
\ep

\begin{proof} A point of ${\mathcal X}$ parametrizes 
a line bundle with connection on $\Gamma$, encoded by the Kasteleyn operator ${K}$. 
Its determinant is a section of the line bundle ${\mathcal L}\otimes {\bf L}$.  
The divisor of zeros of the section is the {\it spectral curve} $C$ on  ${\mathcal N}$. 
The black point $b$ defines a section of a line bundle on $C$. Its degree is the genus $g$ of $C$, 
and its divisor is the divisor $S$. Homology classes $[\alpha_E]$ 
of zig-zag loops 
on $\Gamma$ are in bijection with the sides of the Newton polygon $N$. 
There is a parametrization
\be \la{param}
\gamma_E: \{\mbox{Zig-zag loops on $\Gamma$ in the homology class $[\alpha_E]$}\}\stackrel{\sim}{\lra} D_E\cap C, 
\ee
uniquely determined by the following condition. The function given by the monodromy around a zig-zag loop $\alpha$ 
coincides with the function given 
by the coordinate of the boundary point $\gamma_E(\alpha)$.
\end{proof}
 
Denote by ${\mathcal S}_\chi$ 
the subvariety of ${\mathcal S}$ given by the condition that 
$C \cap {\mathcal N}_{\infty}$ is a given divisor $C_\infty$. 
Using the parametrization (\ref{param}), the divisor $C_\infty$ provides 
a subvariety ${\mathcal X}_\chi$ in ${\mathcal X}$ defined by the condition that 
the monodromies around zig-zag paths are 
given by the coordinates of the corresponding boundary points. 
Since this condition uses only the boundary divisor of the spectral curve, 
there is a similar subvariety ${\mathcal B}_\chi$ in ${\mathcal B}$. 
The map $\kappa_{\Gamma, v}$ is evidently fibered over 
${\mathcal B}_\chi$, providing a commutative diagram
$$
\begin{array}{ccccc}
{\mathcal X}_\chi&&\stackrel{\kappa_{\Gamma, v}}{\lra} &&{\mathcal S}_\chi\\
&\searrow&&\swarrow &\\
&&{\mathcal B}_\chi &&
\end{array}
$$

\bt \la{20:23}
The image of the map $\kappa_{\Gamma, v}$ is open. The map $\kappa_{\Gamma, v}$ is a finite map over the generic 
point.\footnote{In the sequel to this paper we prove that this map is a birational isomorphism.} 
\et

See below for the proof.
Since both spaces have the same dimension, the first claim implies the second.

\subsection{The real positive points \cite{KO}} \label{realpospoints}
Since ${\mathcal X}$ is a union of cluster coordinate tori, and the 
transition functions between different cluster coordinate systems are subtraction free rational functions, 
there is a well defined set ${\mathcal X}(\R_{>0})$ of its real positive points.
It parametrizes connections 
on a minimal bipartite graph $\Gamma$ associated with $N$ 
whose face weights are positive numbers, and does not depend on the choice of $\Gamma$. 

Simple Harnack curves
 were defined by G. Mikhalkin in \cite{M} as real
curves $P(z_1,z_2)=0$ whose real components satisfy
certain topological constraints. They are a very special case of the
curves which appear in the
Harnack constructions \cite{H} of several
series of real curves of degree $d$ with the maximal number of connected
components (M-curves). Several characterizations  of simple
Harnack curves were given in \cite{M}. They include a definition via $2-1$ covering
of its amoeba. Recall that the amoeba is the image of the curve under the map $$
\C^*\times \C^*\to\R^2, ~~~ (z_1,z_2)\mapsto(\log|z_1|,\log|z_2|).
$$ Another characterization of simple Harnack curves,
due to \cite{MR}, is that they are the curves which maximize the area
of their amoeba
among all curves with given
Newton polygon.

The results of Kenyon-Okounkov \cite[Theorems 1,3]{KO} 
are as follows: 

\begin{itemize}

\item 
The real positive part ${\mathcal B}(\R_{>0}):= \pi({\mathcal X}(\R_{>0}))$ 
parametrizes simple Harnack curves on ${\mathcal N}$ 
which lie in the linear system $|{\mathcal L}|$. Such a genus $g$ simple Harnack curve has 
$g$ ovals $O_i$ matching the integral points inside of the Newton polygon $N$. 
\item 
The space ${\mathcal X}(\R_{>0})$ parametrizes the data  
$\{$A simple Harnack curve $C$ on ${\mathcal N}$ 
from $|{\mathcal L}|$,  a choice of a point on each oval $O_i$, a choice of parametrization 
of the intersection of $C$ with each divisor at infinity$\}$.
\end{itemize}

Therefore the  fibers of the  fibration
$$
\pi_+: {\mathcal X}(\R_{>0}) \lra {\mathcal B}(\R_{>0})
$$
are real $g$-dimensional tori. 
They are the  Liouville tori of the integrable system 
restricted to ${\mathcal X}(\R_{>0})$.

\subsection{Independence of the  Hamiltonians} \la{Concl}

\bc \la{9.23.10.1}
Forgetting the parametrization, 
the space ${\mathcal B}$ projects to an open part of $|{\mathcal L}|$. 
\ec

\begin{proof} Thanks to Section \ref{realpospoints} the image of 
${\mathcal B}$ is an algebraic subvariety of $|{\mathcal L}|$, 
and its real part contains an open subset of the real locus of  $|{\mathcal L}|$.
\end{proof}

\bc \la{9.29.10.1} \la{IHA} The Hamiltonians  are independent on ${\mathcal X}$. 
\ec

\begin{proof}  
Denote by $|{\mathcal L}|_\chi$ the moduli space of 
curves from the linear system $|{\mathcal L}|$ with a given intersection 
${\mathcal C}_\infty$ with ${\mathcal N}_\infty$. Thanks to Corollary \ref{9.23.10.1}, 
${\mathcal B}_\chi$ is a finite cover over on open part of $|{\mathcal L}|_\chi$. 
We claim that the Hamiltonians are local 
coordinate functions on $|{\mathcal L}|_\chi$. 
Therefore they are independent on 
${\mathcal X}$. 

Let us prove the claim. Identify the sections of ${\mathcal L}$ 
with Laurent polynomials in $(z_1, z_2)$ with the Newton polygon $N$. 
Let ${\mathcal P}^0_{N}$ be the subspace of Laurent polynomials in $(z_1, z_2)$ whose 
Newton polygon is the interior  of $N$. 
Adding 
such a polynomial to a section of ${\mathcal L}$, we do not change the intersection with 
${\mathcal N}_\infty$ of its divisor of zeros. So we get an action of 
the vector space ${\mathcal P}^0_{N}$ on $|{\mathcal L}|_\chi$.  
This action is transitive.  
The Hamiltonians 
are given by the coefficients of the monomials in  ${\mathcal P}^0_{N}$. 
So their differentials at a point $C\in |{\mathcal L}|_\chi$ form 
 a basis of $T^*_C|{\mathcal L}|_\chi$. 
\end{proof}

\noindent{\bf Proof of Theorem \ref{20:23}}. 
The restriction of the map $\kappa_{\Gamma, v}$ to 
${\mathcal X}_\chi(\R_{>0})$ 
is an isomorphism. Thus its image is open in ${\mathcal S}_\chi$. 
\hfill$\square$

\old{
\paragraph{Concluding the proof of Theorem \ref{completeintegrabilitytheorem}} 

ii) The independence of the Hamiltonians  is established in Corollary \ref{9.29.10.1}. 
The number of the Hamiltonians was counted in Section \ref{numHam}. 
The part ii) of Theorem \ref{completeintegrabilitytheorem} 
follows. \hfill{$\Box$}
\vskip 2mm
}

\section{Appendix: Calculating the Poisson structure on  ${\L}_\Gamma$} \la{sec2}

Let us give a local description 
of the intersection pairing (\ref{INTP}) 
as a sum of local contributions assigned to the vertices 
of $\Gamma$. 

\subsection{The local pairing.} 
The {\it star of a vertex $v$} in a ribbon graph, Figure \ref{di6}, is 
a vertex $v$ together with the cyclically ordered edges  incident to $v$. 

\begin{figure}[ht]
\centerline{\epsfbox{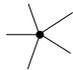}}
\caption{The star of a vertex on an oriented surface.\label{di6}}
\end{figure} 
The star of a vertex $v$ gives rise to an abelian group ${\mathbb A}_v$, given by  
$\Z$-linear combinations of the oriented edges $E_i$ at $v$ of total degree zero: 
$$
{\mathbb A}_v = \{\sum n_i \stackrel{\to}{E}_i ~|~ \sum n_i=0, ~n_i\in \Z\}, 
$$
where $\stackrel{\to}{E}_i$ is the edge $E_i$ oriented out of $v$. 
Clearly ${\mathbb A}_v$ is a torsion free abelian group of rank $m-1$, where $m$ is the valency of $v$. 
It is 
generated by  the ``oriented paths'' $\stackrel{\to}E_i - \stackrel{\to}E_j$, 
where $- \stackrel{\to}E_j$ is the edge $E_j$ oriented towards $v$. 

\begin{figure}[ht]
\centerline{\epsfbox{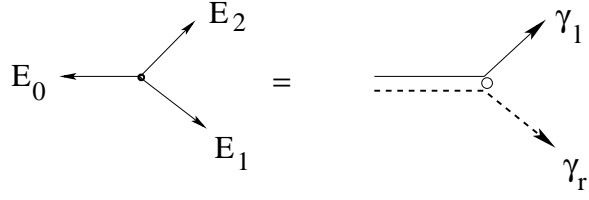}}
\caption{One has  $\delta_v(\gamma_r, \gamma_l)=\frac{1}{2}$.\label{di8}}
\end{figure}

\bl There exists a unique skew symmetric bilinear form
$$
\delta_v: {\mathbb A}_v\wedge {\mathbb A}_v \lra \frac{1}{2}\Z
$$
such that for any triple  $E_0, E_1, E_2$ of edges as in Figure \ref{di8} one has 
\be \la{4.18.10.111}
\delta_v( \gamma_r\wedge \gamma_l) = \frac{1}{2}, \qquad \gamma_r:= \stackrel{\to}{E_1} - 
\stackrel{\to}{E_0}, \quad \gamma_l:= \stackrel{\to}{E_2} - \stackrel{\to}{E_0}. 
\ee
\el

\begin{proof} Let $\{x_1, ..., x_m\}$ be the cyclically ordered 
set of the endpoints of the edges sharing the  vertex $v$. Let 
$\gamma_{i}$ be the oriented  path $x_ivx_{i+1}$. Then the 
elements $\gamma_{1}, ..., \gamma_m$ generate ${\mathbb A}_v$ 
and satisfy the only relation $\sum_i\gamma_{i}=0$. We define the local pairing 
by setting 
$$
\delta_v( \gamma_{i}\wedge \gamma_{i+1}) = \frac{1}{2}.
$$
Then $\sum_i\gamma_{i}$ is in the kernel, so  the local pairing is well defined. 
It is easy to see that it satisfies formula (\ref{4.18.10.111}). 
\end{proof}

\subsection{The global pairing} 
Let $L_1, L_2$ be two oriented paths on $\Gamma$.
Let $v$ be a vertex shared by them. Each path $L_i$ gives rise to an oriented path at the 
star of $v$, providing an element $l_i\in {\mathbb A}_v$.  Let $\varepsilon_v(L_1, L_2)$ 
be the value of the local pairing $\delta_v$ on $l_1 \wedge l_2$. 

\bd \la{4.10.11.3} Let $L_1, L_2$ be two oriented loops on a bipartite surface graph $\Gamma$. 
Then 
$$
\varepsilon(L_1, L_2):= \sum_v{\rm sgn}(v)\varepsilon_v(L_1, L_2),
$$where the sum is over all vertices $v$ shared by $L_1$ and $L_2$, 
and ${\rm sgn}(v)=1$ for the white vertex $v$, and $-1$ for the black vertex.
\ed

\begin{figure}[ht]
\centerline{\epsfbox{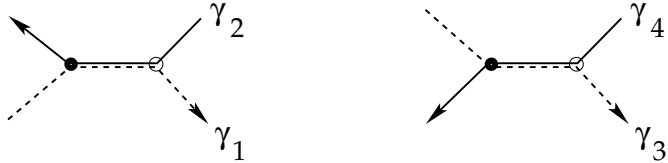}}
\caption{One has $\varepsilon(\gamma_1, \gamma_2)=-1$, $\varepsilon(\gamma_3, \gamma_4)=0$.\label{di9}}
\end{figure}

\begin{figure}[ht]
\centerline{\epsfbox{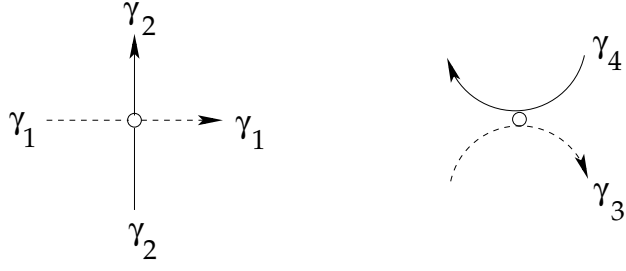}}
\caption{One has $\varepsilon(\gamma_1, \gamma_2)=1$, $\varepsilon(\gamma_3, \gamma_4)=0$.\label{di10}}
\end{figure}

\bp
The bilinear form $\varepsilon(L_1, L_2)$ from Definition \ref{4.10.11.3} coincides with the bilinear form defined 
using the conjugated graph $\widehat \Gamma$ in Section \ref{sec1.1.2}
\ep

\begin{proof} It is sufficient to show that in the examples 
on Figures \ref{di9} and \ref{di10} the sum of the local contributions coincides with the local intersection number 
of the corresponding paths on the conjugated graph. The latter clearly 
gives the results presented on Figures \ref{di9} and \ref{di10}. 
So one needs  to show that the former gives the same result. 
For example, to check that $\varepsilon(\gamma_1, \gamma_2)=1$ for the left picture on Figure \ref{di10}, 
notice that on Figure \ref{di11} we have 
$$
\delta_v(\gamma_1, \gamma_3)= -\frac{1}{2}, \qquad \delta_v(\gamma_1, \gamma_2+\gamma_3)= \frac{1}{2}.
$$
\end{proof}

\begin{figure}[ht]
\centerline{\epsfbox{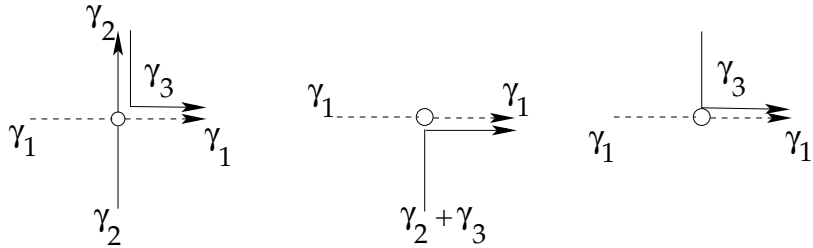}}
\caption{Proving $\varepsilon(\gamma_1, \gamma_2)=1$.\label{di11}}
\end{figure}

{\bf Example.} 
The face orientation provides the counterclockwise orientation of the boundary of the face. 
Given two different faces $F_1, F_2$, and an edge $E$, 
we calculate that  
\begin{equation} 
\delta_{E; F_1, F_2}: =
\left\{ \begin{array}{lll} 
0 & \mbox{if $E \not \in \partial(F_1) \cap \partial(F_2)$,}\\ 
+1 & \mbox{ if the ``black to white'' orientation of $E$ }\\
&\mbox{ is the same as the orientation induced }\\
&\mbox{ by the orientation of the face $F_1$,}\\
-1& \mbox{ otherwise.}
\end{array}\right.
\end{equation} 
Taking now the sum over all edges $E$, we get 
$$
\varepsilon(F_1, F_2) := \varepsilon(\partial(F_1),\partial(F_2)) = \sum_E\delta_{E; F_1, F_2}.
$$

{\bf Example.} Given a face $F$ and an oriented loop $L$ on $\Gamma$, 
let us calculate $\varepsilon(F, L)$. Denote by 
${\rm Or}_{E}$ the  ``black to white'' orientation of the edge $E$. 
The orientation of $L$ induces an orientation ${\rm Or}_{E, L}$ of any edge $E$ of $L$. Set
\begin{equation} \label{f2}
\delta_{E, L} = - {\rm Or}_{E, L}/{\rm Or}_{E} \in \{\pm 1\}. 
\end{equation} 
So if the orientation of  $L$ is opposite to
 the canonical orientation of $E$, then $\delta_{E, L} = 1$. 
Otherwise  $\delta_{E, L}=-1$. 
Then it follows from Figure \ref{di9} that we have
\be \la{3.6.09.1}
\varepsilon(F, L)  = \sum_{E\in L \cap \partial(F)}\delta_{E, L}.
\ee
Here the sum is over all edges $E$ shared by the face $F$ and the loop $L$. 

\vskip 3mm

Changing the orientation of the surface amounts to changing the 
sign of the Poisson bracket. Interchanging black and white we change the 
sign of the Poisson bracket. 

\subsection{Amalgamation and the Poisson structure on  ${\L}_\Gamma$.} 

 We extend the story 
to the case $\Gamma$ is a bipartite graph on an oriented surface $S$ with boundary. 
Pairs $(\Gamma, S)$ are encoded by bipartite ribbon graphs. 

\subsubsection{Bipartite ribbon graphs}
Take a bipartite graph embedded into an oriented surface $S$ with boundary, so that 
some vertices of the graph may lie at the boundary of $S$. 
Let us add to every boundary vertex of the graph a few edges sticking 
out of the surface, called {\it external legs}. 
The graph $\Gamma$ is then a bipartite ribbon graph. Every bipartite ribbon graph is obtained this way. 

\bd Given a bipartite ribbon graph $\Gamma$,  the moduli space ${\L}_\Gamma$ 
 parametrizes line bundles with connection over the graph $\Gamma$, equipped with
 a trivialization 
of the restriction of the bundle to every external edge. 
\ed

Let $\Gamma_1$ and $\Gamma_2$ be  
bipartite ribbon graphs. Let us assume 
that we are given a way to glue some pairs of their external legs into a 
bipartite ribbon graph  $\Gamma$. So each external leg of  each graph is glued 
to no more than one 
external leg of the other graph, and the result is a bipartite ribbon graph. 
Then there is a map of moduli spaces, the {\it amalgamation map}: 
\be \la{4.11.10.1}
\alpha: {\L}_{\Gamma_1}\times {\L}_{\Gamma_2}\lra {\L}_{\Gamma}.
\ee
We use trivializations at the external legs of the bundles on $\Gamma_1$ and 
$\Gamma_2$ to glue 
a new bundle on $\Gamma$. It inherits a connection. The following is obvious.

\bl
The gluing map (\ref{4.11.10.1}) is a principal fibration over ${\L}_\Gamma$ with 
the fiber $(\C^*)^k$, 
where $k$ is the number of pairs of glued boundary legs. 
\el
\begin{figure}[ht]
\centerline{\epsfbox{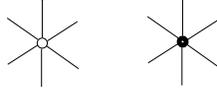}}
\caption{White and black hedgehogs with six external legs.\label{di12}}
\end{figure}

A bipartite ribbon graph with single internal vertex is 
called a {\it hedgehog}, see Figure \ref{di12}. 
It is a white / black hedgehog according to the color of its internal vertex.
Let $H_v$ be a hedgehog with  the internal vertex $v$. Then one has
\be \la{4.11.10.2}
{\L}_{H_v} = {\rm Hom}({\mathbb A}_v, \C^*)\stackrel{}{=} {\rm Ker}\Bigl(\mu: 
(\C^*)^{m} \lra \C^*\Bigr).
\ee
Here $m$ is the number of legs  and
$\mu$ is the multiplication map. The first isomorphism assigns 
to an oriented  path connecting two external legs 
the monodromy of the connection along this path. Since the line bundle 
is trivialized at the endpoints, 
it is just a non zero complex number. The  second 
is obvious.

A ribbon graph $\Gamma$ is
 glued from the hedgehogs provided by the stars of its internal 
vertices. 
Amalgamating the hedgehog moduli spaces we get the moduli space ${\L}_{\Gamma}$. 

\begin{figure}[ht]
\centerline{\epsfbox{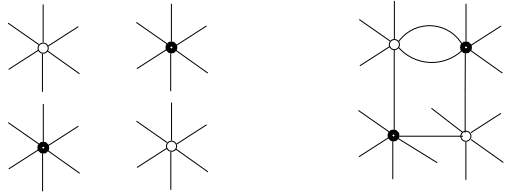}}
\caption{Gluing bipartite ribbon graphs from hedgehogs.\label{di13}}
\end{figure}

Thanks to identification (\ref{4.11.10.2}), the skew symmetric 
local pairing on the abelian group 
${\mathbb A}_v$ provided by the color of the internal vertex $v$ of 
the hedgehog gives rise to a Poisson 
structure on the moduli space ${\L}_{H_v}$. 
The following proposition follows easily from the definitions.
\bl
There is a unique Poisson structure on the space  ${\L}_\Gamma$ such that 
the Poisson structure on the space ${\L}_{H_v}$ assigned to a hedgehog $H_v$ is 
the standard one, and the amalgamation maps (\ref{4.11.10.1}) are Poisson. 
\el

When $\Gamma$ is a bipartite graph on 
a surface without boundary we recover its Poisson structure.

\end{document}